\documentclass[11pt]{amsart}

\usepackage[a4paper, margin=2.7cm]{geometry}
\usepackage[style=alphabetic, maxnames=5]{biblatex}
\renewbibmacro{in:}{}
\AtBeginBibliography{\small}

\usepackage{amsmath}
\usepackage{hyperref}
\usepackage{mathtools}
\usepackage[makeroom]{cancel}
\usepackage{amssymb}
\usepackage{amsthm}
\usepackage{tensor}
\usepackage{enumerate}
\usepackage{cleveref}
\usepackage{bbm}
\usepackage{float}
\usepackage{varwidth}
\usepackage{pictexwd, dcpic}
\usepackage{xcolor}
\hypersetup{
    colorlinks,
    linkcolor={blue!60!black},
    citecolor={red!60!black},
    urlcolor={blue!80!black}
}
\usepackage{graphicx}
\usepackage{fancyhdr}
\usepackage{mathrsfs}
\usepackage{xfrac}
\usepackage[utf8]{inputenc}
\usepackage{tikz}
\usepackage{scalerel,stackengine}
\usepackage{booktabs}
\usepackage{slashed}

\allowdisplaybreaks

\theoremstyle{plain}
\newtheorem{theorem}{Theorem}
\newtheorem*{theorem*}{Theorem}
\newtheorem{lemma}[theorem]{Lemma}
\newtheorem{proposition}[theorem]{Proposition}
\newtheorem{corollary}[theorem]{Corollary}
\newtheorem*{corollary*}{Corollary}

\theoremstyle{definition}
\newtheorem{definition}[theorem]{Definition}

\theoremstyle{remark}
\newtheorem{remark}[theorem]{Remark}

\numberwithin{theorem}{section}
\numberwithin{equation}{section}

\renewcommand{\d}{\mathrm{d}}
\renewcommand{\leq}{\leqslant}
\renewcommand{\geq}{\geqslant}
\renewcommand{\epsilon}{\varepsilon}

\renewcommand{\Im}{\operatorname{Im}}

\newcommand{\D}{\mathrm{D}}
\newcommand{\la}{\lesssim}

\newcommand{\scri}{\mathscr{I}}
\newcommand{\e}{\mathrm{e}}
\newcommand{\slashgrad}{\slashed{\nabla}}
\newcommand{\sech}{\operatorname{sech}}

\DeclareFontFamily{U}{MnSymbolC}{}
\DeclareSymbolFont{MnSyC}{U}{MnSymbolC}{m}{n}
\DeclareFontShape{U}{MnSymbolC}{m}{n}{
    <-6>  MnSymbolC5
   <6-7>  MnSymbolC6
   <7-8>  MnSymbolC7
   <8-9>  MnSymbolC8
   <9-10> MnSymbolC9
  <10-12> MnSymbolC10
  <12->   MnSymbolC12}{}
\DeclareMathSymbol{\intprod}{\mathbin}{MnSyC}{'270}
\newcommand*{\defeq}{\mathrel{\vcenter{\baselineskip0.5ex \lineskiplimit0pt
                     \hbox{\scriptsize.}\hbox{\scriptsize.}}}%
                     =}

\newcommand*{\eqdef}{=\mathrel{\vcenter{\baselineskip0.5ex \lineskiplimit0pt
                     \hbox{\scriptsize.}\hbox{\scriptsize.}}}%
                     }
\newcommand\equalhat{\mathrel{\stackon[1.5pt]{=}{\stretchto{%
    \scalerel*[\widthof{=}]{\wedge}{\rule{1ex}{3ex}}}{0.5ex}}}}
\DeclareFontFamily{U}{BOONDOX-calo}{\skewchar\font=45 }
\DeclareFontShape{U}{BOONDOX-calo}{m}{n}{
  <-> s*[1.05] BOONDOX-r-calo}{}
\DeclareFontShape{U}{BOONDOX-calo}{b}{n}{
  <-> s*[1.05] BOONDOX-b-calo}{}
\DeclareMathAlphabet{\mathcalboondox}{U}{BOONDOX-calo}{m}{n}
\SetMathAlphabet{\mathcalboondox}{bold}{U}{BOONDOX-calo}{b}{n}
\DeclareMathAlphabet{\mathbcalboondox}{U}{BOONDOX-calo}{b}{n}
\DeclareMathOperator{\dvol}{dv}

\date{\today}
\title{Conformal scattering of the Maxwell-scalar field system on de Sitter space}

\author{Grigalius Taujanskas}

\address[G. Taujanskas]{Department of Pure Mathematics and Mathematical Statistics, University of Cambridge, CB3 0WB, UK}
\email{taujanskas@dpmms.cam.ac.uk}

\begin{document}

\maketitle

\begin{abstract} We prove small data energy estimates of all orders of differentiability between past null infinity and future null infinity of de Sitter space for the conformally invariant Maxwell-scalar field system. This allows us to construct bounded and invertible, but nonlinear, scattering operators taking past asymptotic data to future asymptotic data. We also deduce exponential decay rates for solutions with data having at least two derivatives, and for more regular solutions discover an asymptotic decoupling of the scalar field from the charge. The construction involves a carefully chosen complete gauge fixing condition which allows us to control all components of the Maxwell potential, and a nonlinear Gr\"onwall inequality for higher order estimates.
\end{abstract}

\tableofcontents
\addtocontents{toc}{\protect\setcounter{tocdepth}{1}}

Studies of scattering go back to the beginnings of physics. Perhaps the most famous modern mathematical treatment was developed in the 1960s by Lax and Phillips \cite{LaxPhillips1964,LaxPhillips1967}, who used spectral techniques to study the scattering of a wave by an obstacle in flat space. In general relativity it is of interest to study metric scattering, that is the effects of curved space on the asymptotic behaviour of fields. Around the same time as Lax and Phillips were developing their framework, Roger Penrose discovered a way to compactify certain spacetimes by conformally rescaling the metric and attaching a boundary, $\scri$ \cite{Penrose1963,Penrose1965}. He called the class of spacetimes admitting such a compactification \emph{asymptotically simple} and the boundary so attached \emph{null infinity}, for this was where all null geodesics ended up `at infinity'. This led to a brand new way of viewing the asymptotics of massless fields in general relativity: one works in Penrose's conformally compactified spacetime and studies the regularity of fields on $\scri$, and then translates the regularity in the conformally rescaled spacetime to fall-off conditions in the physical spacetime.

It was not until the work of Friedlander \cite{Friedlander1980,Friedlander2001} in 1980, however, that it was understood that the approaches of Lax and Phillips on the one hand and Penrose on the other could be combined into a robust \emph{geometric} formulation of scattering theory. Friedlander showed that, although one cannot perform the same analytically explicit constructions in curved space, one can make sense of the Lax--Phillips asymptotic profiles of fields by identifying them with suitably rescaled limits of fields going to infinity along null directions. These became known as Friedlander's radiation fields. The ideas of such \emph{conformal scattering} were taken up by Baez, Segal and Zhou \cite{Baez1989,Baez1990,BaezSegalZhou1990,BaezZhou1989} to study a nonlinear wave equation and to some extent Yang--Mills equations on flat space, and later by Mason and Nicolas \cite{MasonNicolas2004,MasonNicolas2007} to study linear equations on a large class of asymptotically simple spacetimes constructed by Corvino, Schoen, Chru\'sciel, Delay, Klainerman, Nicol\`o, Friedrich and others \cite{ChruscielDelay2002,ChruscielDelay2003,Corvino2000,CorvinoSchoen2006,KlainermanNicolo1999,KlainermanNicolo2003}.
This spurred a programme of constructing conformal scattering theories for various fields on a variety of backgrounds and since then a number of works have appeared, many focussing on conformal scattering on black hole spacetimes\footnote{See also \cite{KehleShlapentokhRothman2018,VandeMoortel2017} for some results in interiors of black holes.}\cite{HafnerNicolas2004,Joudioux2010,Mokdad2017,Nicolas2013,Nicolas2015}. It should be mentioned that there have been plenty of works studying relativistic scattering theory without employing the conformal method, notably by Dimock and Kay in the 1980s \cite{Dimock1985,DimockKay1986} and later by Bachelot \cite{Bachelot1991,Bachelot1994} and collaborators Nicolas, H\"afner, Daud\'e, and Melnyk, among many others, a programme which eventually led to rigorous proofs of the Hawking effect \cite{Bachelot1999,Melnyk2004}. 

The above programmes were concerned mainly with asymptotically flat spacetimes. However, astronomical observations have by now shown that the cosmological constant $\Lambda$ in our universe, though tiny, is positive \cite{Perlmutter2000,Perlmutter1999,Riess1998,Schmidt1998}. It is thus of interest to study scattering, especially of nonlinear fields, on de Sitter space. De Sitter space is the Lorentzian analogue of the sphere in Euclidean geometry and one of the three maximally symmetric solutions to the vacuum Einstein equations as classified by the sign of the cosmological constant, with flat Euclidean space corresponding to Minkowski space ($\Lambda = 0$) and hyperbolic space corresponding to anti-de Sitter space ($\Lambda < 0$). As such, de Sitter space differs from Minkowski space in several crucial aspects. Firstly, it is not asymptotically flat. Nonetheless, it is asymptotically simple in the sense of Penrose \cite{Penrose1965} and so admits a conformal compactification. Secondly, the positive cosmological constant, no matter how small, renders null infinity \emph{spacelike} in de Sitter space, which has implications for conformal scattering. In the asymptotically flat case the constructions of Mason and Nicolas required the resolution of a global linear Goursat problem, which had been shown by H\"ormander \cite{Hormander1990} to be solvable in some generality. In de Sitter space, however, a spacelike $\scri$ means that the construction of a scattering theory instead requires the resolution of a regular Cauchy problem. Thirdly, while obtaining flat space scattering and peeling results through conformal techniques is fine for linear fields, nonlinear fields generically possess so-called charges at spacelike infinity \cite{Petrescu1996,AbbottDeser1982,ChruscielKondracki1987}. This is a major obstruction to constructing conformal scattering theories for nonlinear fields in asymptotically flat spacetimes and is related to infrared divergences in quantum field theory \cite{KulishFaddeev1970,NashStuller1978}. The problem is entirely absent in de Sitter space as it is spatially compact.

From an analytic point of view, it has been known since the work of Friedrich \cite{Friedrich1986,Friedrich1991} that de Sitter space is a stable solution of Einstein's equations with a positive cosmological constant. Moreover, a recent and much celebrated result of Hintz and Vasy has shown that Kerr-de Sitter black holes are stable \cite{HintzVasy2016}. One therefore expects scattering results on de Sitter space to fit into a larger host of stories on asymptotically de Sitter spacetimes. Other results in this vein have been obtained by, for example, Vasy, Melrose and S\'a Barreto, \cite{Vasy2007,MelroseSaBarretoVasy2014}. From a more physical perspective, de Sitter space has the peculiar feature that no single observer can ever observe the entire spacetime, in contrast to the Minkowski case where an observer's past lightcone eventually contains the whole history of the universe. This is related to the existence of cosmological horizons, null hypersurfaces criss-crossing the Penrose diagram of de Sitter space. Their existence has implications for the definition of a classical scattering matrix: the construction of one requires a timelike Killing or conformally Killing vector field, and here one has a choice in de Sitter space. One might wish to use the Killing field provided by the standard static coordinates, i.e. the coordinates an observer at the south pole in de Sitter space might use for themselves, but this is problematic as it fails to be timelike and future pointing beyond the cosmological horizons. Another approach is to conformally compactify de Sitter space and embed it in the Einstein cylinder, where one has a natural globally timelike Killing field which becomes conformally Killing in physical de Sitter space. This can then be used to define an observer-oblivious classical scattering matrix in de Sitter space. We adopt the latter approach here. The importance of the construction of such scattering matrices for quantum gravity in de Sitter is explained well in \cite{LesHouchesLectureNotes} and the references therein.

This paper is organized as follows. In \Cref{sec:conventions} we state the conventions and notation used in the paper, and in \Cref{sec:confinvariantsystem} we introduce the conformally invariant Maxwell-scalar field system that we subsequently study. In \Cref{sec:deSitter} we describe de Sitter space $\mathrm{dS}_4$, a number of standard coordinate systems on $\mathrm{dS}_4$, its conformal compactification, and our choice of energy-momentum tensor for the Maxwell-scalar field system on the conformally rescaled spacetime. In \Cref{sec:maintheorems} we state the main results in detail. \Cref{sec:fieldequationsgaugefixing,sec:wellposedness} contain a detailed derivation of the required gauge fixing conditions, the formulation of the Cauchy problem for our system, and an existence theorem. \Cref{sec:energies,sec:energyestimates,sec:higherorderestimates} contain the inductive energy estimates on which our results rest. \Cref{sec:proofofH2estimates,sec:proofofhigherorderestimates,sec:proofofdecayrates} finish off the proofs of the main results.

\section{Results}

We prove small\footnote{See, however, \cite{Taujanskas2019} for an extension to large data.} data energy estimates of all orders of differentiability $m$ between $\scri^-$ and $\scri^+$ of de Sitter space for the conformally invariant Maxwell-scalar field system and show the existence of small data scattering operators $\mathscr{S}_m$ for all $m \geq 2$. These estimates rely crucially on the subcritical nature of the nonlinearity of the Maxwell-scalar field system in four dimensions. We find that, using a careful choice of gauge, it is possible to control all components of the Maxwell potential and the scalar field, and close the estimates using a nonlinear Gr\"onwall inequality. We may initially state the main theorem as follows. The full statements of the main theorems can be found in \Cref{sec:maintheorems}.

Consider the Penrose diagram of de Sitter space and an initial surface $\Sigma \simeq \mathbb{S}^3$, as depicted in \Cref{fig:square}.

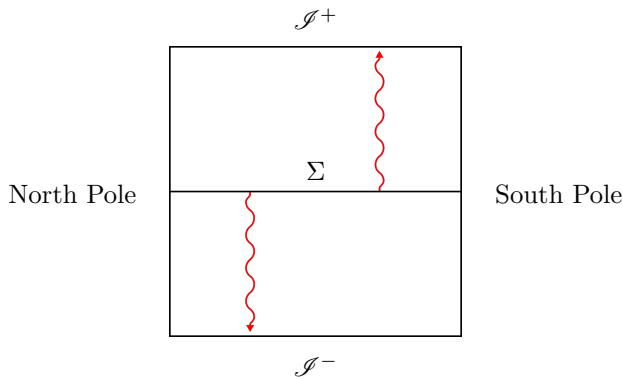
\begin{figure}[h]
\centering
	\begin{tikzpicture}
	\centering
	\node[inner sep=0pt] (desitterpenrose) at (3.4,0)
    	{\includegraphics[width=.24\textwidth]{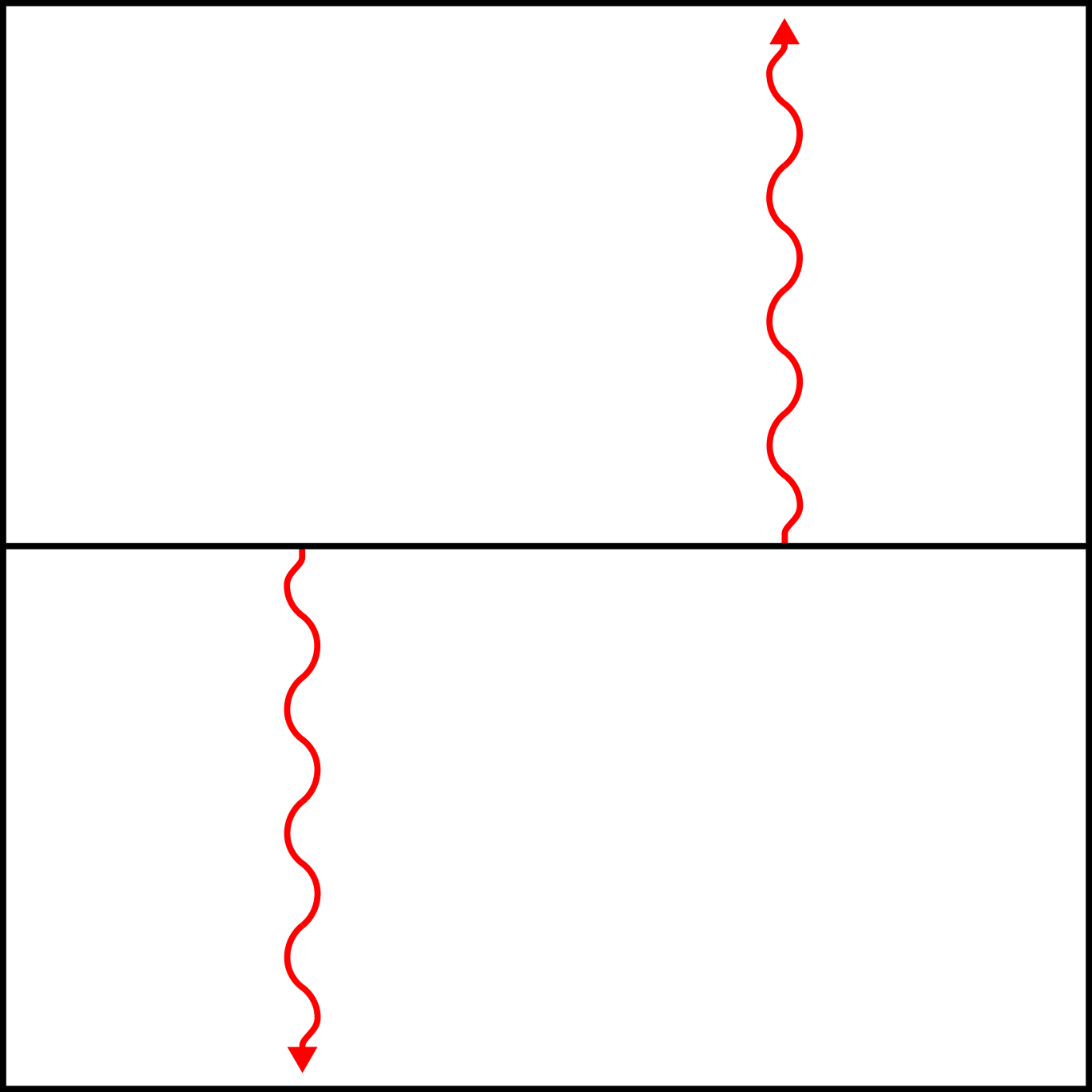}};

	\node[label={[shift={(0.2,-0.4)}]North Pole}] {};
	\node[label={[shift={(6.6,-0.4)}]South Pole}] {};
	\node[label={[shift={(3.4,1.9)}]$\scri^+$}] {};
	\node[label={[shift={(3.4,-2.7)}]$\scri^-$}] {};
	\node[label={[shift={(3.4,-0.1)}]$\Sigma$}] {};
	
	\end{tikzpicture}
	\caption{The Penrose diagram for $\mathrm{dS}_4$. The wavy red lines represent the forward and backward wave operators $\mathscr{W}^\pm_m$.} \label{fig:square}	
\end{figure}

\begin{theorem} For any $m \in \mathbb{N}$, the $H^m \oplus H^{m-1}$ norm on null infinity of the rescaled solution of the Maxwell-scalar field system is equivalent to the $H^m \oplus H^{m-1}$ norm of the initial data, provided the initial data is sufficiently small.	
\end{theorem}

Using these estimates, we construct bounded and invertible, but nonlinear, wave and scattering operators.

\begin{theorem} For any $m \geq 2$ there exist bounded and invertible, but nonlinear, forward and backward wave operators $\mathfrak{T}_m^\pm$ mapping small $H^m \oplus H^{m-1}$ Maxwell-scalar field data on $\Sigma$ to small $H^m \oplus H^{m-1}$ Maxwell-scalar field data on $\scri^\pm$, and a bounded invertible scattering operator 
\[ \mathscr{S}_m = \mathfrak{T}^+_m \circ (\mathfrak{T}^-_m )^{-1} \]
mapping small $H^m \oplus H^{m-1}$ Maxwell-scalar field data on $\scri^-$ to small $H^m \oplus H^{m-1}$ Maxwell-scalar field data on $\scri^+$.
\end{theorem}

As a corollary, our estimates also imply exponential decay rates for the Maxwell-scalar field system on de Sitter space with small $H^2 \oplus H^1$ initial data. The decay rates are a partial extension of the results of Melrose, S\'a Barreto and Vasy \cite{MelroseSaBarretoVasy2014}.

\begin{corollary} The scalar field and the components of the Maxwell potential decay exponentially in proper time along timelike geodesics approaching $\scri$.
\end{corollary}

In addition to their interpretation in terms of peeling and conformal scattering, our results may also be seen as a fixed background stability result in the spirit of Friedrich, Svedberg and Ringstr\"om  \cite{Friedrich1986,Friedrich1991,Svedberg2011,Ringstrom2008}. It is worth mentioning that the estimates we prove here are explicit, allowing us to define the sets of scattering data and read off precise decay rates.

Since the nonlinearities are of the same order, in principle there is no obstruction to extending our estimates to the Yang--Mills--Higgs system on de Sitter space. As a result, the same scattering and decay results should apply there.

\section{Conventions} \label{sec:conventions}

We use the spacetime signature $(+,-,-,-)$. Our main estimates will be performed on the \emph{Einstein cylinder} $\mathfrak{E} = \mathbb{R} \times \mathbb{S}^3$ with metric $\mathfrak{e} = g_{\mathbb{R}} \oplus (- \mathfrak{s}_3)$, where $\mathfrak{s}_3 = g_{\mathbb{S}^3}$ is the standard positive-definite metric on $\mathbb{S}^3$. We will use Penrose's abstract index notation and use the Roman indices $a,b,\dots$ to refer to tensors on $\mathfrak{E}$ and contractions with respect to the full spacetime metric $\mathfrak{e}$ (or sometimes a general spacetime $\mathcalboondox{M}$ with metric $g$), and use the Greek indices $\mu,\nu,\dots$ to refer to tensors on $\mathbb{S}^3$ and contractions with respect to the metric $\mathfrak{s}_3$. At a certain point we will also use the indices $i$, $j$ and $k$ to refer to a basis of vector fields on $\mathbb{S}^3$, but this will be made explicit at the time. We will use $\nabla$ to denote the Levi--Civita connection of the full spacetime metric $\mathfrak{e}$ (or a general metric $g$), and $\slashgrad$ to denote the Levi--Civita connection of $\mathfrak{s}_3$. Thus, as $\mathfrak{e} = g_{\mathbb{R}} \oplus (- \mathfrak{s}_3 ) = 1 \oplus ( - \mathfrak{s}_3 )$, we shall have $\nabla = \nabla^{\mathbb{R}} \oplus \nabla^{\mathfrak{s}_3} = \partial \oplus  \slashgrad$. We will use $\dvol$ to denote the volume form of the full spacetime metric ($\mathfrak{e}$ or $g$), and $\dvol_{\mathfrak{s}_3}$ to denote the volume form of $\mathfrak{s}_3$. In the case of $(\mathfrak{E}, \mathfrak{e})$ we will thus have $\dvol = \d \tau \wedge \dvol_{\mathfrak{s}_3}$, $\tau$ being the coordinate on $\mathbb{R}$. For a $1$-form $A$ on $\mathfrak{E}$ we will use $\mathbf{A}$ to denote the projection of $A$ onto $\mathbb{S}^3$, $A_0$ to denote the component of $A$ along $\partial_\tau$, and dot (as in $\dot{\mathbf{A}}$) to denote differentiation with respect to $\tau$. The Lebesgue and Sobolev norms $L^p$ and $H^m$ of a scalar or vector will refer to $L^p(\mathbb{S}^3)$ and $H^m(\mathbb{S}^3)$, unless specifically stated otherwise. Occasionally we shall use the symbol $\equalhat$ to denote equality \emph{on} null infinity $\scri$ (see \Cref{sec:deSitter}).

We will also adopt Penrose's sign convention for the curvature tensors, meaning that the Riemann curvature tensor $R^c_{\phantom{c}dab}$ will satisfy
\[ [\nabla_a, \nabla_b ] X^c = - R^c_{\phantom{c}dab} X^d. \] The Ricci tensor and the scalar curvature will then be defined as usual,
\[ R_{ab} \defeq R^c_{\phantom{c}acb}, \qquad R \defeq R_a^{\phantom{a}a}, \]
so that in these conventions the scalar curvature of, for example, a $3$-sphere with the \emph{positive-definite} metric $\mathfrak{s}_3$ will be \emph{negative}, $-6$ to be exact. However, since our metrics will be of signature $(+,-,-,-)$, that will mean that a spacelike $3$-sphere in our construction will have positive scalar curvature equal to $6$.

\section{The Conformally Invariant Maxwell-Scalar Field System} \label{sec:confinvariantsystem}

Let $(\mathcalboondox{M}, g)$ be a $4$-dimensional Lorentzian manifold and consider the  Lagrangian density
\begin{equation} \label{conformallyinvariantlagrangian} \mathcal{L} = - \frac{1}{4} F_{ab} F^{ab} + \frac{1}{2} \D_a \phi \overline{\D^a \phi} - \frac{1}{12} R | \phi |^2,
\end{equation}
where $F_{ab} = 2 \nabla_{[a} A_{b]} $ is a real $2$-form called the Maxwell field, $A_a$ is a real $1$-form called the Maxwell potential, $\phi$ is a complex scalar field on $\mathcalboondox{M}$, $R$ is the scalar curvature of $g_{ab}$, and $\D_a \phi = \nabla_a \phi + i A_a \phi$, where $\nabla_a$ is the Levi--Civita connection of $g_{ab}$. The differential operator $\D_a$ is called the gauge covariant derivative. The Euler--Lagrange equations associated to \eqref{conformallyinvariantlagrangian} are
\begin{equation} \label{phiAequation1} \nabla^b F_{ab} = \Im \left( \bar{\phi} \D_a \phi \right) \quad \text{and} \quad \D^a \D_a \phi + \frac{1}{6} R \phi = 0. \end{equation}
The Maxwell-scalar field system \eqref{conformallyinvariantlagrangian} is the simplest classical field theory exhibiting a non-trivial gauge dependence. Indeed, the $1$-form $A_a$ is not uniquely determined by the $2$-form $F_{ab}$, and any transformation of the form
\[ A_a \longmapsto A_a + \nabla_a \chi \]
leaves $F_{ab}$ unchanged. This transforms 
\[ \D_a \phi = \nabla_a \phi + i A_a \phi \longmapsto \nabla_a \phi + i (A_a + \nabla_a \chi) \phi = \mathrm{e}^{-i \chi} \D_a ( \mathrm{e}^{i \chi} \phi), \]
so that if one makes the corresponding transformation
\[ \phi \longmapsto \mathrm{e}^{-i \chi} \phi, \]
the Lagrangian \eqref{conformallyinvariantlagrangian}, and thus also the field equations \eqref{phiAequation1}, remain unchanged.

\begin{remark} The gauge covariant derivative $\D_a$ acting on $\phi$ is a connection on a principal bundle $P$ over $\mathcalboondox{M}$ with fibre $\mathrm{U}(1)$. This connection is represented by the real $1$-form $A_a$ on $\mathcalboondox{M}$ in any trivialisation of $P$, where the factor of $i$ in $\D_a$ comes from $\mathfrak{u}(1) = i \mathbb{R}$. The scalar field $\phi$ is a section of a complex line bundle over $\mathcalboondox{M}$ associated to $P$ by the representation $\mathrm{e}^{i \chi}$ of $\mathrm{U}(1)$.
\end{remark}

Consider a conformal rescaling of $(\mathcalboondox{M}, g)$,
\begin{equation} \label{conformalrescaling} \hat{g}_{ab} = \Omega^2 g_{ab}. \end{equation}
It turns out that in many cases it is possible to fully or partially compactify $\mathcalboondox{M}$ by choosing the conformal factor $\Omega$ so that it compensates for the divergence of distances with respect to the physical metric $g$ and attach the boundary $\scri \defeq \{ \Omega = 0 \}$ to $\mathcalboondox{M}$; this is Roger Penrose's notion of asymptotically simple spacetimes first described around 1963 in \cite{Penrose1963} and \cite{Penrose1965}.  For our purposes it will be sufficient to assume that the spacetime $\mathcalboondox{M}$ is regular enough so that it may be compactified in this way to make a smooth compact manifold with boundary, $\widehat{\mathcalboondox{M}} \defeq \mathcalboondox{M} \cup \scri$, although weaker, partial compactifications leaving singularities at a finite number of points in the boundary are widely used to study, for example, black hole spacetimes \cite{Joudioux2010,MasonNicolas2004,MasonNicolas2007,Mokdad2017,Nicolas2013,Nicolas2015}. We equip $\widehat{\mathcalboondox{M}}$ with the rescaled (also called unphysical) metric $\hat{g}_{ab}$ and call the spacetime $(\widehat{\mathcalboondox{M}},\hat{g})$ the rescaled spacetime.

It is possible to transport the fields $(A_a,\phi)$ into the rescaled spacetime $\widehat{\mathcalboondox{M}}$ by weighting them appropriately by the conformal factor $\Omega$ so that the field equations \eqref{phiAequation1} are preserved in $\widehat{\mathcalboondox{M}}$. The correct choice of conformal weights for $(A_a, \phi)$ are $(0,-1)$,
\[ \hat{A}_a \defeq A_a, \qquad \quad \hat{\phi} \defeq \Omega^{-1} \phi, \]
and we show below that this implies the conformal invariance of the Maxwell-scalar field system \eqref{phiAequation1}. Under the rescaling \eqref{conformalrescaling} the Christoffel symbols $\Gamma^a_{bc}$ of $g_{ab}$ transform as
\[ \hat{\Gamma}^a_{bc} = \Gamma^a_{bc} + \Upsilon_c \delta^a_b + \Upsilon_b \delta^a_c - \Upsilon_d g^{ad} g_{bc}, \]
where $\Upsilon_a \defeq \Omega^{-1} \partial_a \Omega = \partial_a \log \Omega$, and using this one calculates that
\[ -\frac{1}{4} F_{ab} F^{ab} = - \frac{1}{4} \Omega^4 \hat{F}_{ab} \hat{F}^{ab} \]
and
\[ \frac{1}{2} \D_a \phi \overline{\D^a \phi} = \frac{1}{2} \Omega^{4} \hat{\D}_a \hat{\phi}\overline{\hat{\D}^a \hat{\phi}} + \frac{1}{2} \Omega^4 \left( 2 \Upsilon_a \operatorname{Re} (\hat{\phi} \overline{\hat{\D}^a \hat{\phi}}) + \hat{g}^{ab} \Upsilon_a \Upsilon_b |\hat{\phi}|^2 \right). \]
Moreover, because in $4$ dimensions the scalar curvature $R$ transforms as (see \cite{spinorsandspacetime2}, eq. (6.8.25))
\[ \frac{1}{12} R = \Omega^2 \left( \frac{1}{12} \hat{R} - \frac{1}{2} \hat{\nabla}^a \Upsilon_a + \frac{1}{2} \hat{g}^{ab} \Upsilon_a \Upsilon_b \right), \]
one has
\[ - \frac{1}{12} R |\phi |^2 = - \frac{1}{12} \Omega^4 \hat{R} | \hat{\phi} |^2 + \frac{1}{2} \Omega^4 \left( \hat{\nabla}^a \Upsilon_a - \hat{g}^{ab} \Upsilon_a \Upsilon_b \right) |\hat{\phi}|^2. \]
Adding these together one sees that the Lagrangian transforms as
\begin{align*} \mathcal{L} &= \Omega^4 \hat{\mathcal{L}} + \frac{1}{2} \Omega^4 \left( 2 \Upsilon_a \operatorname{Re}(\hat{\phi} \overline{\hat{\D}^a \hat{\phi}}) + (\hat{\nabla}^a \Upsilon_a ) |\hat{\phi}|^2 \right) \\ 
&= \Omega^4 \hat{\mathcal{L}} + \frac{1}{2} \Omega^4 \left( \Upsilon_a \hat{\nabla}^a ( | \hat{\phi} |^2) + (\hat{\nabla}^a \Upsilon_a )|\hat{\phi}|^2 \right) \\
&= \Omega^4 \hat{\mathcal{L}} + \frac{1}{2} \Omega^4 \hat{\nabla}^a ( |\hat{\phi}|^2 \Upsilon_a ).
\end{align*}
Now the volume form $\widehat{\dvol}$ of $\widehat{\mathcalboondox{M}}$ is related to the volume form $\dvol$ of $\mathcalboondox{M}$ by $\dvol = \Omega^{-4} \widehat{\dvol}$, so the action 
\[ S = \int_{\mathcalboondox{M}}  \mathcal{L} \dvol \]
transforms as
\begin{equation} \label{conformalchangeofaction} S = \hat{S} + \frac{1}{2} \int_{\widehat{\mathcalboondox{M}}} \hat{\nabla}^a ( | \hat{\phi}|^2 \Upsilon_a ) \, \widehat{\dvol} = \hat{S} + \frac{1}{2} \int_{\scri} |\hat{\phi}|^2 \Upsilon_a \hat{g}^{ab} \partial_b \intprod \widehat{\dvol}.
\end{equation}
In other words, $S$ is conformally invariant up to a boundary term. Since the Euler-Lagrange equations arise from a \emph{local} variation of the action, this implies the conformal invariance of the field equations \eqref{phiAequation1}.

\section{De Sitter Space} \label{sec:deSitter}

\subsection{Global Coordinates and Conformal Compactification}

The $(3+1)$-dimensional de Sitter space $\mathrm{dS}_4$ is defined to be the hyperboloid
\[ |x|^2 - x_0^2 = \frac{1}{H^2} \]
in $(4+1)$-dimensional Minkowski space
\[ m = \d x_0^2 - \d |x|^2 - |x|^2 \mathfrak{s}_3 , \]
where $|x| = \sqrt{x_1^2 + x_2^2 + x_3^2 + x_4^2}$ and $\mathfrak{s}_3$ is the standard metric on the $3$-sphere $\{ |x|=1 \}$. If we set
\[ x_0 = \frac{1}{H} \sinh\left( H \eta \right), \qquad \quad |x| = \frac{1}{H} \cosh \left( H \eta \right), \]
so that $\eta$ is a coordinate on $\mathrm{dS}_4$, the metric $m$ descends to the metric $\d s^2$ on $\mathrm{dS}_4$,
\begin{equation} \label{closedslicing} \d s^2 = \d \eta^2 - \frac{1}{H^2} \cosh^2 \left( H \eta \right) \mathfrak{s}_3. \end{equation}
This provides a global coordinate system on $\mathrm{dS}_4$ and is known as the closed slicing of de Sitter space. Note that the $\mathbb{R} \times \mathbb{S}^3$ topology is manifest in these coordinates. The metric \eqref{closedslicing} can be visualized as a compact spacelike slice expanding in time $\eta$, as depicted in \cref{fig:expandingdesitter}.
\begin{figure}[h]
\centering
	\begin{tikzpicture}
	\centering
	\node[inner sep=0pt] (expandingdesitter) at (2,0)
    	{\includegraphics[width=.18\textwidth]{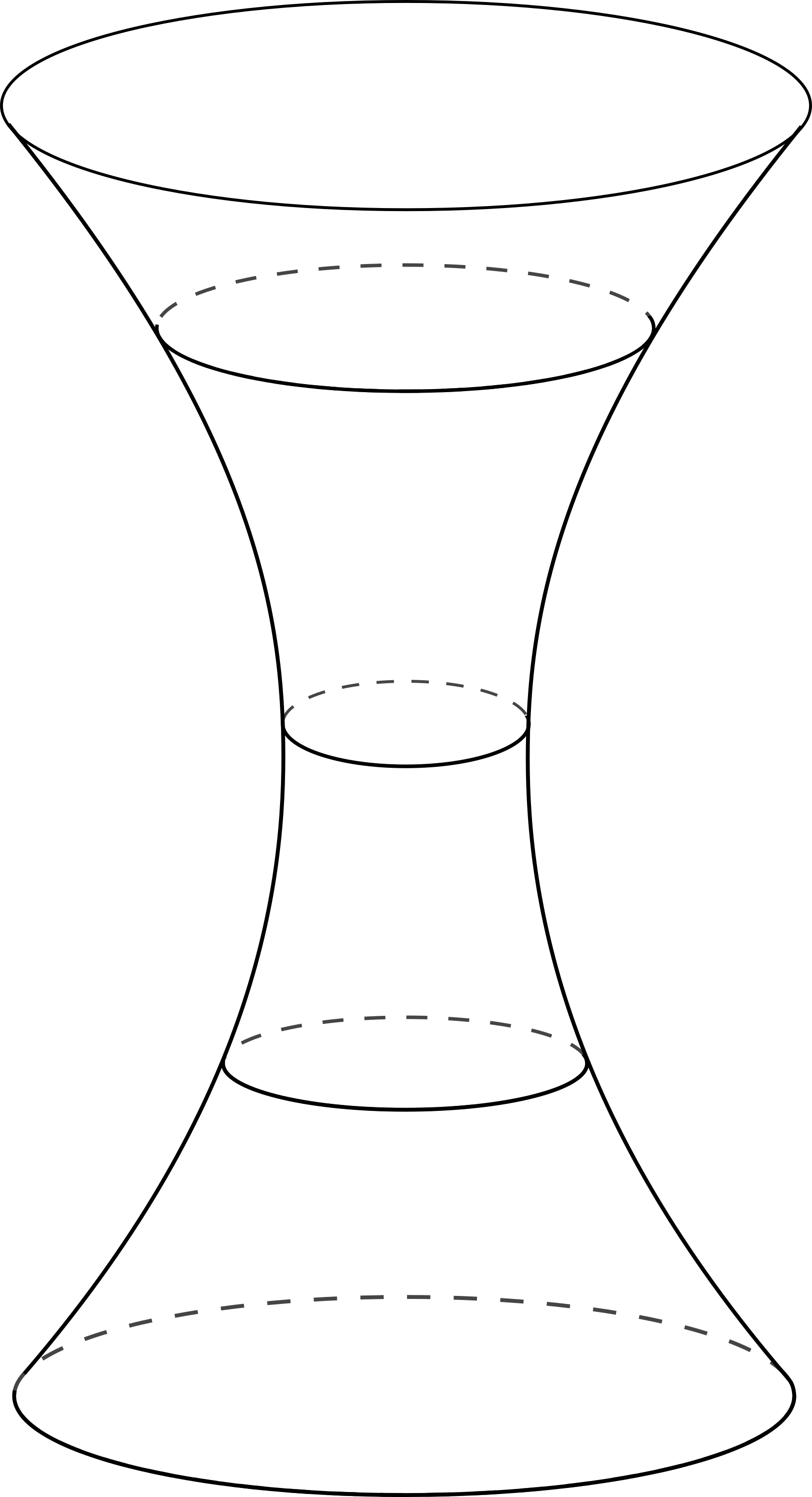}};

	\node[label={[shift={(4.2,-0.4)}]$\eta$}] {};
	\draw[->] (3.9,-0.3) .. controls (3.9,-1) .. (3.9,0.7);

	\end{tikzpicture}
	\caption{The closed slicing of $\mathrm{dS}_4$.} \label{fig:expandingdesitter}	
\end{figure}

\noindent To conformally compactify $\mathrm{dS}_4$, however, we need a further change of coordinates
\[ \tan\left( \frac{\tau}{2} \right) = \tanh \left( \frac{H \eta}{2} \right). \]
In terms of $\tau$ the metric becomes
\begin{equation} \label{unrescaleddeSitter} \d s^2 = \frac{1}{ H^2 \cos^2 \tau } \left( \d \tau^2 - \mathfrak{s}_3 \right), \end{equation}
where $- \pi/2 < \tau < \pi/2$. This makes it obvious as to what should be taken as the conformal factor $\Omega$ to compactify $\mathrm{dS}_4$, namely
\[ \Omega = H \cos \tau, \]
and we define
\begin{equation} \label{EinsteinCylindermetric} \d \hat{s}^2 \defeq \Omega^2 \d s^2 = \d \tau^2 - \mathfrak{s}_3 \eqdef \mathfrak{e}. \end{equation}
In this conformal scale the hypersurfaces $\{ \tau = \pm \pi/2 \}$ are regular, in contrast to the physical metric \eqref{unrescaleddeSitter}. In fact, the metric $\mathfrak{e}$ clearly extends smoothly for all $\tau \in \mathbb{R}$, so one may consider the extended spacetime $(\mathfrak{E}, \mathfrak{e}) \defeq (\mathbb{R} \times \mathbb{S}^3, \mathfrak{e})$ known as the Einstein cylinder. We thus identify compactified de Sitter space $\widehat{\mathrm{dS}}_4$ with the subset $\left[ - \pi/2, \, \pi/2 \right] \times \mathbb{S}^3$ of the Einstein cylinder $\mathfrak{E}$ by attaching to \eqref{unrescaleddeSitter} the boundary $\scri \defeq \{ \Omega = 0\} = \{ | \tau | = \pi/2 \}$. This boundary is the union of two disjoint smooth surfaces
\[ \scri^+ = \left\{ \tau = \frac{\pi}{2} \right\} \quad \text{and} \quad \scri^- = \left\{ \tau = - \frac{\pi}{2} \right\}, \]
which we call \emph{future null infinity} and \emph{past null infinity} respectively. Note that $\scri^\pm$ are \emph{spacelike} hypersurfaces of $\mathfrak{E}$; the name \emph{null} infinity derives from the fact that $\scri^\pm$ is where all future (past) pointing null geodesics in de Sitter space end up at infinity. Note also that the vector field $T \defeq \partial/\partial \tau$ is a timelike Killing field in $\mathfrak{E}$, and in particular it is automatically uniformly timelike since $\mathfrak{E}$ is spatially compact.
\begin{figure}[h]
\centering
	\begin{tikzpicture}
	\centering
	\node[inner sep=0pt] (einsteincylinder) at (-2,0)
    	{\includegraphics[width=.18\textwidth]{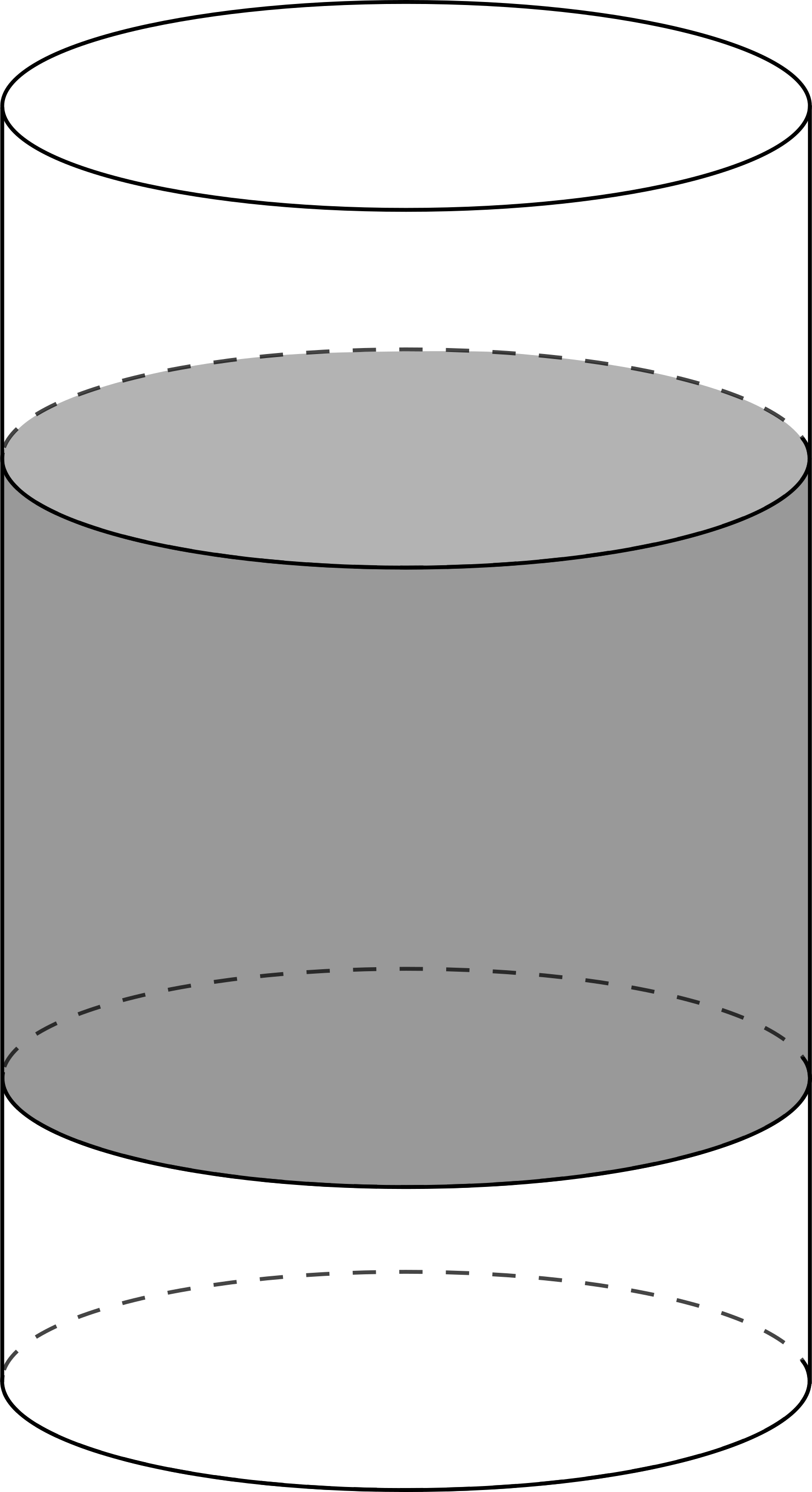}};

	\node[label={[shift={(0.6,-0.4)}]$\tau$}] {};
	\draw[->] (0.3,-0.3) .. controls (0.3,-1) .. (0.3,0.7);
	\node[label={[shift={(-4.3,1.3)}]$\scri^+$}] {};
	\draw[->] (-4.4,1.4) .. controls (-4.6, 0.9) and (-2.1,1.5) .. (-2,0.7);
	\node[label={[shift={(-4.3,-0.9)}]$\scri^-$}] {};
	\draw[->] (-4.4,-0.8) .. controls (-4.6, -1.3) and (-2.1,-0.7) .. (-2,-1.5);

	\end{tikzpicture}
	\caption{Compactified de Sitter space $\widehat{\mathrm{dS}}_4$ in the Einstein cylinder $\mathfrak{E}$.} \label{fig:einsteincylinder}	
\end{figure}
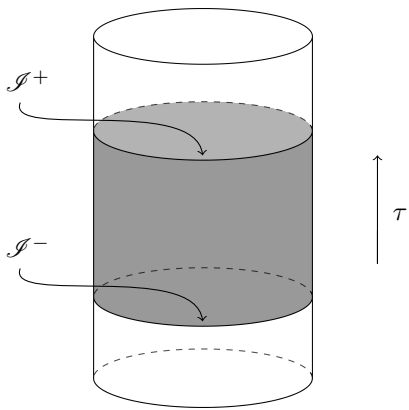

\noindent As a result, $T$ provides a uniformly spacelike foliation of $\mathfrak{E}$ by the level surfaces of the coordinate $\tau$ given explicitly by $\mathcal{F} =  \{ \mathbb{S}^3_\tau \defeq \mathbb{S}^3 \times \{ \tau \} : \tau \in \mathbb{R} \}$. Our energies will be defined with respect to $\mathcal{F}$.

\begin{remark} The fact that $\scri$ is spacelike is, of course, a consequence of the fact that $\mathrm{dS}_4$ is a solution to Einstein's equations with a positive cosmological constant $\lambda$,
\[ R_{ab} = \lambda g_{ab}. \]
Indeed, in general the norm squared \emph{on $\scri$} of the normal to $\scri$ is 
\[ (\nabla_a \Omega) (\nabla^a \Omega) \equalhat \frac{1}{3} \lambda. \]
In the case of $\mathrm{dS}_4$, $\lambda= 3 H^2$ so that $(\nabla \Omega )^2 \equalhat H^2 > 0$. Note that $H$ corresponds to the Hubble constant in vacuum.
\end{remark}

Writing the $3$-sphere metric as $\mathfrak{s}_3 = \d \zeta^2 + (\sin^2 \zeta) \mathfrak{s}_2$ for $\zeta \in [0, \pi]$ and quotienting by the $\mathrm{SO}(3)$ symmetry group of $\mathfrak{s}_2$ we obtain the Penrose diagram for $\mathrm{dS}_4$,
\begin{figure}[h]
\centering
	\begin{tikzpicture}
	\centering
	\node[inner sep=0pt] (desitterpenrose) at (3.4,0)
    	{\includegraphics[width=.24\textwidth]{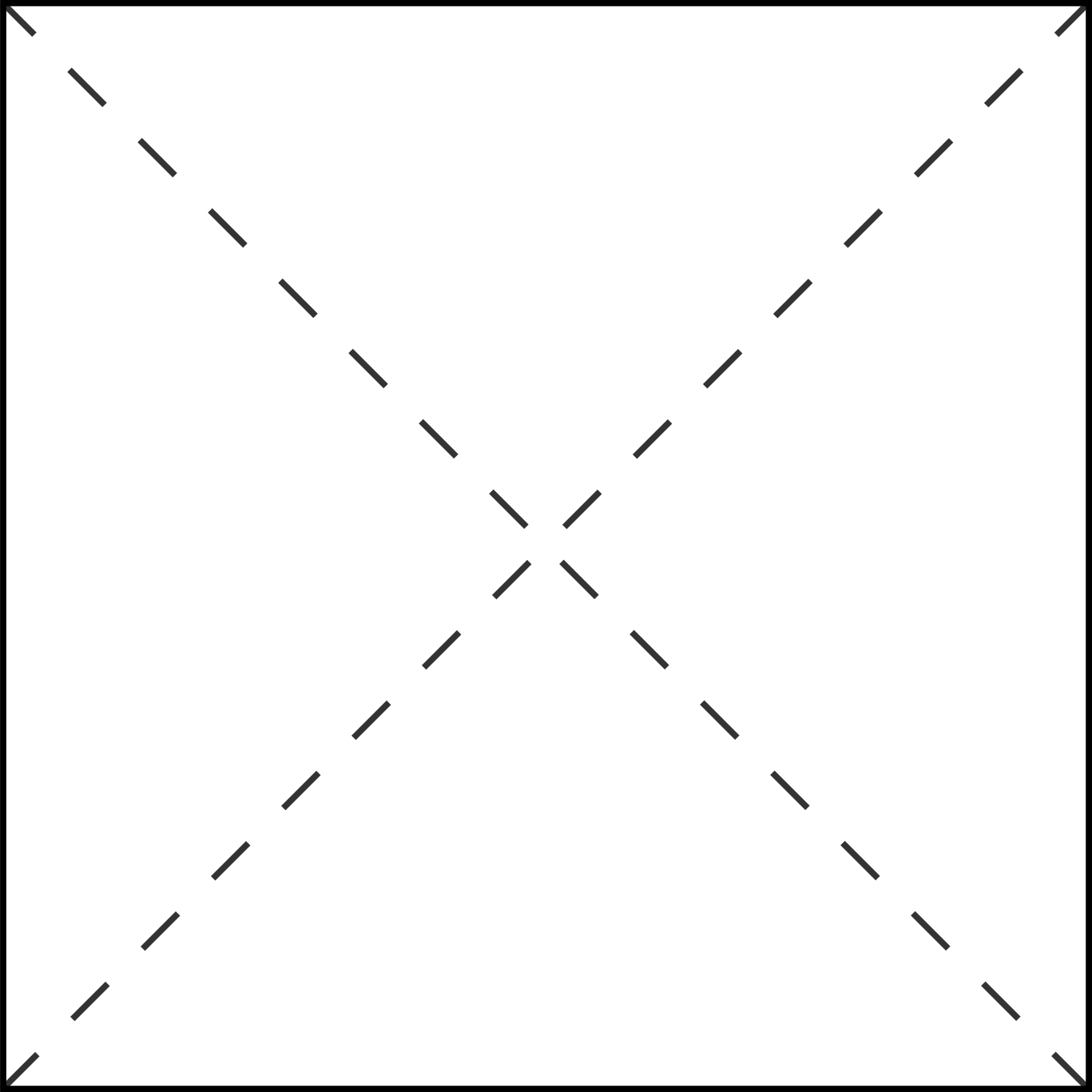}};

	\node[label={[shift={(0.2,-0.4)}]North Pole}] {};
	\node[label={[shift={(6.6,-0.4)}]South Pole}] {};
	\node[label={[shift={(3.4,1.9)}]$\scri^+$}] {};
	\node[label={[shift={(3.4,-2.7)}]$\scri^-$}] {};
	\node[label={[shift={(4.5,-0.4)}] I }] {};
	\node[label={[shift={(3.4,0.6)}] II }] {};
	\node[label={[shift={(2.4,-0.4)}] III }] {};
	\node[label={[shift={(3.4,-1.4)}] IV }] {};

	\end{tikzpicture}
	\caption{The Penrose diagram for $\mathrm{dS}_4$.} \label{fig:desitterpenrose}	
\end{figure}
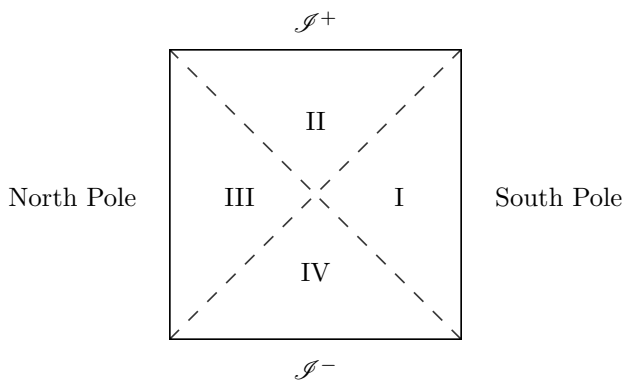

\noindent The coordinate $\zeta$ varies from $0$ to $\pi$ going from left to right, with the vertical lines $\{ \zeta = 0 \}$ and $\{ \zeta = \pi \}$ representing the North Pole and the South Pole respectively. The coordinate $\tau$ varies from $-\pi/2$ to $\pi/2$ going up, with the horizontal lines $\{ \tau = - \pi/2 \} $ and $\{ \tau = \pi/2 \}$ representing past and future null infinities $\scri^\pm$, as remarked earlier. The dashed lines are the past and future horizons for an observer at the South Pole: a classical observer sitting at $\{ \zeta = \pi \}$ can never observe the region $\mathrm{II} \cup \mathrm{III}$, and can never send a signal to the region $\mathrm{III} \cup \mathrm{IV}$. Thus region I is the region of communications for an observer at the South Pole, while region III is completely inaccessible.

\subsection{Static Coordinates}

A set of physical space coordinates on $\mathrm{dS}_4$ that exhibit an explicit future-pointing timelike Killing field in the region $\mathrm{I}$ may be constructed by defining
\[ r = \frac{\sin \zeta}{H \cos \tau}, \qquad \quad \tanh(Ht) = \frac{\sin \tau}{\cos \zeta} \]
for $\tau \in (-\pi/2, \pi/2)$ and $\zeta \in (0, \pi)$. Then the unrescaled $\mathrm{dS}_4$ metric takes the form
\begin{equation} \label{physicaldSmetricstatic} \d s^2 = F(r) \d t^2 - F(r)^{-1} \d r^2 - r^2 \mathfrak{s}_2, \end{equation}
where $F(r) = (1 - H^2 r^2)$. In these coordinates the cosmological horizons represented by the dashed lines in \cref{fig:desitterpenrose} are given by $\{ r = 1/H \}$, $\scri^\pm$ are given by $\{ r = \infty \}$, the North and South Poles are at $\{ r = 0 \}$, and the four corners of the Penrose diagram are at $\{ t = \pm \infty \}$. The vector field $\partial/\partial t$ is manifestly a timelike Killing vector in the region $\{ r < 1/H \}$, but becomes null on the cosmological horizon $\{ r=1/H \}$. It is future-pointing in the region $\mathrm{I}$, past-pointing in the region $\mathrm{III}$, and spacelike in the regions $\mathrm{II}$ and $\mathrm{IV}$. The arrows in \cref{fig:desitterpenrosestatic} represent the directions of the flow of $\partial/\partial t$.
\begin{figure}[h]
\centering
	\begin{tikzpicture}
	\centering
	\node[inner sep=0pt] (desitterpenrose2) at (3.4,0)
    	{\includegraphics[width=.24\textwidth]{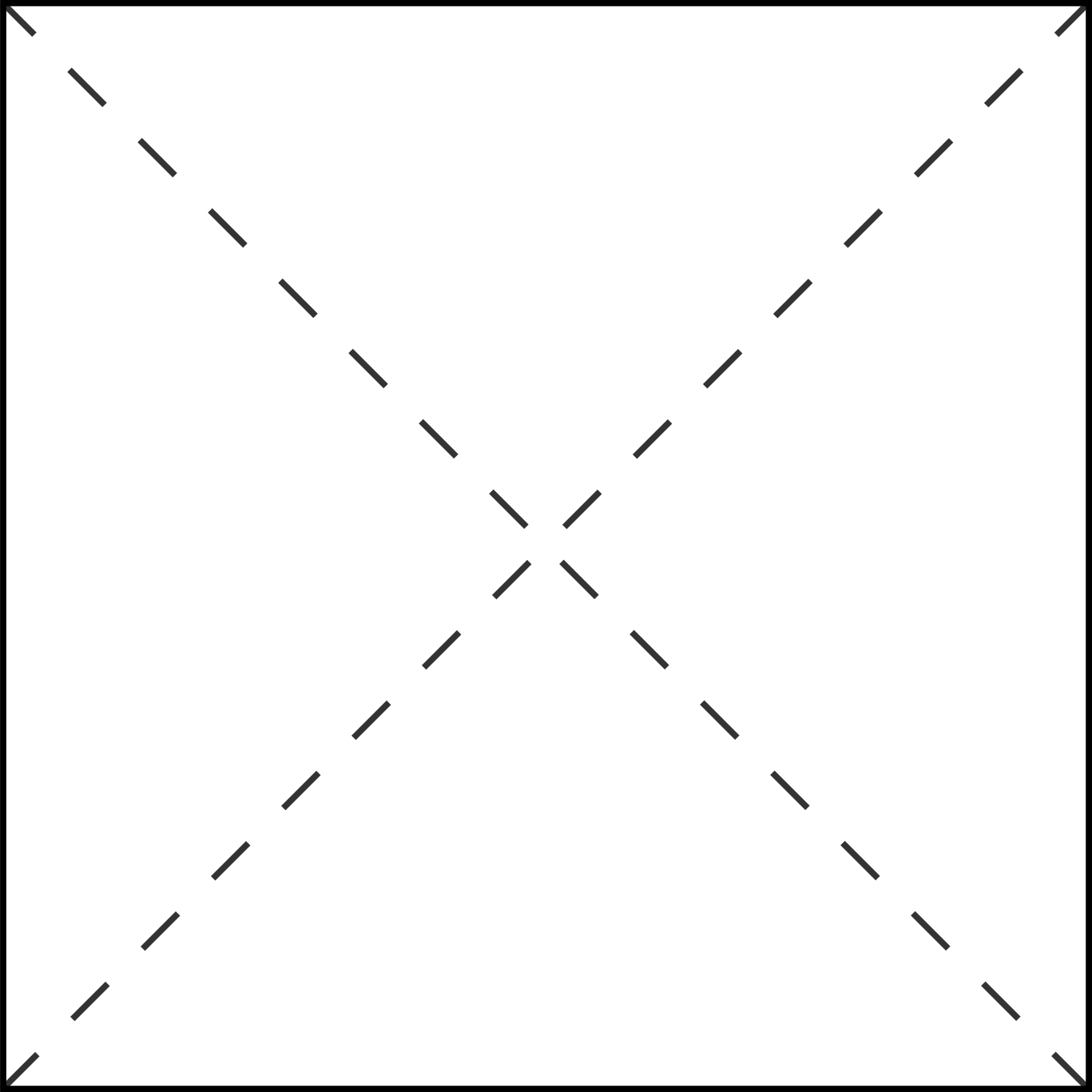}};

	\node[label={[shift={(0.8,-0.4)}]$r=0$}] {};
	\node[label={[shift={(6,-0.4)}]$r=0$}] {};
	\node[label={[shift={(3.4,1.9)}]$r=\infty$}] {};
	\node[label={[shift={(3.4,-2.7)}]$r=\infty$}] {};
	\node[label={[shift={(5.8,-2.75)}]$t=-\infty$}] {};
	\node[label={[shift={(1.2,-2.7)}]$t=+\infty$}] {};
	\node[label={[shift={(5.8,1.85)}]$t=+\infty$}] {};
	\node[label={[shift={(1.2,1.9)}]$t=-\infty$}] {};
	\draw[->] (5,-1.2) .. controls (4, -0.2) and (4,0.2) .. (5,1.2);
	\draw[->] (1.8,1.2) .. controls (2.8, 0.2) and (2.8,-0.2) .. (1.8,-1.2);
	\draw[->] (2.25, 1.6) .. controls (3.25, 0.6) and (3.65, 0.6) .. (4.65,1.6);
	\draw[<-] (2.25, -1.6) .. controls (3.25,-0.6) and (3.65,-0.6) .. (4.65,-1.6);
	%\node[label={[shift={(4.5,-0.4)}] I }] {};
	%\node[label={[shift={(3.4,0.6)}] II }] {};
	%\node[label={[shift={(2.4,-0.4)}] III }] {};
	%\node[label={[shift={(3.4,-1.4)}] IV }] {};

	\end{tikzpicture}
	\caption{Static coordinates on $\mathrm{dS}_4$.} \label{fig:desitterpenrosestatic}	
\end{figure}

\subsection{Choice of Energy-Momentum Tensor on \texorpdfstring{$\mathfrak{E}$}{E}}

From now on we denote by $\phi$ and $A_a$ the scalar field and Maxwell potential on the Einstein cylinder $\mathfrak{E}$, and by $\tilde{\phi}$ and $\tilde{A}_a$ the conformally related physical fields on de Sitter space $\mathrm{dS}_4$,
\begin{equation} \label{conformalweights} \phi = \Omega^{-1} \tilde{\phi}, \qquad \quad A_a = \tilde{A}_a, \end{equation}
where $\Omega = H \cos \tau$.

We define the energy-momentum tensor for the system \eqref{phiAequation1} on $\mathfrak{E}$ to be
\begin{align} \begin{split} \label{chosenenergymomentum} \mathbf{T}_{ab}[\phi,A] & \defeq - F_{ac} F_b^{\phantom{b}c} + \frac{1}{4} \mathfrak{e}_{ab} F_{cd} F^{cd} + \overline{\D_{(a} \phi} \D_{b)} \phi - \frac{1}{2}\mathfrak{e}_{ab} \overline{\D_c \phi} \D^c \phi + \frac{1}{2} \mathfrak{e}_{ab} |\phi|^2 \\
& \defeq \mathbf{T}_{ab}[A] + \mathbf{T}_{ab}[\phi].
\end{split} \end{align}
One can check by direct calculation that, as a consequence of the field equations \eqref{phiAequation1}, $\mathbf{T}_{ab}$ is conserved,
\[ \nabla^a \mathbf{T}_{ab} = 0, \]
so $\mathbf{T}_{ab}$ is suitable for defining a conserved energy for the system \eqref{phiAequation1},
\begin{equation} \label{fullgeometricenergy} \mathcal{E}_\tau[\phi,A] \defeq \int_{\mathbb{S}^3_\tau} \mathbf{T}_{00}[\phi,A] \dvol_{\mathfrak{s}_3} = \int_{\mathbb{S}^3_\tau} \mathbf{T}_{ab}[\phi,A] T^a T^b \dvol_{\mathfrak{s}_3}. \end{equation}
Since $T^a$ is Killing on $\mathfrak{E}$, this clearly satisfies
\[ \frac{\d }{\d \tau} \mathcal{E}_\tau[\phi,A] = 0 \]
if the field equations \eqref{phiAequation1} are satisfied. We call \eqref{fullgeometricenergy} the \emph{geometric} energy for the system \eqref{phiAequation1}. We also define the geometric energies for the individual sectors of the scalar field $\phi$ and the Maxwell potential $A_a$,
\[ \mathcal{E}_\tau[\phi] \defeq \int_{\mathbb{S}^3_\tau} \mathbf{T}_{00}[\phi] \dvol_{\mathfrak{s}_3}, \qquad \quad \mathcal{E}_\tau[A] \defeq \int_{\mathbb{S}^3_\tau} \mathbf{T}_{00}[A] \dvol_{\mathfrak{s}_3}. \]
The sectorial geometric energies $\mathcal{E}_\tau[\phi]$ and $\mathcal{E}_\tau[A]$ are not conserved individually and can exchange energy throughout the evolution, but of course the total geometric energy $\mathcal{E}_\tau[\phi,A] = \mathcal{E}_\tau[\phi] + \mathcal{E}_\tau[A]$ is. For $m \geq 1$ we also define the Sobolev-type approximate energies
\begin{align*} & \mathrm{S}_m[\phi] \defeq \|\dot{\phi} \|^2_{H^{m-1}} + \| \phi \|^2_{H^m}, && \mathrm{S}_m[A] \defeq \mathrm{S}_m[\mathbf{A}] + \mathrm{S}_m[A_0], \\
& \mathrm{S}_m[\mathbf{A}] \defeq \| \dot{\mathbf{A}} \|^2_{H^{m-1}} + \| \mathbf{A} \|^2_{H^m}, && \mathrm{S}_m[\phi, \mathbf{A}] \defeq \mathrm{S}_m[\phi] + \mathrm{S}_m[\mathbf{A}], \\ 
& \mathrm{S}_m[A_0] \defeq \| A_0 \|^2_{H^m}, && \mathrm{S}_m[\phi, A] \defeq \mathrm{S}_m[\phi, \mathbf{A}] + \mathrm{S}_m[A_0],
\end{align*}
where $H^0 = L^2$. Furthermore, for brevity we will often simply write $\mathrm{S}_m$ to mean $\mathrm{S}_m[\phi, A]$.

\subsection{Scaling of Initial Energies}

We will consider initial data on the hypersurface $\{ \tau = 0 \} = \{ \eta = 0 \}$ and use the coordinate $\tau$ and the metric $\mathfrak{e}$ on the rescaled spacetime, and the coordinate $\eta$ and the metric \eqref{closedslicing} on the physical spacetime. By differentiating the relationship $\tan(\tau/2) = \tanh(H \eta/2)$ we find
\[ \d \tau = \frac{H}{\cosh(H \eta)} \d \eta, \]
so raising indices with $\mathfrak{e}^{-1} = \Omega^{-2} g^{-1}$, where $g$ is the metric \eqref{closedslicing}, we find that $\partial_\tau$ and $\partial_\eta$ are related by
\[ \partial_\tau = \frac{\cosh(H \eta)}{H} \partial_\eta. \]
Furthermore, the conformal factor $\Omega$ in the global coordinates \eqref{closedslicing} is given by
\[ \Omega = H \cos \tau = \frac{H}{\cosh(H \eta)}. \]  
Consider the rescaled energies
\[ \mathrm{S}_m[\phi, A](\tau) = \| \dot{\phi} \|^2_{H^{m-1}}(\tau) + \| \phi \|^2_{H^m}(\tau) + \| \dot{\mathbf{A}} \|^2_{H^{m-1}}(\tau) + \| \mathbf{A} \|^2_{H^m}(\tau) + \| A_0 \|^2_{H^m}(\tau). \]
On the initial surface $\{ \tau = 0 \} = \{ \eta = 0 \}$ the conformal factor is a constant and has vanishing derivative, $\partial_\tau \Omega|_{\tau = 0} = 0$, so the rescaled scalar field $\phi$ is related to the physical scalar field $\tilde{\phi}$ by
\[ \phi|_{\tau = 0 } = (\Omega^{-1} \tilde{\phi} )|_{\tau = 0} = \frac{1}{H} \tilde{\phi}|_{\eta = 0}, \]
while their time derivatives are related by
\[ \dot{\phi}|_{\tau = 0} = ( \Omega^{-1} \partial_\tau \tilde{\phi} - (\partial_\tau \Omega)\Omega^{-2} \tilde{\phi} )|_{\tau = 0} = \frac{1}{H^2} \partial_\eta \tilde{\phi}|_{\eta = 0}. \]
Since the conformal factor is independent of the $\mathbb{S}^3$ coordinates, $\slashgrad \Omega = 0$, and the metric induced on $\{ \eta = 0 \}$ by \eqref{closedslicing} is equivalent to $\mathfrak{s}_3$, the rescaled and physical norms of the scalar field are equivalent,
\[ \| \dot{\phi} \|^2_{H^{m-1}}(\tau = 0) + \| \phi \|^2_{H^m}(\tau = 0) \simeq \| \partial_\eta \tilde{\phi} \|^2_{H^{m-1}}(\eta = 0) + \| \tilde{\phi} \|^2_{H^m}(\eta = 0), \]
where there is equality if $H=1$. One similarly checks that
\begin{align*} & \| \dot{\mathbf{A}} \|^2_{H^{m-1}}(\tau = 0) + \| \mathbf{A} \|^2_{H^m}(\tau = 0) \simeq \| \partial_\eta  \tilde{\mathbf{A}} \|^2_{H^{m-1}}(\eta = 0) + \| \tilde{\mathbf{A}} \|^2_{H^m}(\eta = 0)
\end{align*}
and
\[ \| A_0 \|^2_{H^m}(\tau = 0) \simeq \| \tilde{A}_\eta \|^2_{H^m}(\eta = 0),  \]
where $A_0 \d \tau + \mathbf{A}_\mu \d x^\mu = A = \tilde{A} = \tilde{A}_\eta \d \eta + \tilde{\mathbf{A}}_\mu \d x^\mu$, and $x^\mu$ are coordinates on $\mathbb{S}^3$. Thus
\begin{equation} \label{scalingofenergies} \mathrm{S}_m[\phi, \mathbf{A}](\tau = 0) \simeq \mathrm{S}_m[\tilde{\phi}, \tilde{\mathbf{A}}](\eta = 0), \end{equation}
and also $\mathrm{S}_m[A_0](\tau=0) \simeq \mathrm{S}_m[\tilde{A}_\eta](\eta=0)$.

\section{Main Theorems} \label{sec:maintheorems}

\begin{definition} Let $\tilde{\Sigma}$ be a Cauchy surface in $\mathrm{dS}_4$ and consider data for the Maxwell-scalar field system on $\Sigma$ the corresponding Cauchy surface in $\widehat{\mathrm{dS}}_4$. We say the data
\[ (\phi_0, \mathbf{A}_0, \phi_1, \mathbf{A}_1, a_0) = (\phi, \mathbf{A}, \dot{\phi}, \dot{\mathbf{A}}, A_0)|_{\Sigma} \]
is \emph{admissible} if it satisfies the strong Coulomb gauge\footnote{See \Cref{sec:strongCoulombgauge}.} and $a_0$ solves the elliptic equation
\[ - \slashed{\Delta} a_0 + | \phi_0 |^2 a_0 = -\operatorname{Im}(\bar{\phi}_0 \phi_1) \]
on $\Sigma$. 
\end{definition}

\begin{theorem}[Energy Estimates] \label{thm:estimates} Let $m \in \mathbb{N}$. For sufficiently $\mathrm{S}_m[\phi, \mathbf{A}]$-small admissible data on $\Sigma$ for the Maxwell-scalar field system on $\widehat{\mathrm{dS}}_4 \simeq \mathbb{S}^3 \times [-\pi/2, \pi/2]$ in strong Coulomb gauge one has
\[ \mathrm{S}_m[\phi, A](0) \simeq \mathrm{S}_m[\phi, A](\tau) \]
for all $\tau \in [-\pi/2, \pi/2]$. In particular,
\[ \mathrm{S}_m[\phi,A](\scri^-) \simeq \mathrm{S}_m[\phi, A](\scri^+), \]
where $\scri^\pm = \{ \tau = \pm \pi/2 \}$ is the future (past) null infinity of de Sitter space $\mathrm{dS}_4$.
\end{theorem}

\begin{theorem}[Scattering for Small Data] \label{thm:scatteringtheory} For $m \geq 2$ let $\mathrm{S}^0_m$ be the subset of $H^m(\Sigma)^2 \times H^{m-1}(\Sigma)^2 \times H^m(\Sigma)$ of distributions $u_0$ of admissible data on $\Sigma$ and let $\mathrm{S}^\pm_m$ be the subset of $H^m(\scri^\pm)^2 \times H^{m-1}(\scri^\pm)^2 \times H^m(\scri^\pm)$ of distributions $u^\pm$ of admissible data on $\scri^\pm$ of de Sitter space, all equipped with the natural norm $\sqrt{\mathrm{S}_m}$. Denote by $B_\epsilon$ the open ball of radius $\epsilon$ in $(H^m)^2 \times (H^{m-1})^2 \times H^m$, and write $\mathrm{S}_{m,\epsilon}^0 = \mathrm{S}_m^0 \cap B_\epsilon$ and $\mathrm{S}^{\pm}_{m,\epsilon} = \mathrm{S}^\pm_m \cap B_{\epsilon}$. Then for every $m \geq 2$ there exist $\epsilon_0, \epsilon_1 > 0$, $0 < \delta \ll 1$ and sets $\mathscr{D}^\pm_{m,\epsilon_1}$ with $ \mathrm{S}^\pm_{m, \delta} \subset \mathscr{D}^\pm_{m,\epsilon_1} \subset \mathrm{S}^\pm_{m,\epsilon_1}$ such that
\begin{enumerate}[(i)] 

\item there exist bounded invertible nonlinear operators $\mathfrak{T}^\pm_m$, called the forward and backward wave operators
\[ \mathfrak{T}^\pm_m : \mathrm{S}^0_{m,\epsilon_0} \longrightarrow \mathscr{D}^\pm_{m, \epsilon_1} \subset \mathrm{S}^\pm_{m,\epsilon_1}, \]
such that $u^\pm = \mathfrak{T}^\pm_m(u_0)$ is the forward (backward) Maxwell-scalar field development of $u_0$ on $\widehat{\mathrm{dS}}_4$ restricted to $\scri^\pm$, and

\item there exists a bounded invertible nonlinear scattering operator
\[ \mathscr{S}_m: \mathscr{D}^-_{m, \epsilon_1} \longrightarrow \mathscr{D}^+_{m, \epsilon_1} \]
given by 
\[ \mathscr{S}_m = \mathfrak{T}_m^+ \circ (\mathfrak{T}_m^-)^{-1} \]
such that $u^+ = \mathscr{S}_m(u^-)$ is the Maxwell-scalar field development of $u^-$ on $\widehat{\mathrm{dS}}_4$ restricted to $\scri^+$.
\end{enumerate}
\end{theorem}

\begin{theorem}[Small Data Decay Rates] \label{thm:decayrates} Let $\tilde{\phi} = \Omega \phi$ and $\tilde{A}_a = A_a$ be the physical fields related to the conformally rescaled fields $\phi$ and $A_a$ by \cref{conformalweights}. Suppose $\mathrm{S}_2[\tilde{\phi}, \tilde{\mathbf{A}}]$ is small initially. Then the Maxwell-scalar field development $(\tilde{\phi}, \tilde{A})$ of this initial data satisfies the estimates
\begin{align*} &  | \tilde{\phi} | \la \e^{-H |\eta|}, && | \tilde{A}_\eta | \la \e^{-H |\eta|}, &&| \tilde{\mathbf{A}} |_{\mathfrak{s}_3} \la 1	
\end{align*}
as $|\eta| \to \infty$. Furthermore, in the static coordinates \cref{physicaldSmetricstatic}
\begin{align*} & |\tilde{\phi}| \la_r \e^{-H|t|}, && | \tilde{A}_t | \la_r \e^{-H|t|}, && |\tilde{A}_r | \la_r \e^{-H|t|}, && \frac{1}{r}| \tilde{A} |_{\mathfrak{s}_2} \la_r \e^{-H|t|}	
\end{align*}
as $|t| \to \infty$ and $r$ is fixed. Moreover, if $\mathrm{S}_3[\tilde{\phi}, \tilde{\mathbf{A}}]$ is small initially then there exists a constant $c$ such that
\[ \tilde{\phi} \sim c \tilde{\Phi}_1 + \mathcal{O}\left( \e^{-2Ht} \right) \]
as $t \to + \infty$, where $\tilde{\Phi}_1 = F(r)^{-1/2} \e^{-Ht}$ is a solution of the linear uncharged conformally invariant wave equation on $\mathrm{dS}_4$.
\end{theorem}

\section{Field Equations and Gauge Fixing} \label{sec:fieldequationsgaugefixing}

The field equations \eqref{phiAequation1} can be written out in terms of the Maxwell potential $A_a$,
\begin{align} \begin{split} \label{phiAequation2} & \Box A_a - \nabla_a (\nabla^b A_b) + R_{ab} A^b = - \Im \left( \bar{\phi} \D_a \phi \right), \\
& \Box \phi + 2iA_a \nabla^a \phi + \left( \frac{1}{6} R - A_a A^a + i \nabla^a A_a \right) \phi = 0.
\end{split}
\end{align}
We shall be commuting differential operators into these equations, so it will be convenient to introduce the operators representing their left-hand sides. For any $1$-form $\omega$ and any scalar field $\psi$ we set
\[ \mathfrak{M}(\omega)_a \defeq \Box \omega_a - \nabla_a (\nabla_b \omega^b) + R_{ab} \omega^b \quad \text{and} \quad \mathfrak{S}(\psi) \defeq \D^a \D_a \psi + \frac{1}{6} R \psi. \]
The system \eqref{phiAequation2} is then equivalent to 
\begin{equation} \label{abstractfieldequations} \mathfrak{M}(A)_a = - \operatorname{Im}\left( \bar{\phi} \D_a \phi \right) \quad \text{and} \quad \mathfrak{S}(\phi) = 0. \end{equation}
In the following sections we specialise to the case of the Einstein cylinder $(\mathfrak{E}, \mathfrak{e})$. As noted earlier, for ease of notation we will not hat any rescaled quantities on $\mathfrak{E}$ and instead denote the corresponding physical quantities on $\mathrm{dS}_4$ with a tilde, as in $\tilde{\phi}$ or $\tilde{A}_a$. For the metric $\mathfrak{e}$ we compute
\[ R = 6 \qquad \text{and} \qquad R_{ab} \d x^a \d x^b = -2 \mathfrak{s}_3 . \]

\subsection{Strong Coulomb Gauge} \label{sec:strongCoulombgauge}

We will work in the Coulomb gauge adapted to the foliation $\mathcal{F}$,
\begin{equation} \label{Coulombgauge} \slashgrad \cdot \mathbf{A} = 0, \end{equation}
but will also need to use the residual gauge freedom to fix the gauge fully. More precisely, given a solution $(A, \phi)$ to the Maxwell-scalar field system \eqref{phiAequation2}, a general gauge transformation sends $\phi \mapsto \mathrm{e}^{- i \chi} \phi $ and $A_a \mapsto A_a + \nabla_a \chi$, and \cref{Coulombgauge} is imposed by solving the elliptic equation
\[ \slashed{\Delta} \chi = - \slashgrad \cdot \mathbf{A} \]
on $\mathbb{S}^3_\tau$ for every fixed $\tau$. This does not determine $\chi$ uniquely: there is still the residual gauge freedom of $\chi \mapsto \chi + \chi_{\text{res.}}$, where $\chi_{\text{res.}}$ solves
\[ \slashed{\Delta} \chi_{\text{res.}} = 0 \]
on each $\mathbb{S}^3_\tau$. Because $\mathbb{S}^3$ is compact, the kernel of the Laplacian $\slashed{\Delta}$ is just the vector space of constant functions, i.e. those $\chi_{\text{res.}}$ which satisfy $\slashgrad \chi_{\text{res.}} = 0$, but the $\tau$ dependence in the $\chi_{\text{res.}}$ is still arbitrary. Thus in the Coulomb gauge we have the residual gauge freedom 
\begin{align*} & \phi \longmapsto \mathrm{e}^{-i \chi_{\text{res.}}(\tau)} \phi , && 
A_0 \longmapsto A_0 + \dot{\chi}_{\text{res.}}(\tau), && \mathbf{A} \longmapsto \mathbf{A},
\end{align*}
which allows one to choose
\[ \dot{\chi}_{\text{res.}}(\tau) = - \frac{1}{|\mathbb{S}^3|} \int_{\mathbb{S}^3} A_0 (\tau) \dvol_{\mathfrak{s}_3} \eqdef - \bar{A}_0(\tau) \]
and so impose the additional gauge condition
\[ \bar{A}_0(\tau) = 0. \]
This determines $\chi_{\text{res.}}$ up to the addition of a global constant, so there is very little remaining gauge freedom. Indeed, constants are irrelevant for the gauge transformation of $A_a$ and only impart a constant phase change in $\phi$, so we have now fixed the gauge as completely as possible. We call this stronger gauge fixing condition
\begin{equation} \label{strongCoulombgauge} \slashgrad \cdot \mathbf{A} = 0, \qquad \bar{A}_0 =0 \end{equation}
\emph{strong Coulomb gauge}. For us, the most useful feature of the strong Coulomb gauge will be the fact that in this gauge $A_0$ will obey the Poincar\'e inequality on each leaf $\mathbb{S}^3_\tau$ of $\mathcal{F}$,
\[ \| A_0 \|_{L^2}(\tau) \leq C \| \slashgrad A_0 \|_{L^2}(\tau). \]
In strong Coulomb gauge the field equations \eqref{phiAequation2} are equivalent to the system
\begin{align} \begin{split} \label{phiAequationsSCG} & \Box \phi + 2i A_0 \dot{\phi} - 2i \mathbf{A} \cdot \slashgrad \phi + (1 - A^2_0 + |\mathbf{A}|^2 + i \dot{A}_0 ) \phi = 0, \\
& \Box \mathbf{A} + (2 + |\phi|^2) \mathbf{A} = - \operatorname{Im}(\bar{\phi} \slashgrad \phi) + \slashgrad \dot{A}_0, \\ 
 - &\slashed{\Delta} A_0 + |\phi|^2 A_0 = - \operatorname{Im}(\bar{\phi} \dot{\phi}), \\
& \slashgrad \cdot \mathbf{A} = 0, \\
& \bar{A}_0(\tau) = 0.
\end{split}
\end{align}
We do not prescribe initial data on $A_0$ since it is non-dynamical: it is completely determined by $\phi$ and $\dot{\phi}$ via the elliptic equation on each slice of constant $\tau$. It is convenient to incorporate the constraint $\slashgrad \cdot \mathbf{A} = 0$ into the equations by projecting the equation for $\mathbf{A}$ onto divergence free $1$-forms on $\mathbb{S}^3$. Let $\mathcal{P}$ be this projection (see \Cref{sec:divergencefree1forms}); then since 
\[ \slashgrad \cdot \Box \mathbf{A} = \Box (\slashgrad \cdot \mathbf{A}) - 2 \slashgrad \cdot \mathbf{A} = 0 \]
and
\[ \operatorname{curl} \slashgrad \dot{A}_0 = 0, \]
applying $\mathcal{P}$ to the equation for $\mathbf{A}$ gives
\[ \Box \mathbf{A} + 2 \mathbf{A} + \mathcal{P}\left(|\phi|^2 \mathbf{A}\right) = - \mathcal{P} \left( \operatorname{Im} ( \bar{\phi} \slashgrad \phi ) \right). \]
Thus the system \eqref{phiAequationsSCG} is equivalent to 
\begin{align} \begin{split} \label{phiAequationsCGprojected} & \Box \phi + 2i A_0 \dot{\phi} - 2i \mathbf{A} \cdot \slashgrad \phi + (1 - A_0^2 + |\mathbf{A}|^2 + i \dot{A}_0) \phi = 0, \\
& \Box \mathbf{A} + 2 \mathbf{A} + \mathcal{P} \left( |\phi|^2 \mathbf{A} \right) = - \mathcal{P} \left( \operatorname{Im} ( \bar{\phi} \slashgrad \phi) \right), \\
 - &\slashed{\Delta} A_0 + |\phi|^2 A_0 = - \operatorname{Im} (\bar{\phi} \dot{\phi} ), \\
& \bar{A}_0(\tau) = 0,
\end{split}
\end{align}
provided one considers divergence-free initial data for $\mathbf{A}$ and $\dot{\mathbf{A}}$. Indeed, it is easily seen that $v = \slashgrad \cdot \mathbf{A}$ satisfies 
\[ \Box v = 0, \]
so $v \equiv 0$ whenever $v = 0$ and $\dot{v} = 0$ initially.

In addition to the restriction $\slashgrad \cdot \mathbf{A}_0 = 0 = \slashgrad \cdot \mathbf{A}_1$ on the initial data, the extra gauge condition $\bar{A}_0 = 0$ restricts the set of initial data further. Suppose we prescribe initial data $\phi(\tau = 0) = \phi_0$ and $\dot{\phi}(\tau = 0) = \phi_1$. We must then solve for $A_0(\tau = 0) = a_0$ by solving
\begin{equation} \label{ellipticequationa0} - \slashed{\Delta} a_0 + |\phi_0|^2 a_0 = - \operatorname{Im}(\bar{\phi}_0 \phi_1), \end{equation}
so we must choose the initial data so that this solution has $\bar{a}_0 = 0$. Because $A_0$ is non-dynamical, it is not possible to write down an evolution equation for $\bar{A}_0$, but the gauge $\bar{A}_0 = 0$ is propagated nonetheless. This can be seen by simply replacing all instances of $A_0$ in the system \eqref{phiAequationsCGprojected} with $A_0^{\circ} \defeq A_0 - \bar{A}_0$ and solving the system for $A_0^{\circ}$ in the space of mean zero functions. While $A_0$ need not part of the initial data (prescribing $(\phi, \mathbf{A}, \dot{\phi}, \dot{\mathbf{A}})|_\Sigma = (\phi_0, \mathbf{A}_0, \phi_1, \mathbf{A}_1)$ is enough), we can consider $A_0$ as part of the initial data if it is equal to the $a_0$ obtained by solving the elliptic equation initially.

We call data satisfying the above conditions \emph{admissible}.

\begin{remark} The condition $\bar{a}_0 = 0$ is a condition on the initial data for $\phi$ and can be seen explicitly as follows. Consider the operator
\[ L \defeq - \slashed{\Delta} + |\phi_0|^2 \]
on $\mathbb{S}^3$ and assume that $\phi_0$ is not identically zero (if it is, then the equation becomes $\slashed{\Delta} a_0 = 0$ and we can trivially choose the zero solution). We can classify the kernel of $L$ if the data $(\phi_0, \phi_1)$ is sufficiently regular, say $(\phi_0, \phi_1) \in H^2(\mathbb{S}^3) \times H^1(\mathbb{S}^3)$. Multiplying the equation $Lu = 0$ by $u$ and integrating we get
\[ \int_{\mathbb{S}^3} |\slashgrad u |^2 \dvol_{\mathfrak{s}_3} + \int_{\mathbb{S}^3} |\phi_0|^2 u^2 \dvol_{\mathfrak{s}_3} = 0, \]
so that $\slashgrad u = 0$. If $u \in H^2(\mathbb{S}^3) \hookrightarrow C^0(\mathbb{S}^3)$, continuity of $u$ and $\| \phi_0 u \|_{L^2} = 0$ imply that $u \equiv 0$. Thus as an operator from $H^2(\mathbb{S}^3)$ to $L^2(\mathbb{S}^3)$\footnote{For $\phi_0, u \in H^2(\mathbb{S}^3)$ it is easy to check that $|\phi_0|^2 u \in L^2(\mathbb{S}^3)$, so $L$ does indeed map into $L^2(\mathbb{S}^3)$.}, $L$ has trivial kernel. It follows from standard elliptic theory that the equation $L u = \psi $ has a unique solution $u \in H^2(\mathbb{S}^3)$ for $\psi \in L^2(\mathbb{S}^3)$, which we write as $u = L^{-1} \psi$. Since $(\phi_0, \phi_1) \in H^2(\mathbb{S}^3) \times H^1(\mathbb{S}^3)$ ensures\footnote{In fact, $H^2(\mathbb{S}^3) \cdot H^1(\mathbb{S}^3) \subset H^1(\mathbb{S}^3)$, by Sobolev Embedding.} that $\bar{\phi}_0 \phi_1 \in L^2(\mathbb{S}^3)$, we have
\[ a_0 = -L^{-1}\operatorname{Im}(\bar{\phi}_0 \phi_1) = (\slashed{\Delta} - |\phi_0|^2)^{-1} \operatorname{Im}(\bar{\phi}_0 \phi_1). \]
The requirement $\bar{a}_0 = 0$ may thus be written as the condition 
\begin{equation} \label{strongCGconstraint} \int_{\mathbb{S}^3} (\slashed{\Delta} - |\phi_0|^2)^{-1} \operatorname{Im}(\bar{\phi}_0 \phi_1) \dvol_{\mathfrak{s}_3} = 0 \end{equation}
on the initial data $(\phi_0, \phi_1)$.
\end{remark}

\begin{remark} If one defines the electric field $E_a \defeq F_{ab} T^b$, then the index $a = 0$ Maxwell's equation in \eqref{phiAequation1} reads
\[ \slashgrad \cdot \mathbf{E} = \operatorname{Im}(\bar{\phi} \D_0 \phi). \]
Integrating this over $\mathbb{S}^3$ shows that
\[ \int_{\mathbb{S}^3} \operatorname{Im}(\bar{\phi} \D_0 \phi) \dvol_{\mathfrak{s}_3} = 0. \]
In flat space the same observation imposes precise decay rates on the eletric field $\mathbf{E}$ at spatial infinity $i^0$ (and in particular implies a non-zero $r^{-2}$ term), so the source term $\operatorname{Im}(\bar{\phi} \D_0 \phi)$ is said to correspond to charge at $i^0$. Recent work by Yang and Yu \cite{YangYu2018} and Candy, Kauffman, and Lindblad \cite{CandyKauffmanLindblad2018} quantifies such non-zero charge decay rates of the Maxwell-scalar field system in flat space. In de Sitter space, due to the spatial compactness of the topology, there is no analogous behaviour.
\end{remark}

\begin{remark} The system \eqref{phiAequationsCGprojected} in principle exhibits the null structure of Klainerman and Machedon \cite{KlainermanMachedon1994}. However, their original null form estimates \cite{KlainermanMachedon1993} rely on the structure of the real numbers to use Fourier techniques, and are not immediately extendible to curved space.	
\end{remark}

\section{Well-Posedness} \label{sec:wellposedness}

We state a classical theorem, due to Choquet-Bruhat, and apply it to our case. It should be noted that the original theorem is slightly more general (for example, it considers the Dirac--Maxwell--Klein--Gordon system), but we do not wish to clutter the presentation with unnecessary details. Let $I$ be an interval in $\mathbb{R}$ and let
\[ E_m(\mathbb{S}^n \times I) \defeq \bigcap_{k=0}^m C^k_b ( I ; H^{m-k}(\mathbb{S}^n) ) \]
be the standard finite $m$-energy space for hyperbolic systems. The following theorem elucidates why first order (that is, $H^1$) energy estimates are insufficient to construct a scattering theory for the Maxwell-scalar field system and why $H^2$ estimates are good enough ($2 >3/2$).

\begin{theorem}[Y. Choquet-Bruhat, \cite{ChoquetBruhat1982}] \label{thm:ChoquetBruhatexistence} Consider the system \eqref{phiAequation1} on $\mathbb{S}^n \times \mathbb{R}$. Let $T$ be the timelike unit normal to $\mathbb{S}^n_\tau \defeq \mathbb{S}^n \times \{ \tau \}$, set $\mathbf{E}_a \defeq F_{ab} T^b = \nabla_a A_0 - \dot{A}_a$, and suppose that we are given data 
$a$, $\phi_0 \in H^m(\mathbb{S}_0^n)$ and $\mathbf{E}$, $\phi_1 \in H^{m-1}(\mathbb{S}_0^n)$ satisfying the constraint
\begin{equation} \label{constraint} \tag{$\dagger$} \slashgrad \cdot \mathbf{E} = a_0 |\phi_0|^2 + \operatorname{Im}(\bar{\phi}_0 \phi_1), \end{equation}
where $\slashgrad$ is the Levi--Civita connection on $\mathbb{S}^n_0$. Then there exists an interval $I_\sigma = (-\sigma, \sigma) \subset \mathbb{R}$ and $(A_a,\phi) \in E_m(\mathbb{S}^n \times I_\sigma)$ satisfying the system \eqref{phiAequation1} and the Lorenz gauge condition $\nabla_a A^a = 0$ such that 
\[ A|_{\mathbb{S}^n_0} = a, \quad F \cdot T |_{\mathbb{S}^n_0} = \mathbf{E}, \quad \phi|_{\mathbb{S}^n_0} = \phi_0, \quad \dot{\phi}|_{\mathbb{S}^n_0} = \phi_1 \]
if $ m >n/2$. The supremum of such numbers $\sigma > 0$ depends continuously on 
\[ M_1 = \| a \|_{H^m} + \|\phi_0\|_{H^m} + \| \phi_1 \|_{H^{m-1}} + \| \mathbf{E} \|_{H^{m-1}} \]
and tends to infinity as $M_1$ tends to zero. The solution $(A,\phi)$ is unique in $E_m(\mathbb{S}^n \times I_\sigma)$ up to gauge transformations preserving the Lorenz gauge.
\end{theorem}

\begin{corollary} \label{cor:wellposedness} Consider the system \eqref{phiAequationsCGprojected} on $\mathfrak{E} = \mathbb{S}^3 \times \mathbb{R}$ and suppose that for $m \geq 2$ we are given data $\mathbf{A}_0$, $\phi_0 \in H^m(\mathbb{S}^3_0)$ and $\mathbf{A}_1$, $\phi_1 \in H^{m-1}(\mathbb{S}^3_0)$ satisfying the strong Coulomb gauge initially. Then there exists an interval $I_\sigma = (-\sigma, \sigma) \subset \mathbb{R}$ and $(A_0,\mathbf{A},\phi) \in E_m(\mathbb{S}^3 \times I_\sigma)$ satisfying the system \eqref{phiAequationsCGprojected} and the strong Coulomb gauge conditions $\bar{A}_0 = 0$, $\slashgrad \cdot \mathbf{A} = 0$ such that
\[ \mathbf{A}|_{\mathbb{S}^3_0} = \mathbf{A}_0, \quad \dot{\mathbf{A}}|_{\mathbb{S}^3_0} = \mathbf{A}_1, \quad \phi|_{\mathbb{S}^3_0} = \phi_0, \quad \dot{\phi}|_{\mathbb{S}^3_0} = \phi_1.  \]
The supremum of such numbers $\sigma > 0$ depends continuously on 
\[ M_2 = \| a_0 \|_{H^m} + \|\mathbf{A}_0 \|_{H^m} + \| \mathbf{A}_1 \|_{H^{m-1}} + \| \phi_0 \|_{H^{m-1}} + \| \phi_1 \|_{H^{m-1}} \simeq \mathrm{S}_m[\phi, A](0)^{1/2}  \]
and tends to infinity as $M_2$ tends to zero, where $a_0$ is determined by $\phi_0$ and $\phi_1$ via the elliptic equation \eqref{ellipticequationa0} on $\mathbb{S}^3_0$. The solution $(A_0, \mathbf{A}, \phi)$ is unique in $E_m(\mathbb{S}^3 \times I_\sigma)$ up to gauge transformations preserving the strong Coulomb gauge\footnote{Recall that the gauge transformations preserving the strong Coulomb gauge are just the trivial ones $\chi = \mathrm{e}^{i \theta}$ for global constants $\theta \in \mathbb{R}$.}.
\end{corollary}

\begin{proof} Given admissible $\phi_0 \in H^m(\mathbb{S}^3_0)$ and $\phi_1 \in H^{m-1}(\mathbb{S}^3_0)$, the equation
\[ - \slashed{\Delta} a_0 + |\phi_0|^2 a_0 = - \operatorname{Im}(\bar{\phi}_0 \phi_1) \]
on $\mathbb{S}^3_0$ has a unique solution $a_0$ in $H^m$ which by \cref{strongCGconstraint} satisfies $\bar{a}_0 = 0$. We define $\mathbf{E} \defeq \slashgrad a_0 - \mathbf{A}_1$, which by construction satisfies \eqref{constraint}. We may thus apply \Cref{thm:ChoquetBruhatexistence}. Note that we do not prescribe $\dot{A}_0$, but instead construct it so that the Lorenz gauge condition is satisfied initially. The Lorenz gauge is then propagated by the equations \eqref{phiAequation1} in Lorenz gauge (but note that, of course, the strong Coulomb gauge is not). We thus have a solution $(A_a,\phi) \in E_m(\mathbb{S}^3 \times I_\sigma)$ of \eqref{phiAequation1} satisfying $\nabla_a A^a = 0$ throughout $\mathbb{S}^3 \times I_\sigma$. Now perform a gauge transformation as in \Cref{sec:strongCoulombgauge} to convert this solution to a solution $(A_0, \mathbf{A}, \phi) \in E_m(\mathbb{S}^3 \times I_{\sigma})$ of \eqref{phiAequationsCGprojected} satisfying the strong Coulomb gauge. It is easy to see that this gauge transformation preserves $E_m$ regularity, while uniqueness up to gauge transformations is also clear. As for the continuous dependence of $\sigma$ on the data, we note that
\begin{align*} M_1 &= \| a \|_{H^2} + \| \phi_0 \|_{H^2} + \| \phi_1 \|_{H^1} + \| \mathbf{E} \|_{H^1} \\ & \la \| a_0 \|_{H^2} + \| \mathbf{A}_0 \|_{H^2} + \| \phi_0 \|_{H^2} + \| \phi_1 \|_{H^1} + \| \slashgrad a_0 \|_{H^1} + \| \mathbf{A}_1 \|_{H^1} \\ & \la \| a_0 \|_{H^2} + \| \mathbf{A}_0 \|_{H^2} + \| \mathbf{A}_1 \|_{H^1} + \| \phi_0 \|_{H^2} + \| \phi_1 \|_{H^1}  = M_2,
\end{align*}
and similarly $M_2 \la M_1$. Thus $M_1 \simeq M_2$ and we are done.
\end{proof}

\section{Geometric and Sobolev Energies} \label{sec:energies}

\subsection{The Maxwell Sector}

For ease of presentation we treat the Maxwell and the scalar field sectors of the energy-momentum tensor $\mathbf{T}_{ab}$ separately. The energy-momentum tensor for the Maxwell sector in terms of the Maxwell field $F_{ab}$ on $\mathfrak{E}$ is 
\[ \mathbf{T}_{ab}[F] = - F_a^{\phantom{a}c} F_{bc} + \frac{1}{4} \mathfrak{e}_{ab} F_{cd} F^{cd}, \]
or in terms of the potential $A_a$
\begin{align*} \mathbf{T}_{ab}[A] = & - \nabla_a A^c \nabla_b A_c + \nabla^c A_a \nabla_b A_c + \nabla_a A^c \nabla_c A_b - \nabla^c A_a \nabla_c A_b \\ &+ \frac{1}{2} \mathfrak{e}_{ab} \left( \nabla_c A_d \nabla^c A^d - \nabla_c A_d \nabla^d A^c \right). \end{align*}
The Maxwell sector energy density with respect to the foliation $\mathcal{F}$ is given by the component
\begin{align*} \mathbf{T}_{00}[A] &= \mathbf{T}_{ab} T^a T^b \\
& = -\dot{A}^c \dot{A}_c + 2 \dot{A}_c \nabla^c A_0 - \nabla^c A_0 \nabla_c A_0 + \frac{1}{2} \left( \nabla_c A_d \nabla^c A^d - \nabla_c A_d \nabla^d A^c \right),	
\end{align*}
where in the above we have denoted by $A_0 \defeq T^a A_a$ and $\dot{A}_a \defeq T^b \nabla_b A_a$. Note that the metric $\mathfrak{e}$ splits as the direct sum $\mathfrak{e} = g_{\mathbb{R}} \oplus (- \mathfrak{s}_3)$, so in particular the full connection $\nabla$ also splits as $\nabla = \nabla^{\mathbb{R}} \oplus \nabla^{\mathfrak{s}_3} = \partial_\tau \oplus \slashgrad$. This can also be seen at the level of the Christoffel symbols on $\mathfrak{E}$ in \Cref{prop:Christoffels}. Furthermore, there is no curvature in the $\tau$ direction (see \Cref{prop:Riemanntensorcomponents}), so in particular $\partial_\tau$ commutes with the $3$-sphere derivatives, $[ \partial_\tau, \, \slashgrad] = 0$. We have
\begin{equation} \label{T00Maxwell} \mathbf{T}_{00}[A] = \frac{1}{2} | \dot{\mathbf{A}} |^2 + \frac{1}{2} | \slashgrad A_0 |^2 + \frac{1}{2} | \slashgrad \mathbf{A} |^2 - \dot{\mathbf{A}} \cdot \slashgrad A_0 - \frac{1}{2} (\slashgrad_\mu \mathbf{A}_\nu)( \slashgrad^\nu \mathbf{A}^\mu ).
\end{equation}
We impose the Coulomb gauge
\[ \slashgrad \cdot \mathbf{A} = 0 \]
on each $\mathbb{S}^3_\tau \simeq \mathbb{S}^3$ so that the last two terms become non-negative-definite upon integration by parts:
\begin{align*} \int_{\mathbb{S}^3} - \dot{\mathbf{A}} \cdot \slashgrad A_0 \, \dvol_{\mathfrak{s}_3} = \int_{\mathbb{S}^3} A_0 \slashgrad \cdot \dot{\mathbf{A}} \, \dvol_{\mathfrak{s}_3} = 0,
\end{align*}
and
\begin{align*} \int_{\mathbb{S}^3} - \frac{1}{2} (\slashgrad_\mu \mathbf{A}_\nu) (\slashgrad^\nu \mathbf{A}^\mu ) \, \dvol_{\mathfrak{s}_3} &= \int_{\mathbb{S}^3} \left( \frac{1}{2} \mathbf{A}^\mu \slashgrad_\mu \slashgrad_\nu \mathbf{A}^\nu - \frac{1}{2} R({\mathfrak{s}_3)}_{\mu \nu} \mathbf{A}^\mu \mathbf{A}^\nu \right) \dvol_{\mathfrak{s}_3} \\
&= \int_{\mathbb{S}^3} | \mathbf{A} |^2 \, \dvol_{\mathfrak{s}_3}.
\end{align*}
Thus the Maxwell energy on surfaces of constant $\tau$ is
\begin{align*} \mathcal{E}_\tau[A] & \defeq \int_{\mathbb{S}^3} \mathbf{T}_{00}[A] \, \dvol_{\mathfrak{s}_3} (\tau) \\  
& \simeq \| \mathbf{A} \|^2_{H^1}(\tau) + \| \dot{\mathbf{A}} \|^2_{L^2}(\tau) + \| \slashgrad A_0 \|^2_{L^2}(\tau) \\
& = \mathrm{S}_1[\mathbf{A}](\tau) + \| \slashgrad A_0 \|^2_{L^2}(\tau).
\end{align*}
Imposing the additional condition $\bar{A}_0(\tau) = 0$, one has that $\| A_0 \|_{L^2(\mathbb{S}^3)}^2 \la \| \slashgrad A_0 \|^2_{L^2(\mathbb{S}^3)}$, so
\begin{equation} \label{Maxwellphysicalsobolevequivalence1} \mathcal{E}_\tau[A] \simeq \mathrm{S}_1[A](\tau) \end{equation}
for all $\tau \in \mathbb{R}$.

\subsubsection{Higher Order Energies}

More generally, for a $1$-form $\alpha$ set 
\begin{align*} \mathbf{T}_{ab}[\alpha] & \defeq - \nabla_a \alpha^c \nabla_b \alpha_c + \nabla^c \alpha_a \nabla_b \alpha_c + \nabla_a \alpha^c \nabla_c \alpha_b - \nabla^c \alpha_a \nabla_c \alpha_b \\
& + \frac{1}{2} \mathfrak{e}_{ab} \left( \nabla_c \alpha_d \nabla^c \alpha^d - \nabla_c \alpha_d \nabla^d \alpha^c \right).
\end{align*}
When $\alpha_a = A_a$, this is, of course, just the Maxwell energy-momentum tensor written out in terms of the potential. As in \cref{T00Maxwell}, we have
\[ \mathbf{T}_{00}[\alpha] = \frac{1}{2} | \dot{\boldsymbol{\alpha}} |^2 + \frac{1}{2} | \slashgrad \alpha_0 |^2 + \frac{1}{2} | \slashgrad \boldsymbol{\alpha} |^2 - \dot{\boldsymbol{\alpha}}_\mu \slashgrad^\mu \alpha_0 - \frac{1}{2}( \slashgrad_\mu \boldsymbol{\alpha}_\nu )( \slashgrad^\nu \boldsymbol{\alpha}^\mu ).  \]
Integrating by parts as before we obtain
\begin{align*} \mathcal{E}_\tau[\alpha]  \defeq &\int_{\mathbb{S}^3} \mathbf{T}_{00}[\alpha] \, \dvol_{\mathfrak{s}_3} \\  =& \frac{1}{2} \int_{\mathbb{S}^3} | \dot{\boldsymbol{\alpha}} |^2 \, \dvol_{\mathfrak{s}_3} + \frac{1}{2} \int_{\mathbb{S}^3} | \slashgrad \alpha_0 |^2 \, \dvol_{\mathfrak{s}_3} + \frac{1}{2} \int_{\mathbb{S}^3} | \slashgrad \boldsymbol{\alpha} |^2 \, \dvol_{\mathfrak{s}_3} \\ &+ \int_{\mathbb{S}^3} \alpha_0 \slashgrad_\mu \dot{\boldsymbol{\alpha}}^\mu \, \dvol_{\mathfrak{s}_3} - \frac{1}{2} \int_{\mathbb{S}^3} | \slashgrad \cdot \boldsymbol{\alpha} |^2 \, \dvol_{\mathfrak{s}_3} + \int_{\mathbb{S}^3} | \boldsymbol{\alpha} |^2 \, \dvol_{\mathfrak{s}_3}.  \end{align*}
For our second order estimates we will want to set $\alpha_a = X_i^\mu \slashgrad_\mu A_a \defeq \slashgrad_i A_a$ and sum over $i$ for a basis of vector fields $\{ X_i \}_i$ on $\mathbb{S}^3$ (e.g. a basis of left-invariant vector fields on $\mathbb{S}^3 \simeq \mathrm{SU}(2)$). The first term in the above is then clearly
\[ \sum_i | \dot{\boldsymbol{\alpha}} |^2 = \sum_i \slashgrad_i \dot{\mathbf{A}}_\mu \slashgrad_i \dot{\mathbf{A}}^\mu = | \slashgrad \dot{\mathbf{A}} |^2, \]
the second term becomes
\[ \sum_i | \slashgrad \alpha_0 |^2 = \sum_i \slashgrad_\mu \slashgrad_i A_0 \slashgrad^\mu \slashgrad_i A_0 = | \slashgrad^2 A_0 |^2 + \text{l.o.t.s}, \]
the third term becomes
\[ \sum_i |\slashgrad \boldsymbol{\alpha} |^2 = \sum_i \slashgrad_\mu \slashgrad_i \mathbf{A}_\nu \slashgrad^\mu \slashgrad_i \mathbf{A}^\nu = | \slashgrad^2 \mathbf{A} |^2 + \text{l.o.t.s}, \]
the fourth term, after commuting derivatives to impose the Coulomb gauge $\slashgrad \cdot \mathbf{A} = 0$, is
\[ \sum_i \alpha_0 \slashgrad_\mu \dot{\boldsymbol{\alpha}}^\mu = \sum_i \slashgrad_i A_0 \slashgrad_\mu \slashgrad_i \dot{\mathbf{A}}^\mu = \text{l.o.t.s}, \]
and the fifth term similarly becomes
\[ \sum_i | \slashgrad \cdot \boldsymbol{\alpha} |^2 =  \sum_i \slashgrad_\mu \slashgrad_i \mathbf{A}^\mu \slashgrad_\nu \slashgrad_i \slashed{\mathbf{A}}^\nu = \text{l.o.t.s}, \]
where in the above we have written $\slashgrad_j \defeq X_j^\mu \slashgrad_\mu$, and the lower order terms are at most quadratic and of order zero and one in derivatives of $A_a$. The sixth and final term is
\[ \sum_i |\boldsymbol{\alpha} |^2 = \sum_i | \slashgrad_i \mathbf{A} |^2 = \text{l.o.t.s}. \]
The lower order terms can be controlled by $\mathcal{E}_\tau[A] \simeq \mathrm{S}_1[A](\tau)$, so we can find a constant $C>0$ large enough such that
\begin{align*} \mathcal{E}_{\tau}[A] + \sum_i \mathcal{E}_\tau[\slashgrad_i A] &\simeq C \mathcal{E}_\tau[A] + \sum_i \mathcal{E}_\tau[\slashgrad_i A] \\ 
& \simeq \| \mathbf{A} \|^2_{H^2}(\tau) + \| \dot{\mathbf{A}} \|^2_{H^1}(\tau) + \| \slashgrad A_0 \|^2_{H^1}(\tau) \\
&= \mathrm{S}_2[\mathbf{A}](\tau) + \| \slashgrad A_0 \|^2_{H^1}(\tau).
\end{align*}
As before, the strong Coulomb gauge implies $\| A_0 \|_{L^2} \la \| \slashgrad A_0 \|_{L^2}$, and so
\begin{equation} \label{Sobolev2equivalenceA} \mathcal{E}_\tau[A] + \sum_i \mathcal{E}_\tau[\slashgrad_i A] \simeq \mathrm{S}_2[A](\tau).
\end{equation}
Similarly, it is easy to see that the strong Coulomb gauge gives
\[ \sum_{k=0}^{m-1} \mathcal{E}_\tau [ \slashgrad^k A ] \simeq \mathrm{S}_m [A] (\tau), \]
where $\mathcal{E}_\tau [ \slashgrad^k A]$ denotes $\sum_{i_1, \dots, i_k } \mathcal{E}_\tau [ \slashgrad_{i_1} \dots \slashgrad_{i_k} A ]$.

\subsection{The Scalar Field Sector}

The energy-momentum tensor for the scalar field sector on $\mathfrak{E}$ is
\[ \mathbf{T}_{ab}[\phi] = \overline{\D_{(a} \phi} \D_{b)} \phi - \frac{1}{2} \mathfrak{e}_{ab} \overline{\D_c \phi} \D^c \phi + \frac{1}{2} \mathfrak{e}_{ab} | \phi |^2, \]
and we calculate
\begin{align*} \mathbf{T}_{00}[\phi] =& |\D_0 \phi |^2 - \frac{1}{2} \overline{\D_c \phi} \D^c \phi + \frac{1}{2} | \phi |^2 \\ = & \frac{1}{2} |\D_0 \phi |^2 + \frac{1}{2} \overline{\slashed{\D}_\mu \phi} \slashed{\D}^\mu \phi + \frac{1}{2} | \phi |^2 
\end{align*}
and
\[ \mathcal{E}_\tau[\phi] = \frac{1}{2} \| \D_0 \phi \|^2_{L^2}(\tau) + \frac{1}{2} \| \slashed{\D} \phi \|^2_{L^2}(\tau) + \frac{1}{2} \| \phi \|^2_{L^2}(\tau), \]
where $\D_0 \phi = \dot{\phi} + iA_0 \phi$ and $\slashed{\D}_\mu = \slashgrad_\mu + i \mathbf{A}_\mu$.
More generally, we set 
\[ \mathbf{T}_{ab}[\psi] = \overline{\D_{(a} \psi} \D_{b)} \psi - \frac{1}{2} \mathfrak{e}_{ab} \overline{\D_c \psi} \D^c \psi + \frac{1}{2} \mathfrak{e}_{ab} | \psi |^2 \]
and
\[ \mathcal{E}_\tau[\psi] = \frac{1}{2} \| \D_0 \psi \|^2_{L^2}(\tau) + \frac{1}{2} \| \slashed{\D} \psi \|^2_{L^2}(\tau) + \frac{1}{2} \| \psi \|^2_{L^2}(\tau) \]
for any complex scalar field $\psi$ on $\mathfrak{E}$.
As with the Maxwell sector, we will want to choose $\psi = \slashgrad_i \phi$ for our second order estimates.

\subsubsection{Conversion Between Geometric and Sobolev Energies}

\begin{proposition} \label{prop:gradslashphibound} For any fixed $\tau \in \mathbb{R}$ and any sufficiently smooth complex scalar field $\psi$ on $\mathfrak{E}$ there exists $\epsilon > 0$ small enough such that if $\mathrm{S}_1[\mathbf{A}](\tau) \leq \epsilon$, then
\[ \| \slashgrad \psi \|^2_{L^2}(\tau) \la \mathcal{E}_\tau[\psi]. \]
\end{proposition}

\begin{proof} We suppress the $\tau$ variable. Clearly 
\begin{align*} \| \slashgrad \psi \|^2_{L^2} \la \| \slashed{\D} \psi \|^2_{L^2} + \| \mathbf{A} \psi \|^2_{L^2} \la \mathcal{E}[\psi] + \| \mathbf{A} \|^2_{L^6} \| \psi \|^2_{L^3}. \end{align*}
Now since $\mathbb{S}^3$ is compact, $\| \psi \|_{L^3} \la \| \psi \|_{L^6}$, and by Sobolev Embedding (\Cref{thm:sobolevembedding})
\[ \| \mathbf{A} \|^2_{L^6} \la \| \slashgrad \mathbf{A} \|^2_{L^2} + \| \mathbf{A} \|^2_{L^2} \la \mathrm{S}_1[\mathbf{A}] \]
and
\[ \| \psi \|^2_{L^6} \la \| \slashgrad \psi \|^2_{L^2} + \| \psi \|^2_{L^2} \la \| \slashgrad \psi \|^2_{L^2} + \mathcal{E}[\psi]. \]
This gives
\[ \| \slashgrad \psi \|^2_{L^2} \leq C (1+\mathrm{S}_1[\mathbf{A}] )  \mathcal{E}[\psi] + C \mathrm{S}_1[\mathbf{A}] \| \slashgrad \psi \|^2_{L^2} \leq C \epsilon \| \slashgrad \psi \|^2_{L^2} + C (1+ \epsilon) \mathcal{E}[\psi], \]
so
\[ \| \slashgrad \psi \|^2_{L^2} \leq C \left( \frac{1 + \epsilon}{1 - \epsilon C} \right) \mathcal{E}[\psi] \la \mathcal{E}[\psi] \]
for $\epsilon > 0$ small enough.
\end{proof}

\begin{proposition} \label{prop:phidotbound} For any fixed $\tau \in \mathbb{R}$ and any sufficiently smooth complex scalar field $\psi$ on $\mathfrak{E}$ there exists $\epsilon > 0$ such that if $\mathrm{S}_1[\mathbf{A}](\tau) \leq \epsilon$, then 
\[ \| \dot{\psi} \|^2_{L^2}(\tau) \la (1+\mathrm{S}_1[A_0](\tau)) \, \mathcal{E}_\tau[\psi]. \]
\end{proposition}

\begin{proof} Working similarly to the previous proposition,
\begin{align*} \| \dot{\psi} \|^2_{L^2} &\la \| \D_0 \psi \|^2_{L^2} + \| A_0 \psi \|^2_{L^2} \la \mathcal{E}[\psi] + \| A_0 \|^2_{L^6} \| \psi \|^2_{L^3} \la \mathcal{E}[\psi] + \| A_0 \|^2_{H^1} \|\psi\|^2_{L^6}.
\end{align*}
Also $\| \psi \|^2_{L^6} \la \| \slashgrad \psi \|^2_{L^2} + \| \psi \|^2_{L^2} $, so
\begin{align*} \| \dot{\psi} \|^2_{L^2} & \la \mathcal{E}[\psi] + \mathrm{S}_1[A_0] \left( \| \slashgrad \psi \|^2_{L^2} + \| \psi \|^2_{L^2} \right) \la \left( 1 + \mathrm{S}_1[A_0] \right) \mathcal{E}[\psi] + \mathrm{S}_1[A_0] \| \slashgrad \psi \|^2_{L^2}.
\end{align*}
\Cref{prop:gradslashphibound} now gives the result for small $\mathrm{S}_1[\mathbf{A}]$.
\end{proof}

\begin{proposition} \label{prop:Dslashphismall} For any fixed $\tau \in \mathbb{R}$ and any sufficiently smooth complex scalar field $\psi$ on $\mathfrak{E}$ one has
\[ \| \slashed{\D} \psi \|^2_{L^2}(\tau) \la  \mathrm{S}_1[\psi](\tau)( 1 + \mathrm{S}_1[\mathbf{A}](\tau) ). \]
\end{proposition}

\begin{proof} This is a simple consequence of the compactness of $\mathbb{S}^3$ and the Sobolev Embedding Theorem as above,
\begin{align*} \| \slashed{\D} \psi \|^2_{L^2} \la \| \slashgrad \psi \|^2_{L^2} + \| \mathbf{A} \psi \|^2_{L^2} \la \| \slashgrad \psi \|^2_{L^2} + \| \mathbf{A} \|^2_{L^6} \| \psi \|^2_{L^6} \la  \mathrm{S}_1[\psi] + \mathrm{S}_1[\mathbf{A}] \mathrm{S}_1[\psi] .
\end{align*}
\end{proof}

\begin{proposition} \label{prop:D0phismall} For any fixed $\tau \in \mathbb{R}$ and any sufficiently smooth complex scalar field $\psi$ on $\mathfrak{E}$ one has 
\[ \| \D_0 \psi \|^2_{L^2}(\tau) \la (1+\mathrm{S}_1[A_0](\tau)) \mathrm{S}_1[\psi](\tau). \]
\end{proposition}

\begin{proof} This follows from the same splitting and embedding as the previous propositions,
\[ \| \D_0 \psi \|^2_{L^2} \la \| \dot{\psi} \|^2_{L^2} + \| A_0 \psi \|^2_{L^2} \la \mathrm{S}_1[\psi] (1 + \| A_0 \|^2_{H^1} ) \la (1+\mathrm{S}_1[A_0]) \mathrm{S}_1[\psi]. \]
\end{proof}

\begin{theorem} \label{thm:scalarsobolevandphysicalequivalence} For any fixed $\tau \in \mathbb{R}$ and any sufficiently smooth complex scalar field $\psi$ on $\mathfrak{E}$ there exists $\epsilon > 0$ such that if $\mathrm{S}_1[A] \leq \epsilon$, then
\[ \mathrm{S}_1[\psi](\tau) \simeq \mathcal{E}[\psi](\tau). \]
\end{theorem}

\begin{proof} Suppose $\mathrm{S}_1[A]$ is small. Then in particular both $\mathrm{S}_1[\mathbf{A}]$ and $\mathrm{S}_1[A_0]$ are small, so by \Cref{prop:gradslashphibound} $\| \slashgrad \psi \|^2_{L^2} \la \mathcal{E}[\psi]$. By \Cref{prop:phidotbound}, $\| \dot{\psi} \|^2_{L^2} \la \mathcal{E}[\psi]$, so
\[ \mathrm{S}_1[\psi] \la \mathcal{E}[\psi]. \]
Conversely, by \Cref{prop:Dslashphismall,prop:D0phismall}, $\| \slashed{\D} \psi \|^2_{L^2} \la \mathrm{S}_1[\psi]$ and $\| \D_0 \psi \|^2_{L^2} \la \mathrm{S}_1[\psi]$, so
\[ \mathcal{E}[\psi] \la \mathrm{S}_1[\psi]. \]
\end{proof}

\noindent In particular, $\mathcal{E}[\phi] \simeq \mathrm{S}_1[\phi]$ and $\mathcal{E}[\slashgrad \phi] \simeq \mathrm{S}_1[\slashgrad \phi]$. Since $\mathrm{S}_1[\phi] + \mathrm{S}_1[\slashgrad \phi] \simeq \mathrm{S}_2[\phi]$, one then has
\begin{equation} \label{Sobolev2equivalencephi} \mathcal{E}_\tau[\phi] + \mathcal{E}_\tau[\slashgrad \phi] \simeq \mathrm{S}_2[\phi](\tau) \end{equation}
if $\mathrm{S}_1[A](\tau)$ is sufficiently small. Similarly,
\[ \sum_{k=0}^{m-1} \mathcal{E}_\tau[\slashgrad^k \phi] \simeq \mathrm{S}_m[\phi](\tau) \]
if $\mathrm{S}_1[A](\tau)$ is sufficiently small.

\subsection{Elliptic Estimates}

As we have already seen, one useful feature of the Coulomb gauge is that the field equation for $A_0$ becomes elliptic,
\begin{equation} \label{ellipticequationCG} - \slashed{\Delta} A_0 + |\phi|^2 A_0 = - \operatorname{Im}( \bar{\phi} \dot{\phi} ). \end{equation}
Even though the component $A_0$ is non-dynamical, it still carries energy. This energy is controlled by $\dot{\phi}$ as follows.

\begin{proposition} \label{prop:A0estimates} The non-dynamical component $A_0$ satisfies the estimates
\[ \| \slashgrad A_0 \|^2_{L^2}(\tau) + \| \phi A_0 \|^2_{L^2}(\tau) + \| A_0 \|^2_{L^2}(\tau) \la \| \dot{\phi} \|^2_{L^2}(\tau) \]
for every fixed $\tau \in \mathbb{R}$.
\end{proposition} 

\begin{proof} Multiplying equation \eqref{ellipticequationCG} by $A_0$ and integrating, we have
\[ \| \slashgrad A_0 \|^2_{L^2} + \| \phi A_0 \|^2_{L^2} = - \int_{\mathbb{S}^3} \Im(\bar{\phi} \dot{\phi}) A_0 \dvol_{\mathfrak{s}_3} \leq \| \phi A_0 \|_{L^2} \| \dot{\phi} \|_{L^2} \leq \frac{1}{2} \| \phi A_0 \|^2_{L^2} + \frac{1}{2} \| \dot{\phi} \|^2_{L^2}, \]
which gives the first two estimates. The third estimate follows from the Poincar\'e inequality for $A_0$.
\end{proof}

We will need these estimates to extend energy smallness assumptions on $\mathbf{A}$ and $\phi$ to $A_0$.

\section{$H^1$ and $H^2$ Energy Estimates} \label{sec:energyestimates}

\subsection{Conservation of Energy}

For general $\alpha$, $\psi$ one finds that
\begin{align} \begin{split} \label{generalerrorterms} & \nabla^a \mathbf{T}_{ab}[\alpha] = \mathfrak{M}(\alpha)^a \left( \nabla_a \alpha_b - \nabla_b \alpha_a \right), \\ & \nabla^a \mathbf{T}_{ab}[\psi] = \frac{1}{2} \overline{\mathfrak{S}(\psi)} \D_b \psi + \frac{1}{2} \mathfrak{S}(\psi) \overline{\D_b \psi} +  F_{ab} \operatorname{Im}\left( \bar{\psi} \D^a \psi \right). \end{split}
\end{align}
When $\alpha_a = A_a$ and $\psi = \phi$, the field equations $\mathfrak{M}(A)_a = - \operatorname{Im}\left( \bar{\phi} \D_a \phi \right)$ and $\mathfrak{S}(\phi) = 0$ imply that
\[ \nabla^a \mathbf{T}_{ab}[\phi,A] = \nabla^a ( \mathbf{T}_{ab}[A] + \mathbf{T}_{ab}[\phi]) = F_{ab} \left( \operatorname{Im}\left( \bar{\phi} \D^a \phi \right) - \operatorname{Im}\left( \bar{\phi} \D^a \phi \right) \right) = 0. \]

\subsection{\texorpdfstring{$H^1$}{H1} estimates}

Consider admissible initial data for the system \eqref{phiAequationsCGprojected}. We can make no a priori assumptions about the smallness of the non-dynamical component $A_0$, but we will of course be able to extract all the required information about $A_0$ using the elliptic equation \eqref{ellipticequationCG}.

\begin{theorem} \label{thm:H1estimates} There exists an $\epsilon > 0$ such that if $\mathrm{S}_1[\phi, \mathbf{A}](0) \leq \epsilon$, then
\[ \mathrm{S}_1[\phi, A](\tau) \simeq \mathrm{S}_1[\phi, A](0) \]
for all $\tau \in \mathbb{R}$.
\end{theorem}

\begin{proof} Since $\nabla^a \mathbf{T}_{ab} [\phi, A] = 0$ and $T^b = \partial_\tau$ is Killing on $\mathfrak{E}$, integrating $\mathsf{e}_1 \defeq \nabla^a ( T^b \mathbf{T}_{ab}[\phi, A]) = 0$ over the region $\mathbb{S}^3 \times [0,\tau]$ for any $\tau > 0$ immediately gives
\[ 0 = \int_{\mathbb{S}^3 \times [0, \tau]} \mathsf{e}_1 \dvol = \int_{\mathbb{S}^3_\tau } \mathbf{T}_{00}[\phi,A] \dvol_{\mathfrak{s}_3} - \int_{\mathbb{S}^3_0 } \mathbf{T}_{00}[\phi,A] \dvol_{\mathfrak{s}_3},  \]
i.e. 
\begin{equation} \label{conservationofenergy} \mathcal{E}_\tau[\phi] + \mathcal{E}_\tau[A] = \mathcal{E}_{\tau}[\phi,A] = \mathcal{E}_0[\phi, A] = \mathcal{E}_0[\phi] + \mathcal{E}_0[A]. \end{equation}
Now the smallness assumption $\mathrm{S}_1[\phi, \mathbf{A}](0) \leq \epsilon$ implies that $\mathrm{S}_1[\mathbf{A}](0) \leq \epsilon$ and  $\mathrm{S}_1[\phi](0) \leq \epsilon$, so by \Cref{prop:A0estimates}
\[ \| \slashgrad A_0 \|^2_{L^2}(0) \la \mathrm{S}_1[\phi](0) \leq \epsilon, \]
and so $\mathrm{S}_1[A](0) \la \epsilon$. Then by \Cref{thm:scalarsobolevandphysicalequivalence}, $\mathcal{E}_0[\phi] \simeq \mathrm{S}_1[\phi](0)$. Now equation \eqref{Maxwellphysicalsobolevequivalence1} reads $\mathcal{E}_\tau[A] \simeq \mathrm{S}_1[A](\tau), $ which in particular holds at $\tau = 0$, so we have $\mathcal{E}_0[\phi] + \mathcal{E}_0[A] \simeq \mathrm{S}_1[\phi](0) + \mathrm{S}_1[A](0)$, and so by \cref{conservationofenergy} 
\[ \mathcal{E}_\tau[\phi] + \mathcal{E}_\tau[A] \simeq \mathrm{S}_1[\phi](0) + \mathrm{S}_1[A](0). \]
This means that $\mathcal{E}_\tau[\phi] + \mathcal{E}_\tau[A]$ is small too, $\mathcal{E}_\tau[\phi, A] \la \epsilon$. In particular, $\mathcal{E}_\tau[A] \simeq \mathrm{S}_1[A](\tau)$ is small, so again by \Cref{thm:scalarsobolevandphysicalequivalence}, $\mathcal{E}_\tau[\phi] \simeq \mathrm{S}_1[\phi](\tau)$. We deduce that
\begin{equation} \label{H1estimate} \mathrm{S}_1[\phi](\tau) + \mathrm{S}_1[A](\tau) \simeq \mathrm{S}_1[\phi](0) + \mathrm{S}_1[A](0) \end{equation}
for all $\tau > 0$. The same argument works for $\tau < 0$.

\end{proof}

\subsection{\texorpdfstring{$H^2$}{H2} estimates}

\subsubsection{A Nonlinear Gr\"onwall Inequality}

Some useful small data nonlinear Gr\"onwall inequalities may be proved by reduction to the standard Gr\"onwall inequality using a careful change of variables. More precisely, suppose $g(\tau)$ satisfies a nonlinear differential inequality, say
\[ g'(\tau) \leq F\left(g(\tau)\right). \]
If we can find a function $G$ such that 
\[ G(g(\tau))' = G'(g(\tau)) g'(\tau) \overset{!}{\leq} G(g(\tau)), \]
then we can apply the standard Gr\"onwall inequality to $G(\tau) \defeq G(g(\tau))$ and possibly invert $G(g)$ to recover an inequality for $g$. This will not in general produce an immediately useful statement due to the nonlinear nature of $F$ (and hence $G$), but with a smallness assumption on $g(0)$ the offending terms can frequently be dealt with. Clearly finding such a $G$ amounts to solving the differential inequality
\[ G'(g) F(g) \leq G(g). \]

\begin{lemma} \label{lem:nonlinearGronwall} Let $\tau \in [0,1]$ and $f:[0,1] \to \mathbb{R}$ be a continuous non-negative function. Suppose $f$ satisfies the inequality
\[ f(\tau) \leq f(0) + \int_0^\tau f(\sigma) P ( f(\sigma)^{1/2} ) \, \d \sigma \]
for some polynomial $P$ with positive coefficients. Then there exists $\epsilon > 0$ small enough such that if $f(0) \leq \epsilon$, then 
\[ f(\tau) \leq C f(0) \]
for some $C>1$ and all $\tau \in [0,1]$.
\end{lemma}

\begin{proof} The case when $P$ has order zero is trivial, so assume that $P (x) = \sum_{k=0}^d P_k x^k$ for some $d>0$ and some non-negative real numbers $\{ P_k \}_k$. We may reduce the inequality as follows,
\begin{align*} f(\tau) & \leq f(0) + \int_0^\tau f(\sigma) P(f(\sigma)^{1/2}) \, \d \sigma \\
& \leq f(0) + \int_0^\tau \sum_{k=0}^d P_k f(\sigma)^{k/2 + 1} \, \d \sigma \\
& \leq f(0)   + \int_{\{ 0 < \sigma < \tau \, : \, f(\sigma) < 1 \}} \sum_{k=0}^d P_k f(\sigma)^{k/2 + 1} \, \d \sigma \\
& + \int_{ \{ 0 < \sigma < \tau \, : \, f(\sigma) >1 \} } \sum_{k=0}^d P_k f(\sigma)^{k/2 + 1} \, \d \sigma \\
& \leq f(0) + \int_0^\tau \sum_{k=0}^d P_k f(\sigma) \, \d \sigma + \int_0^\tau \sum_{k=0}^d P_k f(\sigma)^{d/2 + 1} \, \d \sigma \\
& \leq f(0) + \int_0^\tau D f(\sigma) \, \d \sigma + \int_0^\tau D f(\sigma)^{d/2 + 1} \, \d \sigma, 
\end{align*}
where $D = (d+1) \max_k P_k$. Now set 
\[ g(\tau) \defeq f(0) + \int_0^\tau D f(\sigma) \, \d \sigma + \int_0^\tau D f(\sigma)^{d/2 +1} \, \d \sigma.  \]
Then $f(\tau) \leq g(\tau)$, $f(0) = g(0)$, and $g'(\tau) \leq D f(\tau) + D f(\tau)^{d/2 + 1} \leq D g(\tau) \left(1 + g(\tau)^{d/2}\right)$. Defining
\[ G(\tau) \defeq g(\tau)^{1/D} D^{-2/(Dd)} \left( 1+ g(\tau)^{d/2} \right)^{-2/(Dd)} \]
and differentiating, one obtains
\begin{align*} G'(\tau) & = g'(\tau) g(\tau)^{1/D -1} D^{-2/(Dd) - 1} \left( 1 + g(\tau)^{d/2} \right)^{-2/(Dd) - 1} \\
& \leq g(\tau)^{1/D} D^{-2(Dd)} \left( 1 + g(\tau)^{d/2} \right)^{-2/(Dd)},	
\end{align*}
so that $G'(\tau) \leq G(\tau)$. Since $\tau$ is contained in a compact interval, this gives $G(\tau) \la G(0)$, or equivalently
\begin{align*} g(\tau)^{1/D} \left( 1 + g(\tau)^{d/2} \right)^{-2/(Dd)} & \la g(0)^{1/D} \left( 1 + g(0)^{d/2} \right)^{-2/(Dd)} \\
& \la g(0)^{1/D}.
\end{align*}
Rearranging gives
\[ g(\tau)^{d/2} \la g(0)^{d/2} \left( 1 + g(\tau)^{d/2} \right), \]
so if $g(0) = f(0)$ is small enough one has $g(\tau)^{d/2} \la g(0)^{d/2}$ and so
\[ f(\tau) \leq g(\tau) \leq C g(0) \leq C f(0). \]
\end{proof}

\begin{remark} Clearly the above proof goes through exactly the same with $[0,1]$ replaced with any interval $[0,r]$, $r \in \mathbb{R}$.
\end{remark}

\subsubsection{Commutators}

\begin{proposition} One has the following bounds on the commutators of $\slashgrad$ with the field equation operators $\mathfrak{M}$ and $\mathfrak{S}$:
\[ \left| [ \slashgrad, \, \mathfrak{M}]A \right|_{\mathbb{S}^3} \la |\slashgrad^2 \mathbf{A} | + | \slashgrad \mathbf{A} | + | \slashgrad \dot{A}_0 |, \]
and
\begin{align*} \left| [ \slashgrad, \, \mathfrak{S} ] (\phi) \right| & \la |\dot{\phi} \slashgrad A_0 | + |\phi \slashgrad \dot{A}_0 | + |\phi A_0 \slashgrad A_0 | + |\slashgrad^2 \phi | + |\phi \slashgrad^2 \mathbf{A} | \\
& + | \mathbf{A} \slashgrad \phi | + | \slashgrad \phi | + | \phi \slashgrad \mathbf{A} | + | \mathbf{A} \phi | + | \slashgrad \phi \slashgrad \mathbf{A} | + | \phi \mathbf{A} \slashgrad \mathbf{A} |.
\end{align*}

\end{proposition}

\begin{proof} Note that in the following the index $i$ always refers to a contraction with a basis vector field $X_i$. Recall that the operator $\mathfrak{M}_\mu$ on $A$ is given by $\mathfrak{M}(A)_\mu = \Box \mathbf{A}_\mu - \slashgrad_\mu \dot{A}_0 - 2 \mathbf{A}_\mu$, so for any $i$
\begin{align*} |[ \slashgrad_i, \, \mathfrak{M}](A)|_{\mathbb{S}^3} & = | \slashgrad_i \mathfrak{M}(A)_\mu - \mathfrak{M}(\slashgrad_i A)_\mu | \\ & = \left| \slashgrad_i \left( \Box \mathbf{A}_\mu - \slashgrad_\mu \dot{A}_0 - 2 \mathbf{A}_\mu \right) - \Box ( \slashgrad_i \mathbf{A}_\mu ) + \slashgrad_\mu \slashgrad_i \dot{A}_0 + 2 \slashgrad_i \mathbf{A}_\mu  \right| \\ & = \left| \slashgrad_i \slashgrad^\nu \slashgrad_\nu \mathbf{A}_\mu - \slashgrad^\nu \slashgrad_\nu ( X_i^\lambda \slashgrad_\lambda \mathbf{A}_\mu ) + \slashgrad_\mu X_i^\nu \slashgrad_\nu \dot{A}_0 \right| \\ & \leq C \left[ | \slashgrad^2 \mathbf{A} | + | \slashgrad \mathbf{A} | + | \slashgrad \dot{A}_0 | \right], 
\end{align*}
where the constant $C$ depends on the geometry of $\mathbb{S}^3$. To calculate the other commutator we need a couple of preliminary formulae. Let $\psi$ be any sufficiently regular complex scalar field. Then
\[ [ \slashgrad_i, \, \D_0 ] (\psi) = \slashgrad_i( \dot{\psi} + i A_0 \psi ) - \D_0 \slashgrad_i \psi = i \psi \slashgrad_i A_0, \]
and similarly
\[ [\slashgrad_i, \, \slashed{\D}_\mu](\psi) = - (\slashgrad_\mu X_i^\nu ) \slashgrad_\nu \psi + i \psi \slashgrad_i \mathbf{A}_\mu, \]
so 
\begin{align*} [\slashgrad_i, \, \D_0 \D_0](\phi) & = \D_0 [\slashgrad_i, \, \D_0 ] (\phi) + [ \slashgrad_i, \, \D_0] (\D_0 \phi) \\ &= \D_0 ( i \phi \slashgrad_i A_0) + i \D_0 \phi \slashgrad_i A_0 \\ & = i \phi \slashgrad_i \dot{A}_0 + 2 i \dot{\phi} \slashgrad_i A_0 - 2 \phi A_0 \slashgrad_i A_0.
\end{align*}
Further, for any vector field $\mathbf{V}$ on $\mathbb{S}^3$
\begin{align*} [\slashgrad_i, \, \slashed{\D}_\mu] \mathbf{V}^\mu & = \slashgrad_i (\slashgrad_\mu \mathbf{V}^\mu + i \mathbf{A}_\mu \mathbf{V}^\mu ) - (\slashgrad_\mu + i \mathbf{A}_\mu ) ( \slashgrad_i \mathbf{V}^\mu ) \\ &= \slashgrad_i \slashgrad_\mu \mathbf{V}^\mu - \slashgrad_\mu \slashgrad_i \mathbf{V}^\mu + i (\slashgrad_i \mathbf{A}_\mu) \mathbf{V}^\mu \\ & \leq C \left[ | \slashgrad \mathbf{V} | + | \mathbf{V} | + | \mathbf{V} \slashgrad \mathbf{A} |  \right], 
\end{align*}
where, as before, $C$ depends on the geometry of $\mathbb{S}^3$. Then
\begin{align*} [\slashgrad_i, \, \slashed{\D}_\mu \slashed{\D}^\mu]\phi & = \slashed{\D}^\mu [ \slashgrad_i, \, \slashed{\D}_\mu]\phi + [\slashgrad_i, \, \slashed{\D}_\mu] \slashed{\D}^\mu \phi \\ 
& \leq \slashed{\D}^\mu \left( - \slashgrad_\mu X_i^\nu \slashgrad_\nu \phi + i \phi \slashgrad_i \mathbf{A}_\mu \right) + C \left[ | \slashgrad \slashed{\D} \phi | + | \slashed{\D} \phi | + | \slashed{\D} \phi \slashgrad \mathbf{A} | \right] \\ 
& \leq - \slashed{\Delta} X_i^\nu \slashgrad_\nu \phi - \slashgrad_\mu X_i^\nu \slashgrad^\mu \slashgrad_\nu \phi + i \slashgrad^\mu \phi \slashgrad_i \mathbf{A}_\mu + i \phi \slashgrad_\mu \slashgrad_i \mathbf{A}_\mu \\
& - i \mathbf{A}^\mu \slashgrad_\mu X_i^\nu \slashgrad_\nu \phi - \phi \mathbf{A}^\mu \slashgrad_i \mathbf{A}_\mu \\
& + C \left[ | \slashgrad^2 \phi | + |\slashgrad(\mathbf{A} \phi)| + | \slashgrad \phi | + | \mathbf{A} \phi | + | \slashgrad \phi \slashgrad \mathbf{A} | + | \mathbf{A} \phi \slashgrad \mathbf{A} | \right] \\ 
& \la |\slashgrad \phi | + |\slashgrad^2 \phi | + |\slashgrad \phi \slashgrad \mathbf{A} | + | \phi \slashgrad^2 \mathbf{A} | + |\phi \slashgrad \mathbf{A} | + | \mathbf{A} \slashgrad \phi | \\
& + | \phi \mathbf{A} \slashgrad \mathbf{A} | + |\mathbf{A} \phi | .
\end{align*}
Putting these together, we have
\begin{align*} [\slashgrad_i, \, \mathfrak{S}]\phi &= [\slashgrad_i, \, \D^a \D_a + 1]\phi \\ 
&= [ \slashgrad_i, \, \D_0 \D_0 ] \phi - [\slashgrad_i, \, \slashed{\D}^\mu \slashed{\D}_\mu ] \phi \\ 
& \la | \phi \slashgrad \dot{A}_0 | + | \dot{\phi} \slashgrad A_0 | + | \phi A_0 \slashgrad A_0 | +  |\slashgrad \phi | + |\slashgrad^2 \phi | + |\slashgrad \phi \slashgrad \mathbf{A} | \\
& + | \phi \slashgrad^2 \mathbf{A} | + |\phi \slashgrad \mathbf{A} | + | \mathbf{A} \slashgrad \phi | + | \phi \mathbf{A} \slashgrad \mathbf{A} | + |\mathbf{A} \phi | .
\end{align*}
\end{proof}

\noindent Most of the terms in the above estimates we can control by the energy directly, with the exception of time derivatives of $A_0$. These terms we shall control using the elliptic equation for $A_0$ and the evolution equation for $\phi$.

\begin{proposition} \label{prop:A0dotH1bound} For any fixed $\tau \in \mathbb{R}$ there exists $\epsilon > 0$ such that if $\mathrm{S}_1[\phi] < \epsilon$ and $A_a$ satisfies the strong Coulomb gauge, then
\[ \| \dot{A}_0 \|^2_{H^1}(\tau) \la \mathrm{S}_2[\phi](\tau) ( 1 + \mathrm{S}_1[A](\tau))^2. \]
\end{proposition}

\begin{proof} First note that in the strong Coulomb gauge $\bar{A}_0(\tau) = 0$ for all $\tau$, and so $\dot{\bar{A}}_0(\tau) = 0$ for all $\tau$ as well. Thus $\| \dot{A}_0 \|_{L^2} \la \| \slashgrad \dot{A}_0 \|_{L^2}$, and we only need to estimate $\| \slashgrad \dot{A}_0 \|_{L^2}$. Differentiating \cref{ellipticequationCG} in $\tau$, we have
\[ - \slashed{\Delta} \dot{A}_0 + | \phi |^2 \dot{A}_0 = - \operatorname{Im}(\bar{\phi} \ddot{\phi}) - \bar{\phi} \dot{\phi} A_0 - \dot{\bar{\phi}} \phi A_0. \]
Multiplying through by $\dot{A}_0$ and integrating we have
\[ \| \slashgrad \dot{A}_0 \|^2_{L^2} + \| \phi \dot{A}_0 \|^2_{L^2} \leq \| \phi \dot{A}_0 \|_{L^2} \| \ddot{\phi} \|_{L^2} + 2 \| \phi \dot{A}_0 \|_{L^2} \| \dot{\phi} A_0 \|_{L^2} \]
which gives
\begin{equation} \label{dotA0estimate} \| \slashgrad \dot{A}_0 \|^2_{L^2} + \delta \| \phi \dot{A}_0 \|^2_{L^2} \la \| \ddot{\phi} \|^2_{L^2} + \| \dot{\phi} A_0 \|^2_{L^2}
\end{equation}
for some $0<\delta<1$. We thus need to estimate $\| \ddot{\phi} \|_{L^2}$, for which we shall use the field equation for $\phi$,
\[ \Box \phi + 2 i A_0 \dot{\phi} - 2 i \mathbf{A} \cdot \slashgrad \phi + ( 1 - A_0^2 + |\mathbf{A} |^2 + i \dot{A}_0 ) \phi = 0. \]
We estimate
\begin{equation} \label{ddotphibound} | \ddot{\phi} |^2 \la | \slashed{\Delta} \phi |^2 + | A_0 \dot{\phi} |^2 + | \mathbf{A} \slashgrad \phi |^2 + |\phi|^2 + | A_0^2 \phi |^2 + | \mathbf{A}^{2} \phi |^2 + | \dot{A}_0 \phi |^2 .
\end{equation}
With the exception of the term $| \dot{A}_0 \phi |^2$, the right-hand side of \cref{ddotphibound} will be easily controlled as we will see shortly. To deal with the problematic term we will use smallness of the data. Integrating \cref{ddotphibound} over the $3$-sphere we have
\begin{align*} \| \ddot{\phi} \|^2_{L^2} & \la \| \slashed{\Delta} \phi \|^2_{L^2} + \| A_0 \dot{\phi} \|^2_{L^2} + \| \mathbf{A} \slashgrad \phi \|^2_{L^2} + \| \phi \|^2_{L^2} + \| A_0^2 \phi \|^2_{L^2} + \| \mathbf{A}^{2} \phi \|^2_{L^2} + \| \dot{A}_0 \phi \|^2_{L^2} \\ 
& \la \| \phi \|^2_{H^2} + \| A_0 \|^2_{L^3} \| \dot{\phi}\|^2_{L^6} + \| \mathbf{A} \|^2_{L^3} \| \slashgrad \phi \|^2_{L^6} \\
& + \| A_0 \|_{L^6}^4 \| \phi \|^2_{L^6} + \| \mathbf{A} \|^4_{L^6} \| \phi \|^2_{L^6} + \| \dot{A}_0 \|^2_{L^3} \| \phi \|^2_{L^6} \\ 
& \la \| \phi \|^2_{H^2} + \| A_0 \|^2_{H^1} \| \dot{\phi} \|^2_{H^1} + \| \mathbf{A} \|^2_{H^1} \| \phi \|^2_{H^2} \\
& + \| A_0 \|^4_{H^1} \| \phi \|^2_{H^1} + \| \mathbf{A} \|^4_{H^1} \| \phi \|^2_{H^1} + \| \dot{A}_0 \|^2_{H^1} \| \phi \|^2_{H^1} \\ 
& \la \mathrm{S}_2[\phi] + \mathrm{S}_1[A] \mathrm{S}_2[\phi] + \mathrm{S}_1[\mathbf{A}] \mathrm{S}_2[\phi] + \mathrm{S}_1[A]^2 \mathrm{S}_1[\phi] + \mathrm{S}_1[\mathbf{A}]^2 \mathrm{S}_1[\phi] + \| \dot{A}_0 \|^2_{H^1} \mathrm{S}_1[\phi] \\ 
& \la \mathrm{S}_2[\phi] ( 1 + \mathrm{S}_1[A] )^2 + \| \dot{A}_0 \|^2_{H^1} \mathrm{S}_1[\phi].
\end{align*}
Putting this into \cref{dotA0estimate} gives
\[ \| \slashgrad \dot{A}_0 \|^2_{L^2} \la \mathrm{S}_2[\phi] ( 1 + \mathrm{S}_1[A] )^2 +  \| \dot{A}_0 \|^2_{H^1} \mathrm{S}_1[\phi],  \]
so provided $\mathrm{S}_1[\phi]$ is sufficiently small the Poincar\'e inequality gives
\[ \| \slashgrad \dot{A}_0 \|^2_{L^2} \la \mathrm{S}_2[\phi] ( 1+ \mathrm{S}_1[A])^2. \]
\end{proof}

\subsubsection{Estimate Algebra}

For ease of presentation we outline a schematic procedure to track how we bound the various terms arising in our $H^2$ estimates. The idea is simply to track the number of derivatives and their Sobolev exponents of the error terms and check that they do not exceed certain critical values. Let $f$ denote either $A$ or $\phi$, let $\partial$ denote either the $\mathbb{S}^3$-derivatives $\slashgrad$ or the $\tau$-derivative $\partial_\tau$, and let $\partial^2$ denote either $\slashgrad^2$ or $\partial_\tau \slashgrad$ (that is, not $\partial_\tau^2$). Then all the error terms that we encounter will in fact be of the form 
\[ \| | \partial^2 f|^m |\partial f|^k |f|^l \|_{L^1(\mathbb{S}^3)}, \]
where $m$, $k$, and $l$ are non-negative integers and in particular $m=0$, $1$, or $2$.

If $m=0$, we have
\[ \| |\partial f |^k |f|^l \|_{L^1} \leq \| f \|_{L^\infty}^l \| \partial f \|^k_{L^k}. \]
Now since $\mathbb{S}^3$ is compact, the Lebesgue spaces $L^p(\mathbb{S}^3)$ form a decreasing sequence in $p$, 
\[ L^{\infty}(\mathbb{S}^3) \hookrightarrow \dots \hookrightarrow L^p(\mathbb{S}^3) \hookrightarrow \dots \hookrightarrow L^q(\mathbb{S}^3) \hookrightarrow \dots \hookrightarrow L^1(\mathbb{S}^3), \]
$p > q$, where $\hookrightarrow$ denotes continuous inclusion. As $\mathbb{S}^3$ has dimension $3$, by Sobolev Embedding we also have
\[ H^1(\mathbb{S}^3) \hookrightarrow L^6(\mathbb{S}^3) \quad \text{ and } \quad H^2(\mathbb{S}^3) \hookrightarrow C^{0,\frac{1}{2}}(\mathbb{S}^3) \hookrightarrow L^\infty(\mathbb{S}^3), \]
so provided $k \leq 6$ we have
\[ \| |\partial f|^k | f |^l \|_{L^1} \la \| f \|^l_{2} \| f \|^k_{2} = \| f \|_{2}^{k+l}, \]
where
\[ \| f \|_2 \defeq \| f \|_{H^2(\mathbb{S}^3)} + \| \dot{f} \|_{H^1(\mathbb{S}^3)} \]
(notice that the norm $\| \cdot \|^2_2$ is the familiar Sobolev-type energy $\mathrm{S}_2$). 

If $m=1$, we perform the splitting
\begin{align*} \| | \partial^2 f | | \partial f |^k |f|^l \|_{L^1} & = \int |\partial^2 f | |\partial f|^k | f |^l \\
& \leq \int | \partial^2 f |^2 + \int |\partial f|^{2k} | f |^{2l} \leq \| f \|^2_{2} + \| | \partial f |^{2k} | f |^{2l} \|_{L^1}.	
\end{align*}

Now provided $ 2k \leq 6$, the second term in the above may be dealt with as in the case $m=0$, so we have
\[ \| | \partial^2 f | | \partial f |^k |f|^l \|_{L^1} \la \| f \|^2_{2} + \| f \|_{2}^{2(k+l)}. \]

Finally, when $m=2$ it will in fact turn out that $k$ is necessarily zero, so we will have
\[ \| | \partial^2 f |^2 |f|^l \|_{L^1} \leq \| f \|^l_{L^\infty} \| f \|^2_{2} \la \| f \|^{l+2}_{2}. \]
It will thus be sufficient to use the following prescription. For terms involving no $|\partial^2 f|$ (i.e. $m=0$) we shall check if $k \leq 6$, and if so, conclude that the term is bounded by $\|f\|_{2}^{k+l}$; for terms involving $| \partial^2 f |$ (i.e. $ m= 1$), we shall check if $k \leq 3$, and if so, conclude that the term is bounded by $\|f\|^2_{2} + \| f \|^{2(k+l)}_{2}$; finally, for terms with $m=2$ we shall check that $k=0$, and if so, conclude that these are bounded by $ \| f \|^{l+2}_{2}$. In the estimates that follow we will write down a term to be estimated,
\[ |\partial^2f|^m |\partial f|^k | f |^l, \]
and underneath note down its `signature' $(m,k,l)$, as in 
\[ \underset{\textcolor{blue}{(m,k,l)}}{|\partial^2f|^m |\partial f|^k | f |^l}. \]
If the criteria outlined above are met (that is, $k \leq 6$ for $m=0$, $k \leq 3$ for $m=1$, and $k=0$ for $m=2$), we shall tick the triplet,
\[ \underset{\textcolor{blue}{(m,k,l)\checkmark}}{|\partial^2f|^m |\partial f|^k | f |^l}. \]
Altogether this notation will thus mean that
\[ \| | \partial^2 f|^m |\partial f|^k |f|^l \|_{L^1(\mathbb{S}^3)} \la Q( \| f \|_2 ) \]
for some polynomial $Q$ with positive coefficients.

\subsubsection{\texorpdfstring{$H^2$}{H2} Error Terms} \label{sec:H2errorterms}

\allowdisplaybreaks

We now take $\alpha_a = \slashgrad_i A_a$ and $\psi = \slashgrad_i \phi$ in \cref{generalerrorterms} and estimate the second order error terms
\[ \mathsf{e}_2 \defeq \sum_i T^b \left( \nabla^a \mathbf{T}_{ab}[\slashgrad_i A] + \nabla^a \mathbf{T}_{ab}[\slashgrad_i \phi ] \right). \]
\Cref{generalerrorterms} gives
\begin{align*} \mathsf{e}_2 &= \sum_i - \mathfrak{M}(\slashgrad_i A)^\mu \left( \slashgrad_\mu \slashgrad_i A_0 - \slashgrad_i \dot{\mathbf{A}}_\mu \right) \\ &+ \sum_i \left( \frac{1}{2} \overline{\mathfrak{S}(\slashgrad_i \phi)} \D_0(\slashgrad_i \phi) + \frac{1}{2} \mathfrak{S}(\slashgrad_i \phi) \overline{\D_0(\slashgrad_i \phi)} - (\slashgrad_\mu A_0 - \dot{\mathbf{A}}_\mu ) \operatorname{Im}(\slashgrad_i \bar{\phi} \slashed{\D}^\mu \slashgrad_i \phi)  \right) \\ & \eqdef \mathsf{e}_2^1 + \mathsf{e}_2^2,
\end{align*}
and we consider $\mathsf{e}_2^1$ and $\mathsf{e}_2^2$ separately. We have
\small
\begin{align*} | \mathsf{e}_2^1 | & = \left| \sum_i - \mathfrak{M}(\slashgrad_i A)^\mu ( \slashgrad_\mu \slashgrad_i A_0 - \slashgrad_i \dot{\mathbf{A}}_\mu ) \right| \\ 
& \leq \sum_i \left| \left( \slashgrad_i \mathfrak{M}(A)^\mu - [\slashgrad_i, \, \mathfrak{M}](A)^\mu \right)\left( \slashgrad_\mu \slashgrad_i A_0 - \slashgrad_i \dot{\mathbf{A}}_\mu \right) \right| \\ 
& \la \left| \slashgrad (\bar{\phi} \slashed{\D} \phi) \right| \left[ | \slashgrad^2 A_0 | + |\slashgrad A_0 | + | \slashgrad \dot{\mathbf{A}} | \right] \\
& + \left[ | \slashgrad^2 \mathbf{A} | + | \slashgrad \mathbf{A} | + | \slashgrad \dot{A}_0 | \right] \left[ | \slashgrad^2 A_0 | + | \slashgrad A_0 | + | \slashgrad \dot{\mathbf{A}} | \right] \\ 
& \la \left[ | \slashgrad \phi |^2 + |\slashgrad \phi ||\phi| | \mathbf{A} | + |\slashgrad^2 \phi ||\phi| + |\slashgrad \mathbf{A} | |\phi|^2 + |\slashgrad \phi ||\phi | |\mathbf{A} | \right] \left[ | \slashgrad^2 A_0 | + | \slashgrad A_0 | + | \slashgrad \dot{\mathbf{A}} | \right] \\ 
& + \left[ | \slashgrad^2 \mathbf{A} | + | \slashgrad \mathbf{A} | + | \slashgrad \dot{A}_0 | \right] \left[ | \slashgrad^2 A_0 | + | \slashgrad A_0 | + | \slashgrad \dot{\mathbf{A}} | \right] \\ 
& \la \underset{\textcolor{blue}{(1,2,0) \checkmark }}{| \slashgrad^2 A_0| | \slashgrad \phi |^2} + \underset{\textcolor{blue}{(1,1,2)\checkmark }}{|\slashgrad^2 A_0 | | \slashgrad \phi | | \phi | | \mathbf{A} |} + \underset{\textcolor{blue}{(2,0,1)\checkmark }}{ |\slashgrad^2 A_0| | \slashgrad^2 \phi | | \phi | } + \underset{\textcolor{blue}{(1,1,2)\checkmark }}{ | \slashgrad^2 A_0 | | \slashgrad \mathbf{A} |  | \phi |^2} \\
& + \underset{\textcolor{blue}{(1,1,2)\checkmark }}{ |\slashgrad^2 A_0 | | \slashgrad \phi | | \phi | | \mathbf{A} | } + \underset{\textcolor{blue}{(0,3,0)\checkmark }}{| \slashgrad A_0 | | \slashgrad \phi |^2} + \underset{\textcolor{blue}{(0,2,2)\checkmark }}{| \slashgrad A_0 | | \slashgrad \phi | | \phi | |\mathbf{A} |} + \underset{\textcolor{blue}{(1,1,1)\checkmark }}{ |\slashgrad^2 \phi ||\slashgrad A_0| |\phi |} \\
& + \underset{\textcolor{blue}{(0,2,2)\checkmark }}{ |\slashgrad A_0 | |\slashgrad \mathbf{A} | |\phi |^2} + \underset{\textcolor{blue}{(0,2,2)\checkmark }}{|\slashgrad A_0 | |\slashgrad \phi| |\phi | |\mathbf{A} |}  + \underset{\textcolor{blue}{(1,2,0) \checkmark }}{|\slashgrad \dot{\mathbf{A}} | | \slashgrad \phi |^2} + \underset{\textcolor{blue}{(1,1,2) \checkmark }}{|\slashgrad \dot{\mathbf{A}} ||\slashgrad \phi | | \phi | | \mathbf{A} |} \\
& + \underset{\textcolor{blue}{(2,0,1)\checkmark }}{|\slashgrad \dot{\mathbf{A}} | | \slashgrad^2 \phi | | \phi|} + \underset{\textcolor{blue}{(1,1,2) \checkmark }}{ | \slashgrad \dot{\mathbf{A}}| | \slashgrad \mathbf{A}| | \phi |^2} + \underset{\textcolor{blue}{(1,1,2)\checkmark }}{|\slashgrad \dot{\mathbf{A}}| | \slashgrad \phi | |\phi | | \mathbf{A} |} +  \underset{\textcolor{blue}{(2,0,0)\checkmark }}{|\slashgrad^2 A_0 | |\slashgrad^2 \mathbf{A}|} + \underset{\textcolor{blue}{(1,1,0)\checkmark }}{|\slashgrad^2 A_0 | | \slashgrad \mathbf{A} |} \\
& + \underset{\textcolor{blue}{(2,0,0)\checkmark }}{|\slashgrad^2 A_0| | \slashgrad \dot{A}_0|} + \underset{\textcolor{blue}{(1,1,0)\checkmark }}{|\slashgrad^2 \mathbf{A} | | \slashgrad A_0 |} + \underset{\textcolor{blue}{(0,2,0)\checkmark }}{|\slashgrad A_0 | | \slashgrad \mathbf{A}|} + \underset{\textcolor{blue}{(1,1,0) \checkmark }}{ |\slashgrad \dot{A}_0 | | \slashgrad A_0 |} + \underset{\textcolor{blue}{(2,0,0) \checkmark }}{|\slashgrad^2 \mathbf{A} | |\slashgrad \dot{\mathbf{A}}|} \\
& + \underset{\textcolor{blue}{(1,1,0) \checkmark }}{|\slashgrad \dot{\mathbf{A}}| | \slashgrad \mathbf{A} |} + \underset{\textcolor{blue}{(2,0,0) \checkmark }}{|\slashgrad \dot{A}_0 | | \slashgrad \dot{\mathbf{A}} |}
\end{align*}
\normalsize
and
\small
\begin{align*} | \mathsf{e}_2^2 | & = \left| \sum_i \left( \frac{1}{2} \overline{\mathfrak{S}(\slashgrad_i \phi)} \D_0(\slashgrad_i \phi) + \frac{1}{2} \mathfrak{S}(\slashgrad_i \phi) \overline{\D_0(\slashgrad_i \phi)} - (\slashgrad_\mu A_0 - \dot{\mathbf{A}}_\mu) \operatorname{Im} ( \slashgrad_i \phi \slashed{\D}^\mu \slashgrad_i \phi) \right) \right| \\ 
&\leq \sum_i \left[ | \mathfrak{S}(\slashgrad_i \phi) | | \D_0 (\slashgrad_i \phi) | + | \slashgrad A_0 - \dot{\mathbf{A}} | |\slashgrad_i \phi | | \slashed{\D} \slashgrad_i \phi | \right] \\
& \la \sum_i | [ \slashgrad_i, \, \mathfrak{S}](\phi) | \left[ |\slashgrad_i \dot{\phi}| + | A_0 \slashgrad_i \phi | \right] + \left[ | \slashgrad A_0 | + | \dot{\mathbf{A}} | \right] | \slashgrad \phi | \left[ | \slashgrad^2 \phi | + | \mathbf{A} \slashgrad \phi | + | \slashgrad \phi | \right] \\ 
& \la \Big[ | \dot{\phi} \slashgrad A_0 | + | \phi \slashgrad \dot{A}_0 | + | \phi A_0 \slashgrad A_0 | + | \slashgrad^2 \phi | + | \phi \slashgrad^2 \mathbf{A} | + | \slashgrad A_0 | | \slashgrad \phi |^2 + | \dot{\mathbf{A}} | | \slashgrad \phi |^2 \\
&+ | \mathbf{A} \slashgrad \phi | + | \slashgrad \phi | + | \phi \slashgrad \mathbf{A} | + | \mathbf{A} \phi | + | \slashgrad \phi \slashgrad \mathbf{A} | + | \phi \mathbf{A} \slashgrad \mathbf{A} |  \Big] \left[ | \slashgrad \dot{\phi} | + | A_0 \slashgrad \phi | \right] \\
& + |\slashgrad^2 \phi | | \slashgrad A_0 | |\slashgrad \phi | + | \slashgrad A_0 | | \slashgrad \phi |^2 |\mathbf{A} | + |\slashgrad^2 \phi| |\slashgrad \phi| |\dot{\mathbf{A}} | + |\slashgrad \phi |^2 | \dot{\mathbf{A}} | | \mathbf{A} |  + |\slashgrad A_0 | |\slashgrad \phi |^2 \\
& + |\dot{\mathbf{A}} | | \slashgrad \phi |^2 \\
& \la \underset{\textcolor{blue}{(1,2,0)\checkmark }}{|\slashgrad \dot{\phi}| | \slashgrad A_0 | |\dot{\phi} |} + \underset{\textcolor{blue}{(2,0,1) \checkmark }}{|\slashgrad \dot{A}_0| | \slashgrad \dot{\phi} | | \phi |} + \underset{\textcolor{blue}{(1,1,2) \checkmark }}{|\slashgrad \dot{\phi}| | \slashgrad A_0 | | \phi | |A_0|} + \underset{\textcolor{blue}{(2,0,0) \checkmark }}{|\slashgrad^2 \phi | | \slashgrad \dot{\phi}|} + \underset{\textcolor{blue}{(2,0,1) \checkmark }}{|\slashgrad^2 \mathbf{A}| | \slashgrad \dot{\phi} | | \phi |} \\
& + \underset{\textcolor{blue}{(1,1,1) \checkmark }}{|\slashgrad \dot{\phi}| | \slashgrad \phi | |\mathbf{A} |} + \underset{\textcolor{blue}{(1,1,0) \checkmark }}{|\slashgrad \dot{\phi} | | \slashgrad \phi |} + \underset{\textcolor{blue}{(1,1,1) \checkmark }}{ | \slashgrad \dot{\phi}| | \slashgrad \mathbf{A}| | \phi | } + \underset{\textcolor{blue}{(1,0,2)\checkmark }}{|\slashgrad \dot{\phi}| | \mathbf{A}| | \phi |} + \underset{\textcolor{blue}{(1,2,0) \checkmark }}{|\slashgrad \dot{\phi}||\slashgrad \phi | |\slashgrad \mathbf{A} |} \\
& + \underset{\textcolor{blue}{(1,1,2)\checkmark }}{| \slashgrad \dot{\phi}| | \slashgrad \mathbf{A} | | \mathbf{A} | | \phi |} + \underset{\textcolor{blue}{(0,3,1) \checkmark }}{|\slashgrad \phi | |\dot{\phi} | | \slashgrad A_0 | |A_0 |} + \underset{\textcolor{blue}{(1,1,2)\checkmark }}{|\slashgrad \dot{A}_0| | \slashgrad \phi| |\phi | |A_0|} + \underset{\textcolor{blue}{(0,2,3) \checkmark }}{|\slashgrad A_0| |\slashgrad \phi | | \phi | |A_0|^2} \\
& + \underset{\textcolor{blue}{(1,1,1)\checkmark }}{|\slashgrad^2 \phi ||\slashgrad \phi| |A_0|} + \underset{\textcolor{blue}{(1,1,2) \checkmark }}{|\slashgrad^2 \mathbf{A}||\slashgrad \phi| |\phi| |A_0|} + \underset{\textcolor{blue}{(0,2,2) \checkmark }}{ |\slashgrad \phi|^2 | A_0 ||\mathbf{A} |} + \underset{\textcolor{blue}{(0,2,1) \checkmark }}{|\slashgrad \phi |^2 |A_0|} + \underset{\textcolor{blue}{(0,2,2)\checkmark }}{|\slashgrad \phi | | \slashgrad \mathbf{A} | | \phi | |A_0|} \\
& + \underset{\textcolor{blue}{(0,1,3) \checkmark }}{|\slashgrad \phi | |\phi | |A_0 | |\mathbf{A}|} + \underset{\textcolor{blue}{(0,3,1)\checkmark }}{|\slashgrad \phi|^2 | \slashgrad \mathbf{A} | |A_0|}  + \underset{\textcolor{blue}{(0,2,3)\checkmark }}{|\slashgrad \phi | |\slashgrad \mathbf{A} | |\phi | |A_0| |\mathbf{A} |} + \underset{\textcolor{blue}{(1,2,0)\checkmark }}{|\slashgrad^2 \phi| |\slashgrad \phi | |\slashgrad A_0|} \\
& + \underset{\textcolor{blue}{(0,3,1)\checkmark }}{|\slashgrad \phi|^2 |\slashgrad A_0| |\mathbf{A}|} + \underset{\textcolor{blue}{(1,2,0)\checkmark }}{|\slashgrad^2 \phi | |\slashgrad \phi| |\dot{\mathbf{A}}|} + \underset{\textcolor{blue}{(0,3,1)\checkmark }}{|\slashgrad \phi|^2 |\dot{\mathbf{A}}| |\mathbf{A} |} + \underset{\textcolor{blue}{(0,3,0) \checkmark }}{| \slashgrad \phi |^2 | \slashgrad A_0 |} + \underset{\textcolor{blue}{(0,3,0) \checkmark }}{ | \slashgrad \phi |^2 |\dot{\mathbf{A}}|}.
\end{align*}
\normalsize
Altogether this says that
\[ \|\mathsf{e}_2\|_{L^1(\mathbb{S}^3)} \la Q^{\text{IV}}\left(\|(\phi, \mathbf{A}, A_0) \|_2 \right) \]
for some polynomial $Q^{\text{IV}}$ with positive coefficients. An inspection of the triplets $(m,k,l)$ above shows that each error term contains at least one full power of $\| f \|_2^2$, so in fact
\begin{align*} \| \mathsf{e}_2 \|_{L^1} & \la \| (\phi, \mathbf{A}, A_0 ) \|_2^2 Q^{\text{III}}\left(\| (\phi, \mathbf{A}, A_0 \right) \|_2) \\
& \la \left( \mathrm{S}_2[\phi, A] + \| \dot{A}_0 \|^2_{H^1} \right) \left( Q^{\text{II}}\left(\mathrm{S}_2[\phi,A]^{1/2} \right) + Q^{\text{I}}\left(\| \dot{A}_0 \|_{H^1}\right) \right)
\end{align*}
for polynomials $Q^{\text{I,II,III}}$. Now by \Cref{prop:A0dotH1bound}, $\| \dot{A}_0 \|^2_{H^1} \la \mathrm{S}_2[\phi] (1 + \mathrm{S}_1[A])^2$. At this point we can either assume the first order estimates (\Cref{thm:H1estimates}), or bound $\| \dot{A}_0 \|^2_{H^1}$ by a polynomial in $\mathrm{S}_2[\phi,A]$ of degree higher than one; both methods are fine, but we will need to assume the first order estimates to close the second order ones anyway, so assuming $\mathrm{S}_1[\phi,A] \la 1$ we have $\| \dot{A}_0 \|^2_{H^1} \la \mathrm{S}_2[\phi,A]$. Hence for any fixed $\tau$
\begin{equation} \label{e2L1estimate} \| \mathsf{e}_2 \|_{L^1}(\tau) \la \mathrm{S}_2[\phi,A](\tau) P \left( \mathrm{S}_2[\phi, A](\tau)^{1/2} \right) \end{equation}
for some polynomial $P$.

\begin{theorem} \label{thm:H2estimates} Let $I$ be a fixed compact interval in $\mathbb{R}$ containing zero. There exists $\epsilon > 0$ such that if $\mathrm{S}_2[\phi, \mathbf{A}](0) \leq \epsilon$, then
\[ \mathrm{S}_2[\phi, A](\tau) \simeq \mathrm{S}_2[\phi, A](0) \]
for all $\tau \in I$.
\end{theorem}

\begin{proof} Integrating $\mathsf{e}_2$ over the region $\mathbb{S}^3 \times [0, \tau]$, $\tau >0$,
\begin{align} \begin{split} \label{2ndorderestimate} \int_{\mathbb{S}^3 \times [0,\tau]} \mathsf{e}_2 \dvol & = \int_0^\tau \int_{\mathbb{S}^3} \mathsf{e}_2(\sigma) \dvol_{\mathfrak{s}_3} \d \sigma \\
& = \sum_i \left( \mathcal{E}_\tau[\slashgrad_i \phi] + \mathcal{E}_\tau[\slashgrad_i A] \right) - \sum_i \left( \mathcal{E}_0[\slashgrad_i \phi] + \mathcal{E}_0[\slashgrad_i A] \right).
\end{split}	
\end{align}
From \Cref{thm:H1estimates} we know that $\mathrm{S}_1[\phi,A](\tau) \simeq \mathrm{S}_1[\phi,A](0)$, and also that $\mathcal{E}_\tau[A] \simeq \mathrm{S}_1[A](\tau)$ and $\mathcal{E}_\tau[\phi] \simeq \mathrm{S}_1[\phi](\tau)$ for all $\tau$. Furthermore, we have that $\mathrm{S}_1[A](\tau)$ is small, so by \cref{Sobolev2equivalencephi}
\[ \mathcal{E}_\tau[\phi] + \sum_i \mathcal{E}_\tau[\slashgrad_i \phi] \simeq \mathrm{S}_2[\phi](\tau). \]
By \cref{Sobolev2equivalenceA},
\[ \mathcal{E}_\tau[A] + \sum_i \mathcal{E}_\tau[\slashgrad_i A] \simeq \mathrm{S}_2[A](\tau), \]
so adding $\mathcal{E}_\tau[\phi,A] = \mathcal{E}_0[\phi,A]$ to both sides of \cref{2ndorderestimate} we have
\begin{align*} \mathcal{E}_\tau[\phi,A] + \sum_i \left( \mathcal{E}_\tau [\slashgrad_i \phi] + \mathcal{E}_\tau[\slashgrad_i A] \right) = \mathcal{E}_0[\phi,A] & + \sum_i \left( \mathcal{E}_0 [ \slashgrad_i \phi ] + \mathcal{E}_0 [ \slashgrad_i A] \right) \\
& + \int_0^\tau \int_{\mathbb{S}^3} \mathsf{e}_2(\sigma) \dvol_{\mathfrak{s}_3} \d \sigma,	
\end{align*}
or equivalently
\begin{equation} \label{Sobolev2error2} \mathrm{S}_2[\phi,A](\tau) \simeq \mathrm{S}_2[\phi, A](0) + \int_0^\tau \int_{\mathbb{S}^3} \mathsf{e}_2(\sigma) \dvol_{\mathfrak{s}_3} \d \sigma. \end{equation}
Now \cref{e2L1estimate} gives
\begin{align*} \mathrm{S}_2[\phi,A](\tau) & \la \mathrm{S}_2[\phi,A](0) + \int_0^\tau \| \mathsf{e}_2 \|_{L^1(\mathbb{S}^3)}(\sigma) \, \d \sigma \\ 
& \la \mathrm{S}_2[\phi,A](0) + \int_0^\tau \mathrm{S}_2[\phi,A](\sigma) P \left( \mathrm{S}_2[\phi, A](\sigma)^{1/2} \right) \, \d \sigma.
\end{align*}
By \Cref{lem:nonlinearGronwall},
\[ \mathrm{S}_2[\phi,A](\tau) \la \mathrm{S}_2[\phi,A](0) \]
for $\tau \in I$. \Cref{Sobolev2error2} similarly shows that $\mathrm{S}_2[\phi,A](0) \la \mathrm{S}_2[\phi, A](\tau)$, and so
\[ \mathrm{S}_2[\phi,A](\tau) \simeq \mathrm{S}_2[\phi,A](0). \]
for all $\tau \in I$. In particular, picking $I$ large enough to contain $[-\pi/2, \pi/2]$ shows
\[ \mathrm{S}_2[\phi,A](\scri^-) \simeq \mathrm{S}_2[\phi,A](\scri^+). \]
\end{proof}

\section{Higher Order Estimates} \label{sec:higherorderestimates}

From here it is straightforward to play the same game for higher order estimates. It is clear that if for a given $\tau$ and $m\geq 1$ the $(m+1)$-th Sobolev energy $\mathrm{S}_{m+1}[\phi,A](\tau)$ is small enough, then
\[ \sum_{k=0}^{m} \mathcal{E}_\tau[\slashgrad^k \phi] \simeq \mathrm{S}_{m+1}[\phi](\tau) \quad \text{and} \quad \sum_{k=0}^{m} \mathcal{E}_\tau[\slashgrad^k A] \simeq \mathrm{S}_{m+1}[A](\tau), \]
where as before $\mathcal{E}_\tau[\slashgrad^k \phi] = \sum_{i_1, \dots, i_k \in \{ 1,2,3 \}} \mathcal{E}_\tau [ \slashgrad_{i_1} \dots \slashgrad_{i_k} \phi]$, and similarly for $A_a$. We suppress sums over the basis vector fields $\{ X_i \}$ from now. It is clear that to prove that
\begin{equation} \label{higherestimates} \mathrm{S}_{m+1}[\phi,A](\tau) \simeq \mathrm{S}_{m+1}[\phi,A](0) \end{equation}
it is enough to prove the estimate 
\begin{equation} \label{generalerrorbound} \| \mathsf{e}_{m+1} \|_{L^1}(\tau) \la \mathrm{S}_{m+1}[\phi,A](\tau) P\left( \mathrm{S}_{m+1}[\phi,A](\tau)^{1/2} \right) \end{equation}
for a polynomial $P$, since then the proof of \cref{higherestimates} goes through exactly as in the proof of \Cref{thm:H2estimates}. Now because 
\[ H^{m+1}(\mathbb{S}^3) \hookrightarrow C^{m-1}(\mathbb{S}^3), \]
in our $(m+1)$-th order estimates we need only track derivatives of order $m$ and higher, since all the others will be $L^{\infty}$-controlled by $\mathrm{S}_{m+1}$. More precisely, since the $\mathrm{S}_{m+1}$ energies control the $L^\infty$ norms of $\slashgrad^{m-1} \phi$, $\slashgrad^{m-1} A$, $\slashgrad^{m-2} \dot{\phi}$ and $\slashgrad^{m-2} \dot{\mathbf{A}}$, we will only track terms of higher order than these (and also $\dot{A}_0$, which we will deal with separately as before). As before, one can write down the bounds for the commutators of $\slashgrad$ with the field equation operators $\mathfrak{M}$ and $\mathfrak{S}$, acting this time on a general $1$-form $\alpha$ and a general scalar field $\psi$,
\[ \left| [ \slashgrad, \, \mathfrak{M} ](\alpha) \right|_{\mathbb{S}^3} \la | \slashgrad^2 \boldsymbol{\alpha} | + |\slashgrad \dot{\alpha}_0 | + \text{l.o.t.s}, \]
and
\[ \left| [ \slashgrad, \, \mathfrak{S} ](\psi) \right| \la | \psi \slashgrad \dot{A}_0| + | \dot{\psi} \slashgrad A_0 | + | \slashgrad^2 \psi | + | \psi \slashgrad^2 \mathbf{A} | + \text{l.o.t.s}, \]
where the lower order terms are terms that are of order one or zero in derivatives of $\alpha$, $A$, or $\psi$. Now estimate the $(m+1)$-th error term:

\small
\begin{align*} \mathsf{e}_{m+1} & \defeq T^b \left( \nabla^a \mathbf{T}_{ab} [ \slashgrad^m A] + \nabla^a \mathbf{T}_{ab} [ \slashgrad^m \phi ] \right) \\
& = T^b \big( \mathfrak{M}(\slashgrad^m A)^a ( \nabla_a (\slashgrad^m A)_b - \nabla_b (\slashgrad^m A)_a ) + \operatorname{Re} \left( \overline{\mathfrak{S}(\slashgrad^m \phi )} \D_b (\slashgrad^m \phi) \right) \\
& + (\nabla_a A_b - \nabla_b A_a ) \operatorname{Im}\left(\slashgrad^m \bar{\phi} \D^a \slashgrad^m \phi \right) \big) \\
& \leq \left| \mathfrak{M}(\slashgrad^m A)^\mu (\slashgrad_\mu (\slashgrad^m A_0) - \slashgrad^m \dot{\mathbf{A}}_\mu ) \right| + \left| \operatorname{Re} \left( \mathfrak{S}(\slashgrad^m \phi) \D_0 (\slashgrad^m \phi ) \right) \right| \\
& + \left| (\slashgrad_\mu A_0 - \dot{\mathbf{A}}_\mu ) \operatorname{Im} \left( \slashgrad^m \bar{\phi} \slashed{\D}^\mu \slashgrad^m \phi \right) \right| \\
& \la \left| \mathfrak{M}(\slashgrad^m A) \right|_{\mathbb{S}^3} \left[ | \slashgrad^{m+1} A_0 | + | \slashgrad^m \dot{\mathbf{A}} | \right] + \left| \mathfrak{S} (\slashgrad^m \phi) \right| \left[ | \slashgrad^m \dot{\phi} | + | A_0 | | \slashgrad^m \phi | \right] \\ 
& + \left[ | \slashgrad A_0 | + | \dot{\mathbf{A}} | \right] \left[ |\slashgrad^m \phi | | \slashgrad^{m+1} \phi | + | \slashgrad^m \phi | | \mathbf{A} | |\slashgrad^m \phi | \right] +  \text{l.o.t.s} \\
& \la \left[ | \slashgrad^{m+1} A_0 | + | \slashgrad^m \dot{\mathbf{A}} | \right] \Big[ \left| \slashgrad^m \mathfrak{M}(A) \right|_{\mathbb{S}^3} + \left| [ \slashgrad^m, \, \mathfrak{M} ](A) \right|_{\mathbb{S}^3} \Big] + \Big[ \left| \slashgrad^m \mathfrak{S}(\phi) \right| \\
& + \left| [\slashgrad^m, \, \mathfrak{S}](\phi) \right| \Big] | \slashgrad^m \dot{\phi} | + | \slashgrad^m \phi | | \slashgrad^{m+1} \phi | \left[ |\slashgrad A_0 | + | \dot{\mathbf{A}} | \right] + \text{l.o.t.s} \\
& \la \left[ | \slashgrad^{m+1} A_0 | + | \slashgrad^m \dot{\mathbf{A}} | \right] \left[ | \slashgrad^m ( \phi \slashed{\D} \phi ) | + | \slashgrad^{m-1} [ \slashgrad, \, \mathfrak{M} ](A) | + | [ \slashgrad, \, \mathfrak{M} ]| ( \slashgrad^{m-1} A ) | \right] \\
& + |\slashgrad^m \dot{\phi} | \left[ | \slashgrad^{m-1} [ \slashgrad, \, \mathfrak{S}](\phi) | + | [ \slashgrad, \, \mathfrak{S} |( \slashgrad^{m-1} \phi) | \right] + | \slashgrad^m \phi | | \slashgrad^{m+1} \phi | \left[ | \slashgrad A_0 | + | \dot{\mathbf{A}} | \right] \\
& + \text{l.o.t.s} \\
& \la \left[ | \slashgrad^{m+1} A_0 | + | \slashgrad^m \dot{\mathbf{A}} | \right] \textcolor{red}{\bigg[} \left| \slashgrad^m ( \phi \slashgrad \phi + \mathbf{A} \phi^2 ) \right| + \left| \slashgrad^{m-1} ( \slashgrad^2 \mathbf{A} + \slashgrad \dot{A}_0 + \text{l.o.t.s} ) \right| \\ 
& + |\slashgrad^{m+1} \mathbf{A} | + |\slashgrad^m \dot{A}_0 | + \text{l.o.t.s} \textcolor{red}{\bigg]} \\
& + | \slashgrad^m \dot{\phi} | \textcolor{blue}{\bigg[} \left| \slashgrad^{m-1} ( \phi \slashgrad \dot{A}_0 + \dot{\phi} \slashgrad A_0 + \slashgrad^2 \phi + \phi \slashgrad^2 \mathbf{A} + \text{l.o.t.s} ) \right| \\
& + | \slashgrad^{m-1} \phi | | \slashgrad \dot{A}_0 | + | \slashgrad^{m-1} \dot{\phi} | |\slashgrad A_0 | + | \slashgrad^{m+1} \phi | + | \slashgrad^{m-1} \phi | | \slashgrad^2 \mathbf{A} | \textcolor{blue}{\bigg]} \\
& + | \slashgrad^m \phi | | \slashgrad^{m+1} \phi | \left[ | \slashgrad A_0 | + | \dot{\mathbf{A}} | \right] + \text{l.o.t.s} \\
& \la \left[ | \slashgrad^{m+1} A_0 | + |\slashgrad^m \dot{\mathbf{A}} | \right] \textcolor{red}{\Bigg[} \sum_{k=0}^m | \slashgrad^{m-k} \phi | | \slashgrad^{k+1} \phi | + | \slashgrad^m ( \mathbf{A} \phi^2 ) | + | \slashgrad^{m+1} \mathbf{A} | + | \slashgrad^m \dot{A}_0 | \textcolor{red}{\Bigg]} \\
& + | \slashgrad^m \dot{\phi} | \textcolor{blue}{\Bigg[} \sum_{k=0}^{m-1} | \slashgrad^{m-1-k} \phi | | \slashgrad^{k+1} \dot{A}_0 | + \sum_{k=0}^{m-1} | \slashgrad^{m-1-k} \dot{\phi} | | \slashgrad^{k+1} A_0 | + |\slashgrad^{m+1} \phi | \\
& + \sum_{k=0}^{m-1} | \slashgrad^{m-1-k} \phi | |\slashgrad^{k+2} \mathbf{A} | \textcolor{blue}{\Bigg]} + | \slashgrad^m \phi | | \slashgrad^{m+1} \phi | \left[ | \slashgrad A_0 | + | \dot{\mathbf{A}} | \right] + \text{l.o.t.s} \\
& \la_{\mathrm{S}_{m+1}^{1/2}} \left[ | \slashgrad^{m+1} A_0 | + | \slashgrad^m \dot{\mathbf{A}} | \right] \textcolor{red}{\bigg[} |\slashgrad^m \phi | | \slashgrad \phi | + |\phi | |\slashgrad^{m+1} \phi | + |\phi|^2 |\slashgrad^m \mathbf{A} | + |\phi | | \mathbf{A} | |\slashgrad^m \phi | \\
& + |\slashgrad^{m+1} \mathbf{A} | + |\slashgrad^m \dot{A}_0 | \textcolor{red}{\bigg]} + | \slashgrad^m \dot{\phi} | \textcolor{blue}{\bigg[} | \slashgrad^m \dot{A}_0 | + | \slashgrad^{m-1} \dot{\phi} | |\slashgrad A_0 | + |\dot{\phi} | |\slashgrad^m A_0 | \\
& + | \slashgrad^{m+1} \phi | + | \slashgrad^{m+1} \mathbf{A} | \textcolor{blue}{\bigg]} + | \slashgrad^m \phi | | \slashgrad^{m+1} \phi | + \text{l.o.t.s} \\
& \la_{\mathrm{S}_{m+1}^{1/2}} \left[ | \slashgrad^{m+1} A_0 | + | \slashgrad^m \dot{\mathbf{A}} | \right] \textcolor{red}{\bigg[} | \slashgrad^m \phi | + | \slashgrad^{m+1} \phi | + | \slashgrad^m \mathbf{A} | + | \slashgrad^{m+1} \mathbf{A} | + |\slashgrad^m \dot{A}_0 | \textcolor{red}{\bigg]} \\
& + |\slashgrad^m \dot{\phi} | \textcolor{blue}{\bigg[} | \slashgrad^m \dot{A}_0 | + | \slashgrad^{m-1} \dot{\phi} | + |\slashgrad^m A_0 | + | \slashgrad^{m+1} \phi | + | \slashgrad^{m+1} \mathbf{A} | \textcolor{blue}{\bigg]} + |\slashgrad^m \phi | | \slashgrad^{m+1} \phi | \\
& + \text{l.o.t.s} \\
& \la_{\mathrm{S}_{m+1}^{1/2}} \left[ | \slashgrad^{m+1} A_0 | + |\slashgrad^m \dot{\mathbf{A}} | \right] \textcolor{red}{\bigg[} | \slashgrad^{m+1} \phi | + | \slashgrad^{m+1} \mathbf{A} | + |\slashgrad^m \dot{A}_0 | \textcolor{red}{\bigg]} + | \slashgrad^m \dot{\phi} | \textcolor{blue}{\bigg[} | \slashgrad^m \dot{A}_0 | \\
& + | \slashgrad^{m-1} \dot{\phi} | + |\slashgrad^m A_0 | + |\slashgrad^{m+1} \phi | + |\slashgrad^{m+1} \mathbf{A} | \textcolor{blue}{\bigg]} + | \slashgrad^m \phi | | \slashgrad^{m+1} \phi | + \text{l.o.t.s} \\
& \la_{\mathrm{S}_{m+1}^{1/2}} | \slashgrad^{m+1} A_0 | | \slashgrad^{m+1} \phi | + | \slashgrad^{m+1} A_0 | | \slashgrad^{m+1} \mathbf{A} | + | \slashgrad^{m+1} A_0 | | \slashgrad^m \dot{A}_0 | \\
& + | \slashgrad^m \dot{\mathbf{A}} | | \slashgrad^{m+1} \phi | + | \slashgrad^m \dot{\mathbf{A}} | | \slashgrad^{m+1} \mathbf{A} | + | \slashgrad^m \dot{\mathbf{A}} | | \slashgrad^m \dot{A}_0 | + | \slashgrad^m \dot{\phi} | | \slashgrad^m \dot{A}_0 | \\
& + | \slashgrad^m \dot{\phi} | |\slashgrad^{m-1} \dot{\phi} | + | \slashgrad^m \dot{\phi} | | \slashgrad^m A_0 | + |\slashgrad^m \dot{\phi} | | \slashgrad^{m+1} \phi | + | \slashgrad^m \dot{\phi} | | \slashgrad^{m+1} \mathbf{A} | \\
& + | \slashgrad^m \phi | | \slashgrad^{m+1} \phi | + \text{l.o.t.s},
\end{align*}

\normalsize

\noindent where by $\la_{\mathrm{S}^{1/2}_{m+1}}$ we mean ``bounded up to a polynomial in $\mathrm{S}_{m+1}^{1/2}$". Note also that, like in the estimate of \Cref{sec:H2errorterms} where the triplets $(m,k,l)$ sum to at least two, the lower order terms in the above are at least quadratic in the fields so that one can control them by a full power of $\mathrm{S}_{m+1}$. Furthermore, inspecting the leading order terms in the above one sees that, with the exception of $\slashgrad^m \dot{A}_0$, they are all easily controlled by $\mathrm{S}_{m+1}$:
\begin{align*} \| \mathsf{e}_{m+1} \|_{L^1} & \la_{\mathrm{S}_{m+1}^{1/2}} \mathrm{S}_{m+1} + \| \slashgrad^{m+1} A_0 \slashgrad^m \dot{A}_0 \|_{L^1} + \| \slashgrad^m \dot{\mathbf{A}} \slashgrad^m \dot{A}_0 \|_{L^1} + \| \slashgrad^m \dot{\phi} \slashgrad^m \dot{A}_0 \|_{L^1} \\
& \la_{\mathrm{S}_{m+1}^{1/2}} \mathrm{S}_{m+1} + \| \dot{A}_0 \|^2_{H^m}.
\end{align*}
As in \Cref{prop:A0dotH1bound}, standard elliptic and wave equation estimates inductively show that for small $\mathrm{S}_{m}$,
\begin{equation} \label{generalA0dotbound} \| \dot{A}_0 \|^2_{H^m} \la_{\mathrm{S}_{m}^{1/2}} \mathrm{S}_{m+1},  \end{equation}
so altogether we have
\[ \| \mathsf{e}_{m+1} \|_{L^1} \la \mathrm{S}_{m+1} P( \mathrm{S}_{m+1}^{1/2} ) \]
for some polynomial $P$.

\section{Proof of \texorpdfstring{\Cref{thm:estimates}}{estimates}} \label{sec:proofofH2estimates}

The $m=1$ case is trivial, while for $m=2$ we have already proved the estimates $\mathrm{S}_m[\phi,A](\tau) \simeq \mathrm{S}_m[\phi, A](0)$ and $\| \dot{A}_0 \|^2_{H^{m-1}}(\tau) \la \mathrm{S}_m[\phi,A](\tau) $ for small initial data. We proceed by induction. Suppose the estimates
\[ \mathrm{S}_m[\phi, A](\tau) \simeq \mathrm{S}_m[\phi,A](0) \quad \text{and} \quad \| \dot{A}_0 \|^2_{H^{m-1}}(\tau) \la \mathrm{S}_m[\phi, A](\tau) \]
hold for some $m \in \mathbb{N}$ provided $\mathrm{S}_m[\phi,A](0)$ is small enough. The second of these is immediate for $m+1$ by \cref{generalA0dotbound}, which then implies \cref{generalerrorbound}. Arguing as in the proof of \Cref{thm:H2estimates} and applying \Cref{lem:nonlinearGronwall} then gives \cref{higherestimates}. \hspace{\fill} $\blacksquare$

\section{Proof of \texorpdfstring{\Cref{thm:scatteringtheory}}{scatteringtheory}} \label{sec:proofofhigherorderestimates}

We restrict ourselves to the case of $\scri^+$, the case of $\scri^-$ being analogous. Pick admissible initial data $u_0$ on $\Sigma$ such that $\mathrm{S}_m[\phi,\mathbf{A}](\Sigma)$ is small enough. Then $\mathrm{S}_m[\phi,A](\Sigma) < \epsilon_0$ for some small $\epsilon_0 >0$, and by \Cref{cor:wellposedness} there exists a solution $(\phi,A_a)$ in $E_m = \bigcap_{k=0}^m C^k_b(I;H^{m-k})$ to the system \cref{phiAequationsCGprojected} unique up to trivial gauge transformations such that $I$ contains $[-\pi/2,\pi/2]$. Since the solution $(\phi,A_a)$ is at least $C^1$ in $\tau$ for $m \geq 2$, $ u = (\phi, \mathbf{A}, \dot{\phi}, \dot{\mathbf{A}}, A_0)$ has a well-defined restriction to $\scri^+$. This defines the forward wave operator
\begin{align*} \mathfrak{T}^+_m : \mathrm{S}_{m,\epsilon_0}^0 &\longrightarrow \mathrm{S}_m^+, \\
u_0 & \longmapsto u^+ = (\phi, \mathbf{A}, \dot{\phi}, \dot{\mathbf{A}}, A_0)|_{\scri^+}.
\end{align*}
By \Cref{thm:estimates}, whenever $\epsilon_0$ is small enough we have the estimate
\begin{equation} \label{scriplusdataestimate} \mathrm{S}_m[\phi, A](\scri^+) \leq C \mathrm{S}_m[\phi,A](\Sigma) \leq C \epsilon_0 \eqdef \epsilon_1, \end{equation}
so the operator $\mathfrak{T}^+_m$ is bounded. The data $u^+$ on $\scri^+$ has size at most $\epsilon_1 = C \epsilon_0$, so reducing $\epsilon_0$ if necessary, we can evolve $u^+$ backwards in time to find data $\tilde{u}_0$ on $\Sigma$. But by uniqueness $u_0 = \tilde{u}_0$. Thus the map $\mathfrak{T}^+_m$ is injective for $\epsilon_0$ small enough.

Now restrict the co-domain of $\mathfrak{T}^+_m$ to its image:
\[ \mathfrak{T}^+_m : \mathrm{S}^0_{m, \epsilon_0} \longrightarrow \mathfrak{T}^+_m(\mathrm{S}^0_{m, \epsilon_0}) \eqdef \mathscr{D}^+_{m, \epsilon_1}. \]
By definition, $\mathfrak{T}^+_m$ is now surjective and so bijective, and from the estimate \eqref{scriplusdataestimate} it is clear that $\mathscr{D}^+_{m, \epsilon_1} \subset \mathrm{S}^+_{m,\epsilon_1}$. The operator $\mathfrak{T}^+_m$ is thus invertible and satisfies the bounds
\[ \| \mathfrak{T}^+_m u_0 \|^2_{\mathrm{S}_m} \la \|u_0 \|^2_{\mathrm{S}_m} \quad \text{and} \quad \| (\mathfrak{T}^+_m)^{-1} u^+ \|^2_{\mathrm{S}_m} \la \| u^+ \|^2_{\mathrm{S}_m} \]
for $u_0 \in \mathrm{S}^0_{m, \epsilon_0}$, $u^+ \in \mathscr{D}^+_{m, \epsilon_1}$. Furthermore, the set $\mathscr{D}^+_{m,\epsilon_1}$ contains a small ball around the origin in $\mathrm{S}^+_m$. Indeed, if $v^+ \in \mathrm{S}_m^+$ has small enough norm, say $\| v^+ \|^2_{\mathrm{S}_m} < \delta \ll \epsilon_0$, then $\| (\mathfrak{T}^+_m)^{-1} v^+ \|^2_{\mathrm{S}_m} \leq C \| v^+ \|^2_{\mathrm{S}_m} < C \delta < \epsilon_0$, and so $(\mathfrak{T}^+_m)^{-1} v^+ \in \mathrm{S}^0_{m, \epsilon_0}$.

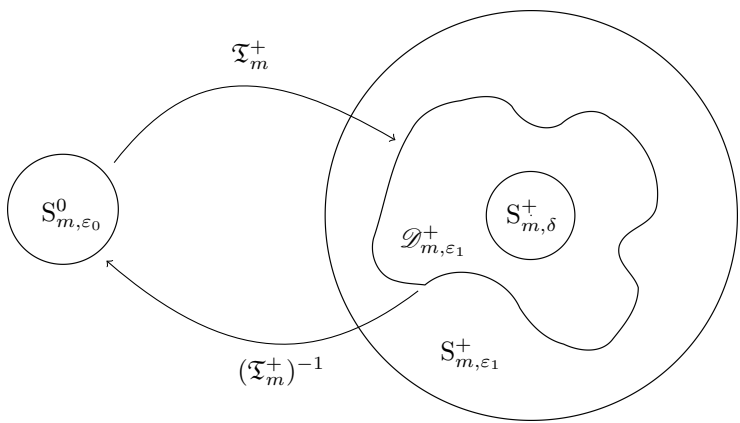
\begin{figure}[h]
\centering
	\begin{tikzpicture}
	\centering
	\node[inner sep=0pt] (wavemaps) at (3.4,0)
    	{\includegraphics[width=.6\textwidth]{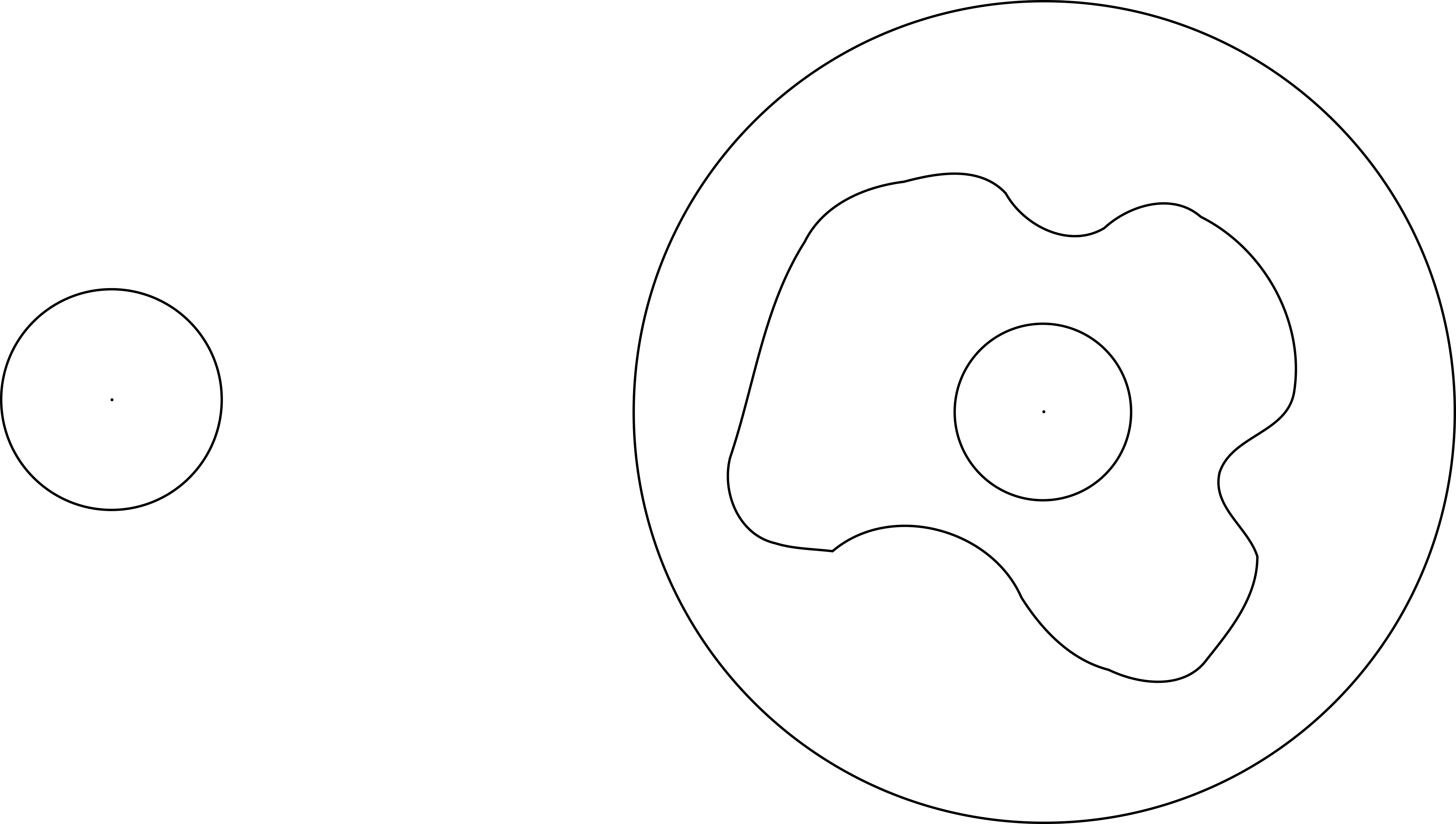}};

	\node[label={[shift={(-0.55,-0.45)}]$\mathrm{S}^0_{m,\epsilon_0}$}] {};
	\node[label={[shift={(1.8,1.7)}]$\mathfrak{T}^+_m$}] {};
	\node[label={[shift={(2.2,-2.5)}]$(\mathfrak{T}^+_m)^{-1}$}] {};
	\node[label={[shift={(4.7,-2.3)}]$\mathrm{S}^+_{m,\epsilon_1}$}] {};
	\node[label={[shift={(5.5,-0.5)}]$\mathrm{S}^+_{m,\delta}$}] {};
	\node[label={[shift={(4.2,-0.8)}]$\mathscr{D}^+_{m,\epsilon_1}$}] {};	
	\draw[->] (0,0.7) .. controls (1, 2) and (2,2) .. (3.7,1);
	\draw[<-] (-0.1,-0.6) .. controls (1.5, -2) and (2.7,-2) .. (4,-1);	
	
	\end{tikzpicture}
	\caption{The image of a small ball under the forward wave operator $\mathfrak{T}^+_m$.} \label{fig:wavemaps}	
\end{figure}

Constructing the scattering operator is now simply a matter of composing the inverse backward wave operator and the forward wave operator. We define
\begin{align*} & \mathscr{S}_m : \mathscr{D}^-_{m, \epsilon_1} \longrightarrow \mathscr{D}^+_{m, \epsilon_1}, \\
& \mathscr{S}_m \defeq \mathfrak{T}^+_m \circ ( \mathfrak{T}^-_m)^{-1}.
\end{align*}
Then $\mathscr{S}_m$ is invertible with inverse $\mathscr{S}_m^{-1} = \mathfrak{T}^-_m \circ ( \mathfrak{T}^+_m)^{-1}$, and the estimates
\[ \| u^+ \|^2_{\mathrm{S}_m} \simeq \| u^- \|^2_{\mathrm{S}_m} \]
for $u^\pm \in \mathscr{D}^\pm_{m,\epsilon_1}$ follow from the estimates for $\mathfrak{T}^\pm_m$. \hspace{\fill} $\blacksquare$

\begin{remark} 
It is not immediately clear what the set $\mathscr{D}^+_{m,\epsilon_1}$ looks like, for two reasons. Firstly, the sets $\mathrm{S}^{\pm,0}_m$ are not vector spaces since admissible initial data is not additive. Secondly, the fact that $\mathfrak{T}^+_m$ is a nonlinear operator precludes any straightforward application of the open mapping theorem, so it is not even obvious that $\mathscr{D}^+_{m,\epsilon_1}$ is open and connected. Nonetheless, by symmetry it is clear that the set of past asymptotic data $\mathscr{D}^-_{m, \epsilon_1}$ and the set of future asymptotic data $\mathscr{D}^+_{m, \epsilon_1}$ are of the same `size' in the sense that they are contained in balls of the same radius in $\mathrm{S}^-_m$ and $\mathrm{S}^+_m$ respectively.
\end{remark}

\begin{remark} The lack of vector space structure on the domains of definition of the operators $\mathfrak{T}_m^{\pm}$ and $\mathscr{S}_m$ makes it difficult to discuss their regularity beyond boundedness. This lack of vector space structure stems, most importantly, from the constraint equations in the system \eqref{phiAequationsCGprojected}. It is fairly easy to see that any extension of e.g. $\mathscr{S}_m$ off the constraint surface that preserves boundedness will automatically be continuous at the zero solution, but continuity at more general solutions will require a more careful analysis of \eqref{phiAequationsCGprojected} linearized around said solution, as well as a choice of extension. Differentiability will pose further complications.
\end{remark}

\section{Proof of \texorpdfstring{\Cref{thm:decayrates}}{decayrates}} \label{sec:proofofdecayrates}

Suppose $\mathrm{S}_m[\tilde{\phi}, \tilde{\mathbf{A}}](\eta = 0)$ is small. We derive the asymptotics for $\scri^+$, the ones for $\scri^-$ being analogous. By \cref{scalingofenergies}, $\mathrm{S}_m[\phi,\mathbf{A}](\tau=0)$ is small too, and $A_0$ estimates imply that the full $\mathrm{S}_m[\phi, A](\tau = 0)$ is small. Then according to our estimates and Sobolev embeddings, $\phi$, $\mathbf{A}$ and $A_0$ are continuous on all of $\widehat{\mathrm{dS}}_4$ with a $C^{m-2}$ trace on $\scri^+$.

Let $m=2$. Then $\phi = \Omega^{-1} \tilde{\phi}$ has a continuous limit on $\scri^+$, so 
\[ | \tilde{\phi} | \la \Omega \la \frac{1}{\cosh(H \eta)} \la \e^{-H \eta} \]
as $\eta \to + \infty$. The timelike component of $A_a$ is $A_0 = \partial_\tau^a A_a = H^{-1} \cosh(H \eta) \partial_\eta^a \tilde{A}_a = H^{-1} \cosh(H \eta) \tilde{A}_\eta$
and has a continuous limit on $\scri^+$, so similarly
\[ | \tilde{A}_\eta | \la \e^{-H \eta} \]
as $\eta \to + \infty $. Finally the $\mathbb{S}^3$ components of $A$ are
\[ |\mathbf{A}|^2_{\mathfrak{s}_3} = - \mathfrak{e}^{ab} \mathbf{A}_a \mathbf{A}_b = - \Omega^{-2} \tilde{g}^{ab} \tilde{\mathbf{A}}_a \tilde{\mathbf{A}}_b = \Omega^{-2} \frac{H^2}{\cosh^2(H \eta)} \mathfrak{s}_3^{\mu \nu} \tilde{\mathbf{A}}_\mu \tilde{\mathbf{A}}_\nu = | \tilde{\mathbf{A}} |^2_{\mathfrak{s}_3}, \]
so $ | \tilde{\mathbf{A}} |_{\mathfrak{s}_3} \la 1$.

Next work in the static coordinates \eqref{physicaldSmetricstatic}. These coordinates are only appropriate in region I of \Cref{fig:desitterpenrose} since they become singular on the horizons $r=1/H$, and $\partial_t$ is spacelike in regions II and IV and past-pointing in region III. Following the flow of the vector field $\partial_t$ in region I, one is forced to the top right corner of \Cref{fig:desitterpenrosestatic} as $t \to +\infty$. A preferred point on $\scri^+$ has therefore been singled out for an observer following the flow of $\partial_t$; this point is the timelike infinity for observers living in region $\mathrm{I}$ of \Cref{fig:desitterpenrose}.

In these coordinates the conformal factor $\Omega$ is given by
\[ \Omega = \frac{H}{\cosh(Ht)} \frac{1}{\sqrt{F_t(r)}}, \]
where $F_t(r) = 1 - \tanh^2(Ht) H^2 r^2$. Keeping $r$ fixed, for the scalar field we then have
\[ | \tilde{\phi} | \la \Omega \la_r \e^{-Ht} \]
as $t \to + \infty$. For the Maxwell potential we find the relations
\begin{align*} & \tilde{A}_t = H^2  \sech^2(Ht)	 F_t(r)^{-1} \left( -r F(r)^{1/2} \sinh(Ht) A_\zeta + H^{-1} F(r)^{1/2}  \cosh(Ht) A_\tau \right), \\
& \tilde{A}_r = H^2 \sech^2(Ht) F_t(r)^{-1} \left( H^{-1} F(r)^{-1/2} \cosh(Ht) A_\zeta - r F(r)^{-1/2} \sinh(Ht) A_\tau \right).
\end{align*}
Since $A_\zeta$ and $A_\tau$ have continuous limits as $t \to +\infty$ for $r$ fixed, we have
\[ | \tilde{A}_t | \la_r \e^{-Ht} \quad \text{and} \quad | \tilde{A}_r | \la_r \e^{-Ht}. \]
Expanding the $3$-sphere norm $|\tilde{\mathbf{A}}|^2_{\mathfrak{s}_3}$,
\[ | \tilde{\mathbf{A}} |^2_{\mathfrak{s}_3} = \tilde{A}_\zeta^2 + \frac{1}{\sin^2 \zeta} | \tilde{A} |^2_{\mathfrak{s}_2} \la 1, \]
we see that $| \tilde{A} |_{\mathfrak{s}_2} \la \sin \zeta$, where one computes $\sin \zeta = \sech(Ht) Hr F_t(r)^{-1/2}$. Thus
\[  \frac{1}{r} |\tilde{A}|_{\mathfrak{s}_2} \la_r \e^{-Ht} \]
as $t \to + \infty$.

Now suppose $m=3$. This in particular means that
\[ | \slashgrad \phi |^2 = ( \partial_\zeta \phi )^2 + \frac{1}{\sin^2 \zeta} | \nabla^{\mathfrak{s}_2} \phi |^2 \]
has a continuous limit on $\scri^+$, and so $\partial_\zeta \phi$ and $ (\sin \zeta)^{-1} | \nabla^{\mathfrak{s}_2} \phi |$ do too. Since $\phi$ scales conformally as $\phi = \Omega^{-1} \tilde{\phi}$, one computes
\begin{align} \label{delzetaphi} \begin{split} \partial_\zeta \phi & = H^{-1} \cosh(Ht) F_t(r)^{1/2} \Big( r F(r)^{-1/2} \sinh(Ht) \partial_t \tilde{\phi} \\
& + H^{-1} F(r)^{1/2} \cosh(Ht) \partial_r \tilde{\phi} \Big)
\end{split}
\end{align}
and
\begin{align} \label{deltauphi}  \begin{split} \partial_\tau \phi + (\partial_\tau \Omega) \Omega^{-1} \phi &= H^{-1} \cosh(Ht) F_t(r)^{1/2} \Big( H^{-1} F(r)^{-1/2} \cosh(Ht) \partial_t \tilde{\phi} \\
& + r F(r)^{1/2} \sinh(Ht) \partial_r \tilde{\phi} \Big).
\end{split}
\end{align}
Since $\Omega \partial_\zeta \phi$ and $\Omega \partial_\tau \phi + (\partial_\tau \Omega ) \phi$ have continuous limits on $\scri^+$, one sees that 
\[ |\partial_t \tilde{\phi}| \la_r \e^{-Ht} \quad \text{and} \quad | \partial_r \tilde{\phi} | \la_r \e^{-Ht} \]
as $t \to + \infty$. For the $\mathbb{S}^2$ derivatives, the fact that $(\sin \zeta)^{-1} |\nabla^{\mathfrak{s}_2} \phi| = \Omega^{-1} (\sin \zeta)^{-1} | \nabla^{\mathfrak{s}_2} \tilde{\phi} |$ has a continuous limit on $\scri^+$ implies that
\begin{equation} \label{delS2phi} \left| \frac{1}{r} \nabla^{\mathfrak{s}_2} \tilde{\phi} \right| \la_r \e^{-2Ht} \end{equation}
as $t \to + \infty$. Let us study the $\e^{-Ht}$ component of $\tilde{\phi}$,
\[ \tilde{\varphi} \defeq \e^{Ht} \tilde{\phi}. \]
Rewriting \cref{delzetaphi} and \cref{deltauphi} in terms of $\tilde{\varphi}$, one has
\[ \mathcal{O}\left(\e^{-Ht}\right) = r F(r)^{-1/2} \sinh(Ht) \e^{-Ht} ( \partial_t \tilde{\varphi} - H \tilde{\varphi}) + H^{-1} F(r)^{1/2} \cosh(Ht)\e^{-Ht}\partial_r \tilde{\varphi}  \]
and
\begin{align*} \mathcal{O}\left(\e^{-Ht}\right) - F(r)^{1/2} \sinh(Ht) \e^{-Ht} \tilde{\varphi} &= H^{-1} F(r)^{-1/2} \cosh(Ht) \e^{-Ht} ( \partial_t \tilde{\varphi}  - H \tilde{\varphi}  ) \\
& + r F(r)^{1/2} \sinh(Ht) \e^{-Ht} \partial_r \tilde{\varphi},
\end{align*}
which taking the limit $t \to +\infty$ become
\begin{align*}  0 & \approx Hr \partial_t \tilde{\varphi} - H^2 r \tilde{\varphi} + F \partial_r \tilde{\varphi}, \\
- H F \tilde{\varphi} & \approx \partial_t \tilde{\varphi} - H \tilde{\varphi} + Hr F \partial_r \tilde{\varphi},
\end{align*}
where $\approx$ denotes equality at $t = + \infty$. Solving these algebraically shows that $\partial_t \tilde{\varphi} \approx 0$ and
\[ H^2 r \tilde{\varphi} \approx F(r) \partial_r \tilde{\varphi}. \]
The bound \eqref{delS2phi} shows that at $t= +\infty$ the function $\tilde{\varphi}$ is independent of the $\mathbb{S}^2$ coordinates, so the above equation is an ODE in $r$, with solution 
\[ \tilde{\varphi}(r) \approx \frac{1}{\sqrt{F(r)}} \tilde{\varphi}(0). \]
We conclude that there exists a constant $c$ such that
\[ \tilde{\phi} \sim c F(r)^{-1/2}  \e^{-Ht} + \mathcal{O} \left( \e^{-2Ht} \right) \quad \text{as } t \to +\infty. \]
as $t \to +\infty$. One can check by hand that $\tilde{\Phi}_1(t,r) = F(r)^{-1/2} \e^{-Ht}$ is a spherically symmetric solution of the uncharged ($A_a = 0$) conformally invariant free wave equation
\[ (\tilde{\Box} + \frac{1}{6} \tilde{R})\tilde{\Phi}_1 = F(r)^{-1}\partial_t^2 \tilde{\Phi}_1 - \frac{1}{r^2} \partial_r ( r^2 F(r) \partial_r \tilde{\Phi}_1) - \frac{1}{r^2} \Delta^{\mathfrak{s}_2}\tilde{\Phi}_1 = 0. \] 
\hspace{\fill} $\blacksquare$

\section*{Acknowledgments} The author wishes to thank Lionel Mason, Qian Wang, Jan Sbierski, \mbox{Jean-Philippe} Nicolas and Mihalis Dafermos for useful guidance and discussions. This work was supported by the Engineering and Physical Sciences Research Council grant [EP/L05811/1].

%\pagebreak

\appendix

%\raggedbottom

\section{The Geometry of \texorpdfstring{$\mathbb{S}^3$}{S3}}

\subsection{Projection onto Divergence Free \texorpdfstring{$1$}{1}-Forms} \label{sec:divergencefree1forms}

Let $*$ denote the Hodge star operator on $\mathbb{S}^3$ and $\d$ the exterior derivative on $\mathbb{S}^3$. Let $\mathbf{A}$ be a $1$-form and $f$ a function on $\mathbb{S}^3$, and write
\begin{align*} & \operatorname{curl} \mathbf{A} \defeq * \, \d \mathbf{A}, \\
& \operatorname{div} \mathbf{A} \defeq *\,  \d * \mathbf{A}, \\
& \operatorname{grad} f \defeq \d f.
\end{align*}
It is easy to check that the definitions of $\operatorname{div} \mathbf{A}$ and $\operatorname{grad} f$ coincide with the notions of $\operatorname{div}$ and $\operatorname{grad}$ in terms of the Levi--Civita connection $\slashgrad$ on $\mathbb{S}^3$, that is $\operatorname{div} \mathbf{A} = \slashgrad_\mu \mathbf{A}^\mu$ and $ (\operatorname{grad} f )_\mu = \slashgrad_\mu f$. With these definitions
\[ \operatorname{curl} ( \operatorname{curl} \mathbf{A} ) - \operatorname{grad} ( \operatorname{div} \mathbf{A} ) = * \, \d * \d \mathbf{A} - \d * \d * \mathbf{A} = \delta \d \mathbf{A} + \d \delta \mathbf{A} \eqdef - \slashed{\Delta}^{(1)} \mathbf{A}, \]
where $\delta \defeq (-1)^{3k} * \d *$ is the \emph{codifferential} acting on $k$-forms on $\mathbb{S}^3$ and the operator
\begin{align*} -\slashed{\Delta}^{(1)} : \Gamma (\Lambda^1 \mathbb{S}^3) & \longrightarrow \Gamma(\Lambda^1 \mathbb{S}^3), \\ - \slashed{\Delta}^{(1)} \defeq \delta & \d + \d \delta,
\end{align*}
is the \emph{Hodge Laplacian} on $1$-forms on $\mathbb{S}^3$. The operator $\slashed{\Delta}^{(1)}$ can be extended to act on arbitrary $k$-forms in the obvious way (giving a number of operators $\slashed{\Delta}^{(k)}$, if one wishes to distinguish between their domains), but it is important to note that if $k\neq 0$ the action of $\slashed{\Delta}^{(k)}$ differs from the connection Laplacian $\slashed{\Delta} \defeq \slashgrad^\mu \slashgrad_\mu$ in a way that depends on the degree of the forms it is acting on. The difference is given by the Weitzenb\"ock formula, which in the case of $1$-forms is known as Bochner's theorem (see \S2.2.2 of \cite{Rosenberg1997}).

\begin{theorem}[Bochner's Theorem] Let $(\mathcalboondox{N}, g)$ be a Riemannian manifold with a positive definite metric $g$ and let $\nabla$ be the Levi--Civita connection of $g$. Considered as operators $\Gamma(\Lambda^1 \mathcalboondox{N}) \to \Gamma(\Lambda^1 \mathcalboondox{N})$, the Hodge Laplacian $\Delta^{(1)}$ and the connection Laplacian $\Delta = \nabla_\mu \nabla^\mu$ are related by
\[ - \Delta^{(1)} = \Delta + R, \] 
where $R$ is the scalar curvature of $g$.
\end{theorem}

If $\mathcalboondox{N} = \mathbb{S}^3$, we thus have
\[ -\slashed{\Delta}^{(1)} = \slashed{\Delta} - 6. \]
Now suppose that $\mathbf{A} \in \Gamma(\Lambda^1 \mathbb{S}^3)$ satisfies the Coulomb gauge $\operatorname{div} \mathbf{A} = 0$. Then
\[ \operatorname{curl} (\operatorname{curl} \mathbf{A}) = - \slashed{\Delta}^{(1)} \mathbf{A} = (\slashed{\Delta} - 6) \mathbf{A}. \]
Given any $\mathbf{A} \in \Gamma(\Lambda^1 \mathbb{S}^3)$, the elliptic equation
\begin{equation} \label{projectiondivergencefree} (\slashed{\Delta} - 6) \mathbf{B} = \operatorname{curl}(\operatorname{curl} \mathbf{A}) \end{equation}
on $\mathbb{S}^3$ has a unique solution $\mathbf{B} \in \Gamma(\Lambda^1 \mathbb{S}^3)$. This allows us to define the projection onto divergence free $1$-forms $\mathcal{P} : \Gamma(\Lambda^1  \mathbb{S}^3) \to \Gamma(\Lambda^1 \mathbb{S}^3)$,
\[ \mathcal{P}\mathbf{A} \defeq (\slashed{\Delta} - 6)^{-1} \operatorname{curl}(\operatorname{curl} \mathbf{A}). \] 
By construction, for any $\mathbf{A}$ satisfying $\operatorname{div} \mathbf{A} = 0$, $\mathcal{P} \mathbf{A} = \mathbf{A}$, and $\operatorname{div} \mathcal{P} \mathbf{B} = 0$ for any $\mathbf{B}$. This second identity follows by commuting the $\operatorname{div}$ operator into the equation \eqref{projectiondivergencefree}. Furthermore, for any function $f$
\[ \mathcal{P} ( \operatorname{grad} f ) = ( \slashed{\Delta} - 6)^{-1} ( \operatorname{curl} ( \operatorname{curl} (\operatorname{grad} f ))) = (\slashed{\Delta} - 6)^{-1} (\mathbf{0} ) = \mathbf{0}. \]

\subsection{Christoffel symbols and Curvature Tensors}

\begin{proposition} Since $\mathbb{S}^3$ is maximally symmetric, the Ricci $R_{\mu \nu} = R_{\mu \nu}(\mathfrak{s}_3)$ and Riemann $R_{\mu \nu \rho \sigma} = R_{\mu \nu \rho \sigma}(\mathfrak{s}_3)$ tensors of $\mathbb{S}^3$ are expressible entirely in terms of the metric $\mathfrak{s}_3$, 
\[ R_{\mu \nu} = - 2 (\mathfrak{s}_3)_{\mu \nu}, \]
and
\[ R_{\mu \nu \rho \sigma} = (\mathfrak{s}_3)_{\rho \nu} (\mathfrak{s}_3)_{\mu \sigma} - (\mathfrak{s}_3)_{\nu \sigma} (\mathfrak{s}_3)_{\mu \rho}. \]
The scalar curvature of $\mathbb{S}^3$ is $R(\mathfrak{s}_3) = -6$.
\end{proposition}

\begin{proposition} \label{prop:Christoffels} In the coordinates $(\tau, \zeta, \theta, \phi)$ the non-zero Christoffel symbols of the metric $\mathfrak{e}$ are
\begin{align*} & \Gamma^{\zeta}_{\theta \theta} = - \sin \zeta \cos \zeta, && \Gamma^{\zeta}_{\phi \phi} = - \sin^2 \theta \sin \zeta \cos \zeta, \\
& \Gamma^{\theta}_{\zeta \theta} = \cot \zeta = \Gamma^{\theta}_{\theta \zeta}, && \Gamma^{\theta}_{\phi \phi} = - \sin \theta \cos \theta, \\
& \Gamma^{\phi}_{\zeta \phi} = \cot \zeta = \Gamma^{\phi}_{\phi \zeta}, && \Gamma^{\phi}_{\theta \phi} = \cot \theta = \Gamma^{\phi}_{\phi \theta}.
\end{align*}
\end{proposition}

\begin{proposition} In the coordinates $(\tau, \zeta, \theta, \phi)$ the non-zero components of the Ricci tensor of $\mathfrak{e}$ are 
\[ R_{\zeta \zeta} = -2, \quad R_{\theta \theta} = -2 \sin^2 \zeta, \quad R_{\phi \phi} = -2 \sin^2 \zeta \sin^2 \theta. \]
In fact,
\[ R_{ab} = -2 \left( 0 \oplus \mathfrak{s}_3 \right), \]
and the scalar curvature is thus
\[ R = 6. \]
\end{proposition}

\begin{proposition} \label{prop:Riemanntensorcomponents} In the coordinates $(\tau, \zeta, \theta, \phi)$ the non-zero components of the Riemann tensor of $\mathfrak{e}$ are
\begin{align*} & R^{\zeta}_{\phantom{\zeta} \theta \zeta \theta} = - \sin^2 \zeta = - R^{\zeta}_{\phantom{\zeta} \theta \theta \zeta}, && R^\zeta_{\phantom{\zeta} \phi \zeta \phi} = - \sin^2 \zeta \sin^2 \theta = - R^{\zeta}_{\phantom{\zeta} \phi \phi \zeta}, \\
& R^{\theta}_{\phantom{\theta} \zeta \zeta \theta} = 1 = - R^{\theta}_{\phantom{\theta} \zeta \theta \zeta}, && R^{\theta}_{\phantom{\theta} \phi \theta \phi} = - \sin^2 \zeta \sin^2 \theta = - R^{\theta}_{\phantom{\theta} \phi \phi \theta}, \\
& R^{\phi}_{\phantom{\phi} \zeta \zeta \phi} = 1 = - R^{\phi}_{\phantom{\phi} \zeta \phi \zeta}, && R^{\phi}_{\phantom{\phi} \theta \theta \phi} = \sin^2 \zeta = - R^{\phi}_{\phantom{\phi} \theta \phi \theta}.
\end{align*}
\end{proposition}

\section{The Sobolev Embedding Theorem}

The following definitions and theorems can be found in chapter 2 of \cite{Aubin1998}.

\begin{definition} Let $(M,g)$ be a smooth Riemannian manifold of dimension $n$. For a real function $\phi$ belonging to $C^k(M)$, $k\geq0$ an integer, we define
\[ |\nabla^k \phi |^2 \defeq \left( \nabla^{a_1} \nabla^{a_2} \dots \nabla^{a_k} \phi \right) \left( \nabla_{a_1} \nabla_{a_2} \dots \nabla_{a_k} \phi \right), \]
and denote by $\mathfrak{C}^{k,p}$ the vector space of $C^{\infty}$ functions $\phi$ such that $|\nabla^l \phi | \in L^p(M)$ for all $ 0 \leq l \leq k$ and $p \geq 1$ a real number.
\end{definition}

\begin{definition} The Sobolev space $W^{k,p}(M)$ is the completion of $\mathfrak{C}^{k,p}$ with respect to the norm
\[ \| \phi \|_{W^{k,p}} \defeq \sum_{l=0}^k \| \nabla^l \phi \|_p. \]
The space $W^{k,p}(M)$ does not depend on the Riemannian metric $g$ (Theorem 2.20, \cite{Aubin1998}).
\end{definition}

\begin{theorem} Let $M$ be a smooth compact Riemannian manifold of dimension $n$, let $k$, $l$ be integers with $k > l \geq 0$, and let $p$, $q$ be real numbers with $1 \leq q < p$ satisfying
\[ \frac{1}{p} = \frac{1}{q} - \frac{(k-l)}{n}. \]
Then 
\[ W^{k,q}(M) \subset W^{l,p}(M), \]
and the identity operator is continuous (the embedding is compact).

Moreover, if 
\[ \frac{(k-r - \alpha)}{n} \geq \frac{1}{q}, \]
then
\[ W^{k,q}(M) \subset C^{r,\alpha}(M), \]
and the identity operator is continuous (the embedding is compact). Here $r \geq 0$ is an integer, $\alpha$ is a real number satisfying $0 < \alpha \leq 1$, $C^{r, \alpha}$ is the space of $C^r$ functions the $r$th derivatives of which belong to $C^\alpha$, $C^r$ is the space of functions $\phi$ of finite $\| \phi \|_{C^r} \defeq \max_{0 \leq l \leq r} \sup |\nabla^l u |$ norm, and $C^\alpha$ is the space of functions of finite $\| \phi \|_{C^\alpha} \defeq \sup | \phi | + \sup_{P \neq Q} \{ |\phi(P) - \phi(Q)| d(P,Q)^{-\alpha} \}$ norm.
\end{theorem}

\begin{theorem} \label{thm:sobolevembedding} Let $M$ be a smooth compact Riemannian manifold of dimension $n$ and let the real numbers $p$, $q$ satisfy
\[ \frac{1}{p} = \frac{1}{q} - \frac{1}{n} >0. \]
Then for every $\epsilon > 0$ there exists a constant $A_q(\epsilon)$ such that every $\phi \in W^{1,q}(M)$ satisfies
\[ \| \phi \|_p \leq \left( \mathsf{K}(n,q) + \epsilon \right) \| \nabla \phi \|_q + A_q(\epsilon) \| \phi \|_q, \]
where $\mathsf{K}(n,q)$ is the smallest constant having this property and is given by
\[ \mathsf{K}(n,q) = \left( \frac{q-1}{n-q} \right) \left( \frac{n-q}{n(q-1)} \right)^{\frac{1}{q}} \left( \frac{\Gamma(n+1)}{\Gamma(n/q) \Gamma(n+1 - n/q) \omega_{n-1}} \right)^{\frac{1}{n}} \]
for $1 < q< n$ and
\[ \mathsf{K}(n,1) = \frac{1}{n} \left( \frac{n}{\omega_{n-1}} \right)^{\frac{1}{n}}. \]
\end{theorem}

\nocite{LesHouchesLectureNotes}

\nocite{Friedrich1991}

%\pagebreak

%\bibliographystyle{siam}
%\bibliography{bibliography}

\printbibliography

\end{document}